\documentclass[reqno,11pt]{amsart}

\usepackage[english]{babel}
\usepackage[utf8]{inputenc}
\usepackage[T1]{fontenc}

\usepackage[margin=1in, letterpaper]{geometry}
\usepackage{amssymb,amsmath,amsfonts,amsthm}
\usepackage{bm,bbm,mathrsfs}
\usepackage[dvipsnames]{xcolor}
\usepackage{url}
\usepackage{enumitem}
\usepackage{mathtools}
\usepackage{placeins}
\usepackage{extarrows}
\usepackage{subfig}
\usepackage[font=small,labelfont=bf]{caption}
\usepackage{floatrow}
\usepackage{verbatim}
\usepackage{wrapfig}
\usepackage{tikz}
\usetikzlibrary{arrows,arrows.meta,cd,decorations.pathmorphing,backgrounds,positioning,fit,matrix}
\usepackage{xcolor,colortbl,graphics,graphicx}
\usepackage[colorlinks=true, allcolors=Blue]{hyperref}
\usepackage[noabbrev,capitalize]{cleveref}
\crefname{equation}{}{}

\newtheorem{theorem}{Theorem}[section]
\newtheorem{proposition}[theorem]{Proposition}
\newtheorem{lemma}[theorem]{Lemma}

\newtheorem{conjecture}[theorem]{Conjecture}
\newtheorem*{question*}{Question}
\Crefname{question}{Question}{Questions}

\newtheorem{knowntheorem}{Theorem}

\theoremstyle{definition}

\newtheorem{notation}[theorem]{Notation}
\newtheorem*{notation*}{Notation}

\newtheorem*{example*}{Example}

\theoremstyle{remark}
\newtheorem*{remark}{Remark}

\numberwithin{equation}{section}
\allowdisplaybreaks

\newcommand{\psum}{\ \sideset{}{^*}\sum}
\newcommand{\Z}{\mathbf{Z}}
\newcommand{\R}{\mathbf{R}}

\newcommand{\C}{\mathbf{C}}
\newcommand{\N}{\mathbf{N}}

\renewcommand{\H}{\mathbf{H}}

\newcommand{\msE}{\mathscr{E}}
\newcommand{\mE}{\mathcal{E}}

\newcommand{\mD}{\mathcal{D}}

\newcommand{\mS}{\mathcal{S}}
\newcommand{\mK}{\mathcal{K}}

\newcommand{\mB}{\mathcal{B}}
\newcommand{\mX}{\mathcal{X}}

\newcommand{\mI}{\mathcal{I}}

\newcommand{\ma}{\mathfrak{a}}

\newcommand{\e}{\textnormal{e}}
\newcommand{\one}{\mathbbm{1}}
\renewcommand{\Re}{\textnormal{Re}}

\newcommand{\PSL}{\textnormal{PSL}}

\newcommand{\Id}{\textnormal{Id}}

\newcommand{\eps}{\varepsilon}
\newcommand{\uminus}{\mathbin{\ooalign{$\cup$\cr%
   \hfil\raise0.42ex\hbox{$\scriptscriptstyle\minus$}\hfil\cr}}}

\newcommand{\sgn}{\textnormal{sgn}}
\renewcommand{\bar}{\overline}
\renewcommand{\hat}{\widehat}

\newcommand{\cond}{\textnormal{cond}}
\renewcommand{\mod}[1]{\ \textnormal{mod } #1}
\renewcommand{\pmod}[1]{\ (\textnormal{mod } #1)}

\newcommand{\ch}{\textnormal{ch}}
\newcommand{\cremph}[1]{\emph{\cref{#1}}}

\newcommand{\rad}{\textnormal{rad}}

\newenvironment{psmall}
  {\left(\begin{smallmatrix}}
  {\end{smallmatrix}\right)}
\newenvironment{bsmall}
  {\left[\begin{smallmatrix}}
  {\end{smallmatrix}\right)}

\newcommand{\twoline}[2]{$\substack{\text{\normalsize \vphantom{y}#1} \\ \text{\normalsize \vphantom{y}#2}}$}

\newcommand{\twolinemath}[2]{\let\scriptstyle\textstyle \substack{{#1} \\ {#2}}}
\definecolor{firebrick}{rgb}{0.7, 0.13, 0.13}
\definecolor{myBlue}{rgb}{0, 0, 0.6}

\begin{document}

\title{Smooth numbers in arithmetic progressions to large moduli}
\author[Alexandru Pascadi]{Alexandru Pascadi}
\email{alexpascadi@gmail.com}
\address{Mathematical Institute, Radcliffe Observatory Quarter, Woodstock Road, Oxford OX2 6GG, England}

\begin{abstract}
    We show that smooth numbers are equidistributed in arithmetic progressions to moduli of size $x^{66/107-o(1)}$. This overcomes a longstanding barrier of $x^{3/5-o(1)}$ present in previous works of Bombieri--Friedlander--Iwaniec, Fouvry--Tenenbaum, Drappeau, and Maynard.

    We build on Drappeau's variation of Linnik's dispersion method and on exponential sum manipulations of Maynard, ultimately relying on optimized Deshouillers--Iwaniec type estimates for sums of Kloosterman sums.
\end{abstract}

\maketitle

{
\setlength{\parskip}{0em}
\setcounter{tocdepth}{1}
\tableofcontents
}

\section{Introduction} \label{sec:intro} 
The Bombieri--Vinogradov theorem \cite{bombieri1965large,vinogradov1965density} famously states that for any $x > 2$, $A > 0$, and $B$ sufficiently large in terms of $A$, one has
\begin{equation} \label{eq:bv}
    \sum_{q \le x^{1/2}/(\log x)^B} \max_{(a, q) = 1} 
    \left\vert \pi(x; q, a) - \frac{\pi(x)}{\varphi(q)} \right\vert 
    \ll_A 
    \frac{x}{(\log x)^A},
\end{equation}
where $\pi(x)$ denotes the number of primes up to $x$, and $\pi(x; q, a)$ denotes the number of these primes which are congruent to $a$ modulo $q$. Informally, \cref{eq:bv} states that the primes are well-distributed in arithmetic progressions when averaging over moduli almost as large as $x^{1/2}$. Without the sum over $q$, the Siegel--Walfisz theorem only controls the summand uniformly in the (much smaller) range $q \le (\log x)^A$, and the Generalized Riemann Hypothesis would improve this to $q \le x^{1/2}/(\log x)^B$. Thus \cref{eq:bv} provides an unconditional substitute for GRH when some averaging over $q$ is available; this is very often the case in sieve theory, where results like \cref{eq:bv} have led to multiple major breakthroughs (including, e.g., the existence of infinitely many bounded gaps between primes \cite{zhang2014bounded,polymath2014new,maynard2015small,polymath2014variants}).

We say that the primes have \emph{exponent of distribution} $\alpha < 1$ iff the analogue of \cref{eq:bv} holds true when summing over all moduli $q \le x^\alpha$. The Elliott--Halberstam conjecture \cite{elliott1968conjecture} asserts that $\alpha = 1 - \eps$ works for any $\eps > 0$ (the implied constant depending on $\eps$), but it remains open whether \cref{eq:bv} holds for any $\alpha > 1/2$. Quite remarkably, it \emph{is} possible to go beyond this square-root barrier if one slightly weakens the left-hand side of \cref{eq:bv}, by fixing the residue $a$, assuming various factorization properties of the moduli $q$, and/or replacing the absolute values with suitable weights. On this front, we mention the pioneering work of Fouvry \cite{fouvry1984autour,fouvry1987autour,fouvry1985probleme,fouvry1982repartition} and Fouvry--Iwaniec \cite{fouvry1980theorem,fouvry1983primes}, a series of three papers by Bombieri--Friedlander--Iwaniec \cite{bombieri1986primes,bombieri1987primes2,bombieri1989primes3}, the main estimate in Zhang's work on bounded gaps \cite{zhang2014bounded}, and three recent papers of Maynard \cite{maynard2025primes,maynard2025primes2,maynard2025primes3}; in particular, in \cite{maynard2025primes2}, Maynard achieved exponents of distribution as large as $3/5-\eps$ assuming well-factorable weights.

In this paper, we are concerned with the case of \emph{$y$-smooth} (or \emph{$y$-friable}) numbers rather than primes; the objects of study here are the sets
\[
    S(x, y) := \{n \in \Z_+ : n \le x, \text{ and all prime factors of $n$ are $\le y$}\},
\]
defined by two parameters $x, y \ge 2$, where $y$ will grow like $x^{o(1)}$. To state our main result, we denote
\[
\begin{aligned}
    \Psi(x, y) &:= \# S(x, y), \\
    \Psi_q(x, y) &:= \# \{n \in S(x, y) : (n, q) = 1\}, \\
    \Psi(x, y; a, q) &:= \# \{n \in S(x, y) : n \equiv a \pmod{q}\}.
\end{aligned}
\]
\begin{theorem}[Smooth numbers in APs with moduli beyond $x^{3/5}$] \label{thm:distribution-simplified}
Let $a \in \Z \setminus \{0\}$ and $A, \eps > 0$. Then there exists $C = C(A, \eps) > 0$ such that in the range $x > 2$, $(\log x)^C \le y \le x^{1/C}$, one has
\[
    \sum_{\substack{q \le x^{66/107-\eps} \\ (q, a) = 1}} 
    \left\vert \Psi(x, y; a, q) - \frac{\Psi_q(x, y)}{\varphi(q)}
    \right\vert 
    \ll_{a, A, \eps} 
    \frac{\Psi(x, y)}{(\log x)^A}.
\]
\end{theorem}

A similar result with an exponent of $1/2 - \eps$ and uniformity in $a$ (as in \cref{eq:bv}) is due to Granville \cite[Theorem 2]{granville1993integers}; see also \cite{wolke1973mittlereI,wolke1973mittlereII,fouvry1991entiers,granville1993integers2}.
Virtually all results of this type that go beyond the $x^{1/2}$-barrier rely on equidistribution estimates for convolutions of sequences, but unless a long smooth sequence is involved, the exponents have been limited to $3/5$ or less. Bombieri--Friedlander--Iwaniec proved a triple convolution estimate handling moduli up to $x^{3/5}$ for sequences of convenient lengths \cite[Theorem 4]{bombieri1986primes}, and Fouvry--Tenenbaum used this result (along with the flexible factorization properties of smooth numbers) to prove an analogue of \cref{thm:distribution-simplified} for $q \le x^{3/5-\eps}$, and with a right-hand side of $x/(\log x)^A$ \cite[Th\'eor\`eme 2]{fouvry1996repartition}. Motivated by an application to the Titchmarsh divisor problem for smooth numbers, Drappeau improved the bound to $\Psi(x, y)/(\log x)^A$ in the same range $q \le x^{3/5-\eps}$ \cite[Th\'eor\`eme 1]{drappeau2015theoremes}. Unfortunately, the BFI estimates and subsequent arguments seem limited to this range of moduli.

Maynard recently introduced a different arrangement of exponential sums \cite[Chapter 18]{maynard2025primes}, which would in principle allow for a triple convolution estimate with moduli up to $x^{5/8}$, if the Selberg eigenvalue conjecture for Maass forms \cite{selberg1965estimation,sarnak1995selberg,iwaniec1985character,iwaniec1985density,iwaniec1990small,luo1995selberg} held true; but his unconditional estimates were still limited below $x^{3/5}$ \cite[Proposition 8.3]{maynard2025primes}. We introduce a further variation of Maynard's argument which eliminates certain coefficient dependencies, allowing one to use more efficient estimates for sums of Kloosterman sums in some ranges. More precisely, we rely on an optimized estimate of Deshouillers--Iwaniec type (see \cref{thm:di-type-bound}), which averages over exceptional Maass forms (and their levels) more carefully, ultimately allowing us to go beyond $3/5 = 0.6$ unconditionally. Our exponent of $66/107 \approx 0.617$ uses the best progress towards Selberg's conjecture, due to Kim--Sarnak \cite[Appendix 2]{kim2003functoriality} (based on the automorphy of symmetric fourth-power $L$-functions).

\begin{notation}[Exceptional eigenvalues] \label{not:exceptional}
For $q \in \Z_+$, define $\theta_q := \sup_\lambda \sqrt{\max(0,1 - 4\lambda)}$, where $\lambda$ runs over all eigenvalues of the hyperbolic Laplacian for the Hecke congruence subgroup $\Gamma_0(q)$ (such $\lambda$ is called exceptional iff $\lambda < 1/4$). Also, let $\theta_{\max} := \sup_{q \ge 1} \theta_q$. 
\end{notation}

\begin{conjecture}[Selberg \cite{selberg1965estimation}] \label{conj:selberg}
One has $\theta_{\max} = 0$, i.e., there are no exceptional eigenvalues.
\end{conjecture}

\begin{knowntheorem}[Kim--Sarnak \cite{kim2003functoriality}] \label{thm:kim-sarnak}
One has $\theta_{\max} \le 7/32$.
\end{knowntheorem}

\begin{remark}
We warn the reader of another common normalization for the $\theta$-parameters, which differs by a factor of $2$ (resulting in a bound of $7/64$ in \cref{thm:kim-sarnak}); our normalization follows \cite{deshouillers1982kloosterman, maynard2025primes}. We give more details on the role of exceptional eigenvalues in our work in \cref{sec:deshouillers-iwaniec}.
\end{remark}
We now state a more general version of our main result from \cref{thm:distribution-simplified}, which makes the dependency on $\theta_{\max}$ explicit, gives a refined bound in the right-hand side (following \cite{drappeau2015theoremes}), and allows for some small uniformity in the residue parameter $a$.

\begin{theorem}[Conditional exponent of distribution] \label{thm:distribution}
For any $\eps > 0$, there exist $C, \delta > 0$ such that the following holds. 
Let $x > 2$, $(\log x)^C \le y \le x^{1/C}$, and denote $u := (\log x)/(\log y)$, $H(u) := \exp \left(u \log^{-2} (u+1)\right)$. Then with an exponent of
\begin{equation} \label{eq:exponent-theta}
    \alpha := \frac{5-4\theta_{\max}}{8-6\theta_{\max}} - \eps,
\end{equation}
one has
\begin{equation} \label{eq:exponent}
    \sum_{\substack{q \le x^{\alpha} \\ (q, a_1 a_2) = 1}} 
    \left\vert \Psi(x, y; a_1 \bar{a_2}, q) - \frac{\Psi_q(x, y)}{\varphi(q)}
    \right\vert 
    \ll_{\eps, A} 
    \Psi(x, y) \left(H(u)^{-\delta} (\log x)^{-A} + y^{-\delta}\right),
\end{equation}
for all $a_1, a_2 \in \Z$ with $1 \le |a_1|, |a_2| \le x^\delta$, and all $A \ge 0$. The implicit constant is effective if $A < 1$.
\end{theorem}

\begin{remark}
In \cref{eq:exponent}, $\bar{a_2}$ denotes a multiplicative inverse of $a_2$ modulo $q$; so the residue $a_1\bar{a_2}$ corresponds to congruences of the form $a_2 n \equiv a_1 \pmod{q}$. The right-hand side of \cref{eq:exponent} is the same as in Drappeau's result \cite[Th\'eor\`eme 1]{drappeau2015theoremes}, and ultimately comes from a result of Harper (see \cref{lem:harper}). 
\end{remark}
In particular, \cref{conj:selberg} would imply an exponent of distribution of $5/8 - o(1)$, while \cref{thm:kim-sarnak} leads to the unconditional exponent of $66/107 - o(1)$ from \cref{thm:distribution-simplified}. 
As in previous approaches, our main technical result leading to \cref{thm:distribution} is a triple convolution estimate, given in \cref{thm:triple-conv-estimate}; this improves on \cite[Th\'eor\`eme 4]{bombieri1986primes}, \cite[Th\'eor\`eme 3]{drappeau2015theoremes}, \cite[Lemma 2.3]{drappeau2017smooth}, and \cite[Proposition 8.3]{maynard2025primes}. We expect all such approaches to face a significant barrier at the exponent $2/3 = 0.\bar{6}$ (see the remark after \cref{eq:optimal-parameters}), so we may view the exponents of $66/107 \approx 0.617$ and $5/8 = 0.625$ as progress towards this limit.

In fact, \cref{thm:triple-conv-estimate} is already in a suitable form to improve the analogous results of Drappeau, Granville and Shao about smooth-supported multiplicative functions \cite{drappeau2017smooth}. More precisely, using \cref{thm:triple-conv-estimate} instead of \cite[Lemma 2.3]{drappeau2017smooth}, one can improve the exponent of $3/5$ in \cite[Theorem 1.2]{drappeau2017smooth} to the same value as in \cref{eq:exponent-theta}. We state a particular case of this result below, borrowing the notation
\[
    \Delta(f, x; q, a) := \sum_{\substack{n \le x \\ n \equiv a \pmod{q}}} f(n) 
    -
    \frac{1}{\varphi(q)} \sum_{\substack{n \le x \\ (n, q) = 1}} f(n)
\]
from \cite{drappeau2017smooth}; we also say that an arithmetic function $f$ satisifes the \emph{Siegel--Walfisz criterion} iff
\[
    \forall\ A > 0,\ (a, q) = 1,\ x \ge 2 : 
    \qquad 
    \Delta(f, x; q, a) \ll_A \frac{1}{(\log x)^A} \sum_{n \le x} |f(n)|.
\]

\begin{theorem}[Smooth-supported multiplicative functions in APs] \label{thm:distribution-functions}
For any $\eps, A > 0$, there exists $\delta > 0$ such that the following holds. 
Let $x > 2$, $x^\delta \ge y \ge \exp(\sqrt{\log x} \log \log x)$, and $f$ be a $1$-bounded completely multiplicative function supported on $y$-smooth integers, satisfying the Siegel--Walfisz criterion. Then for $\alpha := \frac{5-4\theta_{\max}}{8-6\theta_{\max}} - \eps$, and all $a_1, a_2 \in \Z$ with $1 \le |a_1|, |a_2| \le x^{\delta}$,
\[
    \sum_{\substack{q \le x^{\alpha} \\ (q, a_1a_2) = 1}} |\Delta(f, x; q, a_1 \bar{a_2})|
    \ll_{\eps,A} 
    \frac{\Psi(x, y)}{(\log x)^A}.
\]
\end{theorem}

\begin{remark}
The improvement from $3/5-\eps$ to the exponent in \cref{eq:exponent-theta} follows through in most applications of Drappeau's result \cite[Th\'eor\`eme 1]{drappeau2015theoremes}, such as \cite[Corollaire 1]{drappeau2015theoremes}. Following \cite[Sections 2 and 4]{de2020niveau}, our triple convolution estimate also implies a version of \cref{thm:distribution} restricted to smooth moduli, which can be used to deduce refined upper bounds for the number of smooth values assumed by a factorable quadratic polynomial. For instance, one should obtain
\[
    \#\{n \sim x : n(n+1) \text{ is $y$-smooth} \} \ll_\eps 
    x \varrho(u)^{1+66/107-\eps},
\]
for $(\log x)^{O_\eps(1)} \le y \le x$,
where $u = (\log x)/(\log y)$ and $\varrho(u)$ is the Dickman function (satisfying $\Psi(x, y) = x \varrho(u) e^{O_\eps(u)}$ in this range \cite{de2020niveau,hildebrand1986number}).
\end{remark}

\section{Overview of key ideas} \label{sec:overview}

Let us give a very rough sketch of our argument, for a simpler case of \cref{thm:distribution-simplified}. Consider the residue $a = 1$, the smoothness parameter $y = x^{1/\sqrt{\log \log x}}$, and a sum over moduli just above the $3/5$ threshold, say
\[
    r \le R := x^{3/5 + \sigma},
\]
for some small $\sigma \gg 1$ (we switch the variable $q$ to $r$, following \cite{drappeau2015theoremes,drappeau2017smooth}). 
Using the factorization properties of smooth numbers, it suffices to prove a triple convolution estimate roughly of the form
\begin{equation} \label{eq:sketch-triple-conv}
    \sum_{r \sim R} \rho_r \sum_{m \sim M} \alpha_m
    \sum_{n \sim N} \beta_n
    \sum_{\ell \sim L} \gamma_\ell \left(1_{mn\ell \equiv 1 \pmod{r}} - \frac{\one_{(mn\ell, r)=1}}{\varphi(r)} \right)
    \ll_A
    \frac{x}{(\log x)^A},
\end{equation}
where $(\rho_r)$, $(\alpha_m)$, $(\beta_n)$, $(\gamma_\ell)$ are arbitrary $1$-bounded complex sequences, but we are free to choose the parameters $M, N, L \gg 1$ subject to $MNL \asymp x$. We pick
\begin{equation} \label{eq:sketch-ranges}
    M := \frac{x^{1-\delta}}{R}, 
    \qquad\qquad
    N := \frac{x^{1-2\delta}}{R},
    \qquad\qquad 
    L := \frac{R^2}{x^{1-3\delta}},
\end{equation}
for some small $\delta = o(1)$; thus $M, N \approx x^{2/5-\sigma}$ and $L \approx x^{1/5+2\sigma}$. Note additionally that $NL = x^{\delta} R$.

\subsection{First steps and limitations of previous approaches} \label{subsec:first-steps}
Following previous works \cite{bombieri1986primes,drappeau2015theoremes,maynard2025primes} based on Linnik's \emph{dispersion method} \cite{linnik1963dispersion}, we apply Cauchy--Schwarz in the $r, m$ variables, expand the square, and Fourier-complete the resulting sums in $m$ to sums over $h$. Ignoring GCD constraints, the key resulting exponential sum is a smoothed variant of
\begin{equation} \label{eq:sketch-main-exponential}
    \sum_{r \sim R}\,
    \sum_{k \sim NL} u_k
    \sum_{n \sim N} \beta_n
    \sum_{\substack{\ell \sim L \\ n\ell \equiv k \pmod{r}}} \gamma_\ell
    \sum_{h \sim R/M} \e\left(h \frac{\bar{k}}{r} \right),
\end{equation}
where $(u_k)$ is the convolution of the original sequences $(\beta_n)$, $(\gamma_\ell)$. We then flip moduli in the exponential via Bezout's identity, and substitute $t := (k - n\ell)/r$; this leads to the sum
\begin{equation} \label{eq:sketch-bfi-style}
    \sum_{t \sim x^\delta}\,
    \sum_{k \sim NL} u_k
    \sum_{n \sim N} \beta_n
    \sum_{\substack{\ell \sim L \\ n\ell \equiv k \pmod{t}}} \gamma_\ell
    \sum_{h \sim R/M} \e\left(h t \frac{\bar{n\ell}}{k} \right).
\end{equation}
Following Maynard \cite{maynard2025primes}, we apply Cauchy--Schwarz in the $t, n, k$ variables (keeping the congruence modulo $t$ inside), and expand the square to reach the sum 
\begin{equation} \label{eq:sketch-bfi-style-2}
    \sum_{t \sim x^\delta}
    \sum_{\substack{\ell_1, \ell_2 \sim L \\ \ell_1 \equiv \ell_2 \pmod{t}}}
    \gamma_{\ell_1} \bar{\gamma_{\ell_2}}
    \sum_{h_1, h_2 \sim R/M}
    \sum_{k \sim NL}
    \sum_{\substack{n \sim N \\ n \equiv k\bar{\ell_1} \pmod{t}}}
    \e\left(t(h_1 \ell_2 - h_2 \ell_1) \frac{\bar{n\ell_1\ell_2}}{k} \right),
\end{equation}
which we ultimately need to bound by $\ll_A RNL^2 (\log x)^{-A}$; note that the trivial bound is larger by about $(R/M)^2$, due to the sum over $h_1, h_2$ introduced by Poisson summation. 
We bound the contribution of the diagonal terms (with $h_1\ell_2 = h_2\ell_1$) by
\begin{equation} \label{eq:sketch-diagonal-terms}
    \sum_{t \sim x^\delta}\,
    \sum_{\ell_1 \sim L}\,
    \sum_{h_2 \sim R/M} x^{o(1)}
    \sum_{k \sim NL} 
    \sum_{\substack{n \sim N \\ n \equiv k\bar{\ell_1} \pmod{t}}} 1
    \ll 
    x^{o(1)} RNL^2 \frac{N}{M},
\end{equation}
which is acceptable since $N/M = x^{-\delta}$.
We then introduce Kloosterman sums $S(i, j; k)$ by completing the sum in $n$ to a sum over $j$, and find acceptable contributions from the zeroth Fourier coefficient ($j = 0$), as well as from the terms with $\ell_1 = \ell_2$; here it suffices to use the Ramanujan bound, respectively an estimate of Deshouillers--Iwaniec \cite[Theorem 9]{deshouillers1982kloosterman}. It ultimately remains to bound a variant of
\begin{equation} \label{eq:sketch-final-sum}
    \sum_{t \sim x^\delta}
    \sum_{\substack{\ell_1, \ell_2 \sim L \\ \ell_1 \equiv \ell_2 \pmod{t} \\ \ell_1 \neq \ell_2}}
    \sum_{\substack{h_1, h_2 \sim R/M \\ h_1\ell_2 \neq h_2\ell_1}}
    \left\vert
    \sum_{j \sim x^\delta L} \e\left(\frac{j\bar{\ell_1}}{t}\right)  
    \sum_{k \sim NL}
    S\left((h_1\ell_2 - h_2\ell_1)\bar{\ell_1\ell_2}, j; k\right)\right\vert,
\end{equation}
by $\ll_A N^2 L^4 (\log x)^{-A}$. Inserting $1$-bounded coefficients $\xi_{h_1, h_2}$ (also depending on $t, \ell_1, \ell_2$), and letting $a_d := \sum_{h_1 \ell_2 - h_2 \ell_1 = d} \xi_{h_1, h_2}$, the sum over $h_1, h_2$ in \cref{eq:sketch-final-sum} roughly reduces to
\begin{equation} \label{eq:sketch-kloosterman}
    \sum_{d \sim RL/M} a_d \sum_{j \sim x^\delta L} \e\left(\frac{j\bar{\ell_1}}{t}\right) 
    \sum_{k \sim NL} S\left(d \bar{\ell_1 \ell_2}, j; k\right).
\end{equation}
Such sums can be bounded using the spectral theory of automorphic forms, specifically through the aforementioned work of Deshouillers--Iwaniec \cite{deshouillers1982kloosterman} (based on the Kuznetsov trace formula and the Weil bound); the relevant \emph{level} of the congruence group in \cref{not:exceptional} is $Q = \ell_1\ell_2$. Indeed, Maynard \cite{maynard2025primes} uses \cite[Theorem 9]{deshouillers1982kloosterman} to bound (a smoothed variant of) the sum in \cref{eq:sketch-kloosterman} by
\[
    x^{\theta_Q/2 + o(1)} NL^{5/2} \left(\sum_{d \sim RL/M} |a_d|^2\right)^{1/2},
\]
and consequently the sum in \cref{eq:sketch-final-sum} by
\begin{equation} \label{eq:sketch-final-bound}
    x^{\theta_{\max}/2 + o(1)} NL^{5/2} \frac{R^{3/2}L^{3/2}}{M^{3/2}} 
    \approx 
    x^{(\theta_{\max}+3)/2 + 10\sigma}.
\end{equation}
Unfortunately, this falls short of the desired bound of $N^2 L^4 (\log x)^{-A} \approx x^{8/5 + 6\sigma}$, unless
\begin{equation} \label{eq:sketch-sigma}
    4\sigma < \frac{1}{10} - \frac{\theta_{\max}}{2}.
\end{equation}
This is (barely) impossible with the currently best-known bound of $\theta_{\max}/2 \le 7/64 \approx 0.109$.

\subsection{Improved exponential sum manipulations for specific ranges}

Starting from the work of Deshouillers--Iwaniec \cite{deshouillers1982kloosterman}, better bounds (in the $\theta$-aspect) for sums like \cref{eq:sketch-kloosterman} have been available when one additionally averages over the level $\ell_1\ell_2$, and at least one of the sequences of coefficients is independent of the level. Indeed, Drappeau's triple convolution estimate \cite{drappeau2015theoremes} and prior works rely on \cite[Theorem 12]{deshouillers1982kloosterman}, which gives such a result for incomplete Kloosterman sums.

Following Maynard's argument (which is in turn based on Bombieri--Friedlander--Iwaniec's work in \cite[Section 10]{bombieri1987primes2}), we prefer to complete our Kloosterman sums and bound the contribution from the zeroth Fourier coefficient by hand, and separate into terms with $\ell_1 = \ell_2$ and $\ell_1 \neq \ell_2$, all before invoking Deshouillers--Iwaniec-style bounds. We then aim to apply an optimized bound for sums of complete Kloosterman sums with averaging over the level (given in \cref{thm:di-type-bound}), which improves \cite[Theorem 11]{deshouillers1982kloosterman} by making the dependency on the $\theta_{\max}$ parameter explicit. But for this strategy to work out, we would need:
\begin{itemize}
    \item[(1).] The range of $(\ell_1, \ell_2)$ in \cref{eq:sketch-final-sum} to be (discretely) dense inside $[L, 2L]^2$, and
    \item[(2).] Crucially, the coefficients $e\left(j \bar{\ell_1} / t \right)$ to not depend on $\ell_1, \ell_2$.
\end{itemize}
While (2) is obviously false in our case, it is only barely false for the specific ranges in \cref{eq:sketch-ranges}, due to the smallness of the parameter $t \sim x^\delta$. In particular, losing a factor of at most $x^{O(\delta)} = x^{o(1)}$ in \cref{eq:sketch-final-sum}, we may fix $t$ and the values of $\ell_1$ and $\ell_2 \pmod{t}$, turning (2) into a true statement at the expense of (1). The number of pairs $(\ell_1, \ell_2)$ now becomes $\asymp L^2/t^2$, which ends up costing us another acceptable factor of $x^{o(1)}$. Overall, it remains to bound a sum of the form
\begin{equation} \label{eq:sketch-kloosterman-2}
    \sum_{\ell_1,\ell_2 \sim L}\, \sum_{d \sim RL/M} a_{d,\ell_1,\ell_2} \sum_{j \sim x^\delta L} \e\left(j\omega\right) 
    \sum_{k \sim NL} S\left(d \bar{\ell_1 \ell_2}, j; k\right),
\end{equation}
for some fixed $\omega \in \R/\Z$ (independent of $\ell_1, \ell_2$). Using \cref{thm:di-type-bound}, we obtain a bound like \cref{eq:sketch-final-bound} where the factor depending on $\theta_{\max}$ is
\[
    \left(\frac{NL}{L^2}\right)^{\theta_{\max}} \approx x^{\left(\frac{1}{5}-3\sigma\right) \theta_{\max}},
\]
rather than $x^{\theta_{\max}/2}$. Thus instead of \cref{eq:sketch-sigma}, we now reach the desired bound provided that
\[
    \left(4 - 3\theta_{\max}\right)\sigma < \frac{1}{10} - \frac{\theta_{\max}}{5},
\]
which is possible since $(1/5) \cdot (7/32) = 0.04375 < 0.1$. In fact, this handles all values
\[
    \sigma < \frac{1-2\theta_{\max}}{40-30\theta_{\max}}
    \qquad\qquad 
    \iff 
    \qquad\qquad 
    \frac{3}{5} + \sigma < 
    \frac{5-4\theta_{\max}}{8-6\theta_{\max}},
\]
reaching the exponent of distribution in \cref{eq:exponent-theta}. Plugging in Kim--Sarnak's bound of $\theta_{\max} \le 7/32$ (\cref{thm:kim-sarnak}) yields the unconditional exponent of $66/107 \approx 0.617$ from \cref{thm:distribution-simplified}.

\begin{remark}
It is likely that optimized Deshouillers--Iwaniec-style bounds like \cref{thm:di-kloosterman} could also improve Drappeau's argument \cite{drappeau2015theoremes}, leading to a triple convolution estimate with different ranges than in our \cref{thm:triple-conv-estimate}. In terms of the final exponent of distribution of smooth numbers, all such methods currently seem limited below $66/107$ unconditionally (and $5/8$ conditionally).
\end{remark}

\subsection{Completing the argument} \label{subsec:completing-argument}

To increase the range of uniformity in $y$, we adapt Drappeau's version of the dispersion method \cite{drappeau2015theoremes}: we aim for a triple convolution estimate with a power-saving in \cref{thm:triple-conv-estimate}, after separating the contribution of small-conductor Dirichlet characters
\[
    \frac{1}{\varphi(r)} \sum_{\substack{\chi \pmod{r} \\ \cond(\chi) \le x^\eps}} \chi(mn\ell)
\]
from \cref{eq:sketch-triple-conv}; this can be handled via \cref{lem:harper}. As a result, the simpler two dispersion sums $\mS_2$, $\mS_3$ and their main terms involve Dirichlet characters (see \cref{prop:setup,prop:main-terms}), which ultimately bring in the classical Gauss sum bound (\cref{lem:gauss}) and the multiplicative large sieve (\cref{lem:large-sieve}). The difficulties in working with a general residue $a_1 \bar{a_2}$ for $a_1, a_2 \ll x^\delta$, and in obtaining power savings throughout the computations in \cref{subsec:first-steps}, are quite tedious but purely technical (following \cite{drappeau2015theoremes}).

We also adapt a ``deamplification'' argument of Maynard \cite{maynard2025primes}, which introduces an artificial sum over $e \sim E = x^{o(1)}$ into the dispersion sums (by averaging over the residue of $n\ell \pmod{e}$ before applying Cauchy--Schwarz); for instance, the sum in \cref{eq:sketch-main-exponential} becomes
\[
    \sum_{e \sim E}\, 
    \sum_{r \sim R}\,
    \sum_{k \sim NL} u_k
    \sum_{n \sim N} \beta_n
    \sum_{\substack{\ell \sim L \\ n\ell \equiv k \pmod{re}}} \gamma_\ell
    \sum_{h \sim R/M} \e\left(h \frac{\bar{k}}{r} \right).
\]
Keeping $e$ inside the second application of Cauchy--Schwarz, this essentially reduces the contribution of the diagonal terms in \cref{eq:sketch-diagonal-terms} by a factor of $E$, allowing us to cover wider ranges of sequence lengths in \cref{thm:triple-conv-estimate} (including the case $M = N$).
This is generally convenient, and critical when one has less control over the sizes of the sequence lengths (which is the case in applications to the primes, but not to smooth numbers). 

\cref{fig:outline} gives a visual summary of our formal argument, outlining the logical dependencies between our main lemmas, propositions, and theorems.

\begin{figure}[ht]
\centering 
\begin{tikzpicture}[
square/.style={rectangle, draw=black!60, fill=white, very thick, minimum size=5mm},
important/.style={rectangle, double, draw=black!60, fill=white, very thick, minimum size=5mm},
]
\node[square, fill=teal!11] (distribution)
    {\twoline{Smooth numbers in APs}{(\cremph{thm:distribution})}};
\node[square, fill=teal!1] (small-conductors) [xshift=-3.2cm] [below=of distribution]
    {\twoline{Contribution of small-conductor}{characters (\cremph{lem:harper})}};
\node[important, fill=teal!11] (triple-conv) [right = 1cm of small-conductors]
    {\twoline{Triple convolution estimate}{(\cremph{thm:triple-conv-estimate})}};
\node[square, fill=teal!7] (setup) [xshift=-7cm] [below=of triple-conv]
    {\twoline{Dispersion $+$ deamplification}{setup (\cremph{prop:setup})}};
\node[square, fill=teal!7] (sums) [right=0.5cm of setup]
    {\twoline{Dispersion sum estimates}{(\cremph{prop:dispersion-sums})}};
\node[square, fill=teal!7] (main-terms) [right=0.5cm of sums]
    {\twoline{Contribution of main}{terms (\cremph{prop:main-terms})}};
\node[square, fill=teal!1] (large-sieve) [below=0.5cm of main-terms]
    {\twoline{Large sieve}{(\cremph{lem:large-sieve})}};
\node[square, fill=teal!7] (bfi-expo) [below=0.5cm of large-sieve]
    {\twoline{Improved BFI-style bound}{(\cremph{prop:bfi-expo,prop:bfi-expo-2})}};
\node[square, fill=teal!1] (weil-bound) [left=0.3cm of bfi-expo]
    {\twoline{Incomplete Weil}{bound (\cremph{lem:incomplete-weil})}};
\node[square, fill=teal!1] (gauss) [left=0.3cm of weil-bound]
    {\twoline{Gauss sum bound}{(\cremph{lem:gauss})}};
\node[square, fill=teal!1] (poisson) [left=0.3cm of gauss]
    {\twoline{Truncated Poisson}{(\cremph{lem:poisson})}};
\node[square, fill=teal!1] (completion) [xshift=2cm] [below=of gauss]
    {\twoline{Completion of Kloosterman}{sums (\cremph{lem:completion})}};
\node[square, fill=teal!11] (avg-forms) [right=1cm of completion] 
    {\twoline{DI-type Kloosterman}{sum bound (\cremph{thm:di-type-bound})}};

\draw[-] (triple-conv) -- ([yshift=-0.4cm] distribution.south);
\draw[-] (small-conductors) -- ([yshift=-0.4cm] distribution.south);
\draw[-triangle 45] ([yshift=-0.4cm] distribution.south) -- (distribution.south);
\draw[-triangle 45] (large-sieve) -- (main-terms);
\draw[-] (setup) -- ([yshift=-0.4cm] triple-conv.south);
\draw[-] (sums) -- ([yshift=-0.4cm] triple-conv.south);
\draw[-] (main-terms) -- ([yshift=-0.4cm] triple-conv.south);
\draw[-triangle 45] ([yshift=-0.4cm] triple-conv.south) -- (triple-conv.south);
\draw[-] (poisson) -- ([yshift=-0.4cm] sums.south);
\draw[-] (gauss) -- node[midway,fill=white]{$\mS_3$} ([yshift=-0.4cm] sums.south);
\draw[-] (weil-bound) -- node[midway,fill=white]{$\mS_2$} ([yshift=-0.4cm] sums.south);
\draw[-] (bfi-expo) -- node[midway,fill=white]{$\mS_1$} ([yshift=-0.4cm] sums.south);
\draw[- triangle 45] ([yshift=-0.4cm] sums.south) -- (sums.south);
\draw[-] (avg-forms) -- ([xshift=-1.5cm, yshift=-0.3cm] bfi-expo.south);
\draw[-] (completion) -- ([xshift=-1.5cm, yshift=-0.3cm] bfi-expo.south); 
\draw[- triangle 45] ([xshift=-1.5cm, yshift=-0.3cm] bfi-expo.south) -- (bfi-expo);
\end{tikzpicture}
\caption{Structure of argument (\emph{arrows show logical implications}).}
\label{fig:outline}
\end{figure}
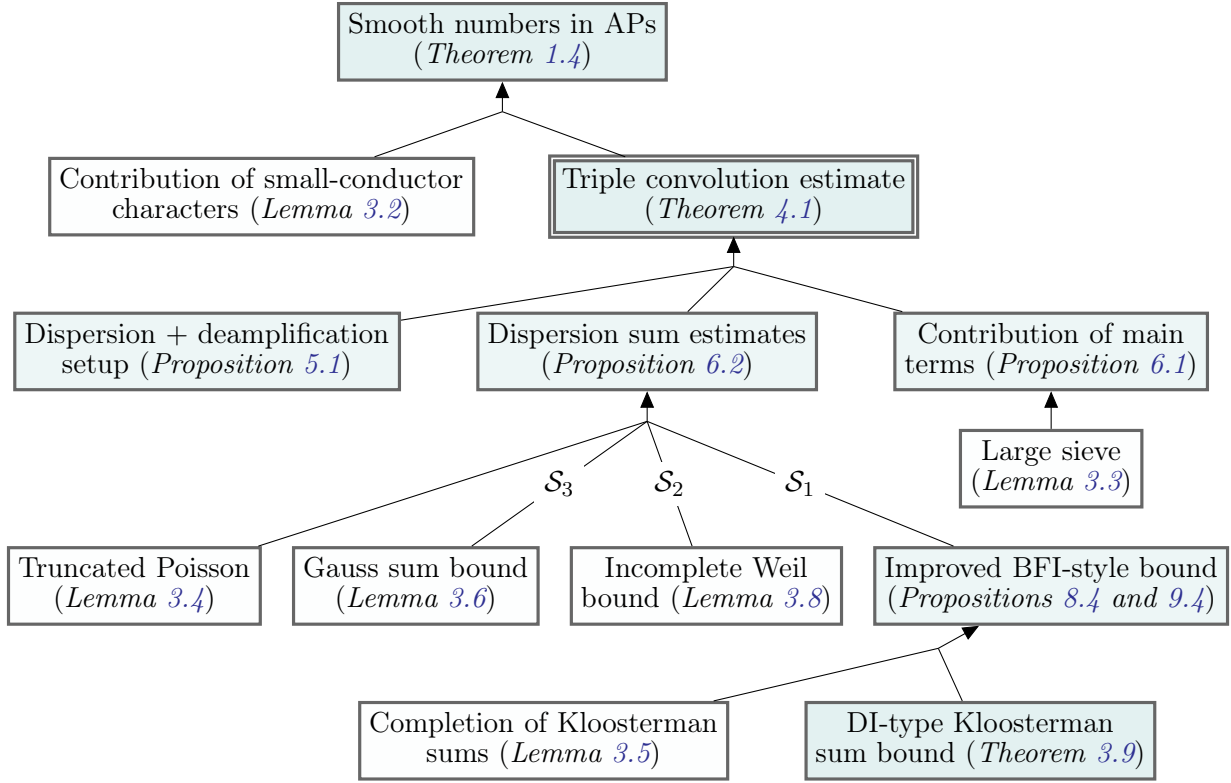

\section{Notation and preliminaries}

\subsection{Sets, sums, estimates, congruences}
We use the standard asymptotic notation in analytic number theory, with $f = O(g)$ (or $f \ll g$) meaning that there exists some constant $C > 0$ such that $|f| \le Cg$ globally. We write $f \asymp g$ when $f \ll g$ and $g \ll f$, and indicate that the implied constants may depend on a parameter $\eps$ by placing it in the subscript (e.g., $f = O_\eps(g)$, $f \ll_\eps g$, $f \asymp_\eps g$). When $g \ge 0$, we also say that $f = o(g) = o_{x \to \infty}(g)$ if and only if $f(x)/g(x) \to 0$ as $x \to \infty$. Given $q \in [1, \infty]$, we write $\|f\|_q$ for the $L^q$ norm of a measurable function $f : \R \to \C$ (using the Lebesgue measure), and $\|a_n\|_q$ for the $\ell^q$ norm of a complex sequence $(a_n)$.

We denote by $\Z_+, \Z, \R, \C, \H$ the sets of positive integers, integers, real numbers, complex numbers, and complex numbers with positive imaginary part, and set $\e(x) := \exp(2 \pi i x)$ for $x \in \R$ (or $x \in \R/\Z$). We write $\Z/n\Z$ and $(\Z/n\Z)^\times$ for the additive and multiplicative groups modulo a positive integer $n$, and denote the inverse of $c \in (\Z/n\Z)^\times$ by $\bar{c}$. We may abuse notation slightly by identifying integers 
$a, b, c$ with their residue classes modulo $n$ where this is appropriate (e.g., in congruences $a \equiv b \bar{c} \pmod{\pm n}$, $x \equiv b\bar{c}/n \pmod{1}$, or in exponentials $\e (b \bar{c}/n)$); the following simple lemma is an example of this.

\begin{lemma}[B\'ezout's identity] \label{lem:bezout}
For any relatively prime integers $a, b$, one has
\[
    \frac{1}{ab} \equiv \frac{\bar{a}}{b} + \frac{\bar{b}}{a} \quad \pmod{1}.
\]
\end{lemma}
\begin{proof}
Note that here, $\bar{a}$ and $\bar{b}$ denote the inverses of $a$ and $b$ modulo $b$ and $a$ respectively, so we have $a \bar{a} \equiv 1 \pmod{b}$ and $b \bar{b} \equiv 1 \pmod{a}$. The conclusion follows from the Chinese remainder theorem, once we multiply the congruence by $ab$ and verify it modulo $a$ and $b$ separately.
\end{proof}

Given $N > 0$, we write $n \sim N$ for the statement that $N < n \le 2N$, usually in the subscripts of sums. Given a statement $S$, we write $\one_S$ for its truth value (e.g., $\one_{2 \mid n}$ equals $1$ when $n$ is even and $0$ otherwise); we may use the same notation for the indicator function of a set $S$ (i.e., $\one_S(x) = \one_{x \in S}$). 

Given $a_1, \ldots, a_k \in \Z$, we write $(a_1, \ldots, a_k)$ (if not all $a_i$ are $0$) and $[a_1, \ldots, a_k]$ (if none of the $a_i$ are $0$) for their greatest common divisor and lowest common multiple, among the positive integers. 
Given $a \in \Z \setminus \{0\}$, we write $\rad(a)$ for the largest square-free positive integer dividing $a$; for $b \in \Z$, we also write $a \mid b^\infty$ if and only if $\rad(a) \mid b$ (i.e., $a$ divides a large enough power of $b$), and $(a, b^\infty)$ for the greatest divisor of $a$ whose prime factors divide $b$. If $x > 0$ and $m \in \Z \setminus \{0\}$, sums like $\sum_{n \le x}$, $\sum_{n \sim x}$, $\sum_{d \mid m}$, $\sum_{d \mid m^\infty}$, $\sum_{(a, m) = 1}$, $\sum_{(a, m^\infty) = 1}$ and $\sum_{ab = m}$ are understood to range over all positive integers $n, d, a, b$ with the respective properties. 

We also keep the notations specific to smooth numbers from the introduction, for $S(x, y)$, $\Psi(x, y) = \# S(x, y)$, $\Psi_q(x, y)$, $\Psi(x, y; a, q)$, and $H(u)$.

\subsection{Multiplicative number theory}

We denote by $\mu$, $\tau$ and $\varphi$ the M\"obius function, the divisor-counting function ($\tau(n) := \sum_{d \mid n} 1$) and the Euler totient function ($\varphi(n) := \sum_{1 \le a \le n} \one_{(a, n) = 1}$). We may use various classical bounds involving these functions implicitly, including the divisor bound $\tau(n) \ll_\eps n^\eps$ (valid for all $\eps > 0$), the lower bound $\varphi(n) \gg n/(\log \log n)$, and the upper bounds
\[
    \sum_{n \sim N} \frac{1}{\varphi(n)} \ll 1, \qquad\qquad 
    \sum_{n \le x} \frac{\tau(n)}{\varphi(n)} \ll (\log x)^2.
\]
(The latter follows from the former, using that $\varphi(ab) \gg \varphi(a) \varphi(b)$ for positive integers $a, b$.)

We write $\chi \pmod{q}$ to indicate that $\chi$ is a Dirichlet character with period $q$ (of which there are $\varphi(q)$), and denote by $\cond(\chi)$ the conductor of $\chi$ (which divides $q$; this is the smallest positive integer $d$ such that there exists a Dirichlet character $\chi' \pmod{d}$ with $\chi(n) = \chi'(n) \one_{(n, q) = 1}$ for all $n \in \Z$); we say that $\chi \pmod{q}$ is \emph{primitive} when $\cond(\chi) = q$. We will require a couple of results involving Dirichlet characters, the first being essentially due to Harper.

\begin{lemma}[Contribution of small-conductor characters] \label{lem:harper}
There exist constants $\eps, \delta, C > 0$ such that for $(\log x)^C \le y \le x$, $Q \le x$ and $A > 0$, one has
\[
    \sum_{q \le Q} \frac{1}{\varphi(q)} \sum_{\substack{\chi \pmod{q} \\ 1 < \cond(\chi) \le x^\eps}} 
    \left\vert \sum_{n \in S(x, y)} \chi(n) \right\vert \ll_A 
    \Psi(x, y) 
    \left(H(u)^{-\delta} (\log x)^{-A} + y^{-\delta} \right),
\]
with $u := (\log x)/(\log y)$, $H(u) := \exp\left(u \log^{-2}(u+1)\right)$. The implicit constant is effective if $A < 1$.
\end{lemma}

\begin{proof}
This is the same as \cite[Lemme 5]{drappeau2015theoremes}, and follows from the work of Harper in \cite[Section 3.3]{harper2012bombieri}.
\end{proof}

\begin{remark}
The condition $\cond(\chi) > 1$ in \cref{lem:harper} leaves out the trivial character $\chi_0$.
\end{remark}

The second result is the classical multiplicative large sieve, as stated in \cite[Lemme 6]{drappeau2015theoremes}.

\begin{lemma}[Multiplicative large sieve] \label{lem:large-sieve}
For $Q, M, N \ge 1$ and any sequence $(a_n)$ of complex numbers, one has
\[
    \sum_{q \le Q} \frac{q}{\varphi(q)} \sum_{\substack{\chi \pmod{q} \\ \chi \textnormal{ primitive}}} 
    \left\vert 
    \sum_{M < n \le M+N} a_n \chi(n)
    \right\vert^2
    \le 
    (N + Q^2 - 1) 
    \sum_{M < n \le M+N} |a_n|^2.
\]
\end{lemma}

\begin{proof}
See, for example, \cite[Theorem 7.13]{iwaniec2021analytic}. 
\end{proof}

\subsection{Fourier analysis}
Given an integrable function $f : \R \to \C$, we write
\[
    \hat{f}(\xi) := \int_{-\infty}^\infty f(t)\, \e(-\xi t)\ dt
\]
for its Fourier transform. We will need the truncated version of Poisson summation stated below.

\begin{lemma}[Truncated Poisson / Fourier completion] \label{lem:poisson}
Let $C > 0$, $x > 1$, $1 < M \ll x$, and $\Phi : \R \to \R$ be a smooth function supported in $[1/10, 10]$ such that $\|\Phi^{(j)}\|_\infty \ll_j (\log x)^{jC}$ for $j \ge 0$. Then for all positive integers $q \ll x$, any $a \in \Z/q\Z$, and any $\eps > 0$, $H \ge x^{\eps} qM^{-1}$, one has
\[
    \sum_{m \equiv a \pmod{q}} \Phi\left(\frac{m}{M}\right) = \frac{M}{q} \sum_{|h| \le H} \hat{\Phi} \left(\frac{hM}{q} \right) \e\left(\frac{ah}{q}\right)
    + 
    O_{\eps,C}\left(x^{-100}\right).
\]
\end{lemma}

\begin{proof}
This is the same as \cite[Lemma 13.4]{maynard2025primes} (see also \cite[Lemme 2]{drappeau2015theoremes}), following directly from the Poisson summation formula.
\end{proof}

While \cref{lem:poisson} will introduce exponential sums into our estimates, we will need an additional corollary (and generalization) of it to obtain sums of complete \emph{Kloosterman sums}, defined by
\[
    S(m, n; c) := \sum_{b \in (\Z/c\Z)^\times} 
    \e\left(\frac{mb + n\bar{b}}{c}\right),
    \qquad\qquad 
    c \in \Z_+,\ m, n \in \Z \textnormal{ (or $\Z/c\Z$)}.
\]
The following is the same as \cite[Lemma 13.5]{maynard2025primes}, and can be quickly deduced from \cref{lem:poisson}.

\begin{lemma}[Kloosterman completion] \label{lem:completion}
Let $C, x, M, \Phi$ be as in \cref{lem:poisson}. Then for all positive integers $c, q \ll x$ with $(c, q) = 1$, any $a \in \Z/q\Z$, $n \in \Z/c\Z$, and any $\eps > 0$, $H \ge x^\eps cq M^{-1}$, one has
\[
    \sum_{\substack{m \equiv a \pmod{q} \\ (m, c) = 1}} 
    \Phi\left(\frac{m}{M}\right) \e\left(\frac{\bar{m}n}{c}\right)
    =
    \frac{M}{cq} \sum_{|h| \le H} \hat{\Phi}\left(\frac{hM}{cq}\right) \e\left(\frac{ah \bar{c}}{q}\right)
    S(h\bar{q}, n; c)
    +  
    O_{\eps, C}\left(x^{-99}\right).
\] 
\end{lemma} 
\begin{proof}
Rewrite the left-hand side as
\[
    \sum_{b \in (\Z/c\Z)^\times} \e\left(\frac{\bar{b} n}{c} \right)
    \sum_{\substack{m \equiv a \pmod{q} \\ m \equiv b \pmod{c}}} \Phi\left(\frac{m}{M}\right),
\]
and apply \cref{lem:poisson} to expand the inner summation, for the unique residue class $r \in \Z/cq\Z$ which is congruent to $a \pmod{q}$ and to $b \pmod{c}$ (invoking the Chinese remainder theorem). Noting that
\[
    \e \left(\frac{rh}{cq}\right)
    =
    \e \left(\frac{rh \bar{c}}{q} + \frac{rh \bar{q}}{c}\right)
    =
    \e \left(\frac{ah \bar{c}}{q} + \frac{bh \bar{q}}{c}\right)
\]
by \cref{lem:bezout}, swapping sums and taking out the factor depending on $a$, the conclusion follows.
\end{proof}

\subsection{Bounds for exponential sums}

The simpler two of the three dispersion sums arising in our computations (see \cref{eq:disp-sums}) will be estimated using the classical bounds for Gauss and Kloosterman sums.

\begin{lemma}[Gauss sum bound] \label{lem:gauss}
For any $a \in \Z$, $q \in \Z_+$, and Dirichlet character $\chi \pmod{q}$, one has
\[
    \left\vert \sum_{b \in \Z/q\Z} \chi(b)\ \e\left(\frac{ab}{q}\right) \right\vert
    \le 
    \cond(\chi)^{1/2} \sum_{d \mid (a, q)} d.
\]
\end{lemma}
\begin{proof}
This follows from \cite[Lemmas 3.2 and 3.1]{iwaniec2021analytic}, and is also used in \cite[Section 3.2]{drappeau2015theoremes}.
\end{proof}

\begin{lemma}[Weil and Ramanujan bounds] \label{lem:weil}
For $c \in \Z_+$ and $m, n \in \Z$ (or $\Z/c\Z$), one has
\[
    S(m, n; c) \ll \tau(c)\, (m, n, c)^{1/2} c^{1/2}.
\]
For $m = 0$, we have in fact $|S(0, n; c)| \le (n, c)$.
\end{lemma}

\begin{proof}
The first (Weil) bound is \cite[Corollary 11.12]{iwaniec2021analytic}, while the second (Ramanujan) bound can be deduced by M\"obius inversion.
\end{proof}

\begin{lemma}[Incomplete Weil bound] \label{lem:incomplete-weil}
Let $x > 1$, $1 < M \ll x$, and let $n, c, k, \ell \ll x$ be positive integers. Then for any $\eps > 0$, one has
\[
    \sum_{\substack{(m, ck) = 1 \\ \ell \mid m \\ m \le M}} 
    \frac{m}{\varphi(m)} \e\left(\frac{\bar{m}n}{c}\right)
    \ll_{\eps}
    x^\eps \left((n, c)^{1/2} c^{1/2} + \frac{(n, c)}{\ell c} M \right).
\]
\end{lemma}
\begin{proof}
This follows immediately from \cite[(2.5)]{drappeau2015theoremes} and the divisor bound (see also \cite[Lemme 4]{fouvry1982repartition} and \cite[Lemma 16.1]{maynard2025primes}). It can be proven by expanding $m\varphi(m)^{-1} = \sum_{v \mid m^\infty} v^{-1}$ and $\one_{(m, k)} = \sum_{d \mid k} \mu(d) \one_{d \mid m}$, changing variables $m \gets m [\ell, \rad(v), d]$, completing sums via a result like \cref{lem:completion}, and finally applying \cref{lem:weil} for the terms $h = 0$ and $h \neq 0$ separately.
\end{proof}

To estimate the first dispersion sum, will crucially need the following bound for sums of Kloosterman sums, which is an optimization of \cite[Theorem 11]{deshouillers1982kloosterman} (see also \cite[Lemma 5]{bombieri1987primes2}).

\begin{theorem}[DI-type Kloosterman bound] \label{thm:di-type-bound}
Let $1 \ll M, N, R, S, C \ll x^{O(1)}$, $(a_{m,r,s})$ be a complex sequence supported in $m \sim M, r \sim R, s \sim S$, and $\omega \in \R/\Z$. Also, let $g(t)$ be a smooth function supported on $t \asymp 1$, with bounded derivatives $\|g^{(j)}\|_\infty \ll_{j} 1$ for $j \ge 0$. Then for any $\eta > 0$, one has
\[
\begin{aligned}
    \sum_{\substack{r \sim R \\ s \sim S \\ (r, s) = 1}}
    \sum_{m \sim M} a_{m,r,s} \sum_{n \sim N} \e(n\omega)
    \sum_{(c, r) = 1} g\left(\frac{c}{C}\right)\, S(m\bar{r}, \pm n; sc)
    \ll_\eta 
    x^\eta\,
    \left(1 + \frac{C}{R\sqrt{S}}\right)^{\theta_{\max}}\, 
    \|a_{m,r,s}\|_2
    \\
    \times 
    \sqrt{NRS}
    \left(\frac{C^2}{R} (M + RS)(N + RS) + MN \right)^{1/2},
\end{aligned}
\]
where we recall that $\theta_{\max} \le 7/32$ by \cref{thm:kim-sarnak}. (The `$\pm n$' notation indicates that either consistent choice of sign is allowable.)
\end{theorem}

\cref{thm:di-type-bound} makes use of the spectral theory of automorphic forms, and follows from a variation of the landmark arguments of Deshouillers--Iwaniec (all of the necessary ingredients being already present in \cite{deshouillers1982kloosterman}). We leave its proof, which requires much additional notation, to \cref{sec:deshouillers-iwaniec}.

\section{The triple convolution estimate} \label{sec:triple-conv-estimate}

Here we state our main technical result, \cref{thm:triple-conv-estimate}, which concerns the distribution in arithmetic progressions of convolutions of three bounded sequences (we point the reader to similar results in \cite[Theorem 4]{bombieri1986primes}, \cite[Th\'eor\`eme 3]{drappeau2015theoremes}, \cite[Lemma 2.3]{drappeau2017smooth}, and \cite[Proposition 8.3]{maynard2025primes}). We then deduce \cref{thm:distribution} from \cref{thm:triple-conv-estimate}.

\begin{remark}
One can apply the most efficient convolution estimates directly to the setting of smooth numbers (and smooth-supported multiplicative functions), since these can essentially be factorized into any number of factors of pre-specified sizes. By contrast, in the case of primes, combinatorial decompositions of the von Mangoldt function produce more types of convolution sums, requiring different estimates for different ranges (typically organized into `Type I' and `Type II' information).
\end{remark}

To achieve a power-saving in \cref{thm:triple-conv-estimate}, appropriate for the application to smooth numbers, one needs a better approximation for indicator functions of the form $\one_{k \equiv 1 \mod{r}}$ than $\frac{1}{\varphi(r)}\one_{(k, r) = 1}$ (given $r \in \Z_+$ and $k \mod{r}$). Drappeau \cite{drappeau2015theoremes,drappeau2017smooth} noticed that since
\[
    \one_{k \equiv 1 \mod{r}} = \frac{1}{\varphi(r)} \sum_{\chi \pmod{r}} \chi(k) =
    \frac{\one_{(k, r) = 1}}{\varphi(r)}
    \sum_{\substack{\chi \text{ primitive} \\ \cond(\chi) \mid r}} \chi(k),
\]
one can instead consider the partial sum
\begin{equation} \label{eq:omega-d}
    \omega_D(k; r) := \sum_{\substack{\chi \in \mX_D \\ \cond(\chi) \mid r}} \chi(k),
    \qquad\quad
    \text{where}
    \qquad\quad
    \mX_D := \{\chi \text{ primitive} : \cond(\chi) \le D\},
\end{equation}
and work with the error term
\[
    \mE_D(k; r) := \one_{k \equiv 1 \pmod{r}} - \frac{\one_{(k,r) = 1}}{\varphi(r)} \omega_D(k; r)
    \ = 
    \frac{\one_{(k,r) = 1}}{\varphi(r)} \sum_{\substack{\chi \pmod{r} \\ \cond(\chi) > D}} \chi(k).
\]
One should then expect to obtain better bounds for $\mE_D(k; r)$ when $D$ is moderately large (i.e., a small power of $x$) than when $D = 1$ (and $\omega_1(k; r) = 1$). We also note the crude bound 
\[
    |\omega_D(k; r)| \le |\mX(D)| \le D^2,
\]
which may be used implicitly in our proofs.

\begin{theorem}[Triple convolution estimate] \label{thm:triple-conv-estimate}
For all sufficiently small $\eps > 0$, there exists $\delta > 0$ such that the following holds. Let $a_1, a_2$ be coprime nonzero integers, and let $M, N, L, R, x > 2$ satisfy
\begin{equation}\label{eq:conditions}
\begin{gathered}
    a_1, a_2 \ll x^\delta, 
    \qquad 
    MNL \asymp x,
    \qquad\quad 
    x^{(1-\eps)/2} \ll R \ll x^{-5\eps} NL \ll x^{(2/3) - 11\eps},
    \qquad 
    N^9 L^8 \ll x^3 R^4,
    \\
    N \ll \frac{x^{1-2\eps}}{R}, \qquad\quad N^4 L^7 \max(1, N/L)^{2\theta} \ll x^{2-17\eps}R^2, 
    \qquad 
    N^{12-6\theta} L^{11-6\theta} \ll x^{6-4\theta-5\eps} R^2,
\end{gathered}
\end{equation}
for $\theta = \theta_{\max}$.
Then for any $1$-bounded complex sequences $(\alpha_m)$, $(\beta_n)$, $(\gamma_\ell)$, one has
\begin{equation} \label{eq:triple-conv-estimate}
    \sum_{\substack{r \sim R \\ (r, a_1a_2) = 1}}
    \left\vert \sum_{m \sim M} \sum_{n \sim N} \sum_{\ell \sim L} \alpha_m \beta_n \gamma_\ell\,
    \mE_D(mn\ell \bar{a_1} a_2; r)
    \right\vert
    \ll_\eps  
    \frac{x (\log x)^4}{\min\left(x^\delta, \sqrt{D}\right)},
\end{equation}
for all $1 \le D \le x^\eps$.
\end{theorem}


\begin{remark}
Error terms of the form $O_\eps(x^{1-\delta})$ are dominated by the right-hand side of \cref{eq:triple-conv-estimate}, and will be available throughout most of our proof. If $x^{2\delta} \le D \le x^\eps$, then the right-hand side of \cref{eq:triple-conv-estimate} becomes $x^{1-\delta} (\log x)^4$, i.e., a power-saving; having an explicit dependence on the conductor bound $D$ is required for the application to smooth-supported multiplicative functions, as in \cite{drappeau2017smooth}.
\end{remark}

\begin{remark}
If one is free to choose the parameters $M$, $N$, $L$ subject only to the constraints in \cref{eq:conditions} and $MNL \asymp x$, then in order to maximize the range $R$, it is optimal to pick (up to $x^{o(1)}$ factors)
\begin{equation} \label{eq:optimal-parameters}
    R \approx x^{(5-4\theta)/(8-6\theta)},
    \qquad\quad
    M \approx \frac{x}{R},
    \qquad\quad 
    N \approx \frac{x}{R},
    \qquad\quad
    L \approx \frac{R^2}{x}.
\end{equation}
This improves on the conditions from Drappeau's triple convolution estimate \cite[(3.2)]{drappeau2015theoremes}, which can handle moduli up to $R \approx x^{3/5}$. 
\end{remark}

\begin{remark}
Although our (conditional) results hit a barrier at $R = x^{5/8-o(1)}$, a more essential limitation of triple convolution estimates proven via the dispersion method lies at $R \le x^{2/3 - o(1)}$, corresponding to the case of three equal parameters $M = N = L \asymp x^{1/3}$ in \cref{eq:optimal-parameters}. Indeed, the diagonal terms in the first Cauchy--Schwarz step require $R < NL$, and already our bounds for the second dispersion sum will use $NL < x^{2/3}$ (moreover, it is natural to Fourier-complete in the largest variable $m$, leaving $NL \le x^{2/3}$). We note again that \cref{thm:triple-conv-estimate} allows for the case of two equal parameters $M = N \approx x/R$, for any $x^{1/2-o(1)} \ll R \ll x^{(5-4\theta)/(8-6\theta)-o(1)}$; this is possible due to Maynard's deamplification argument \cite{maynard2025primes}. In particular, the ranges $M = N = x^{2/5}$, $L = x^{1/5}$, a limiting case in Drappeau's work \cite[Th\'eor\`eme 3]{drappeau2015theoremes}), are now admissible (this is analogous to the infamous case of convolving five sequences of equal sizes).
\end{remark}

Given \cref{thm:triple-conv-estimate} and \cref{lem:harper}, deducing \cref{thm:distribution} is now a routine modification of Drappeau's argument in \cite[Section 3.7]{drappeau2015theoremes} (we follow the same reasoning, using \cref{thm:triple-conv-estimate} instead of \cite[Th\'eor\`eme 3]{drappeau2015theoremes}, and with the choice of parameters in \cref{eq:optimal-parameters}).

\begin{proof}[Proof of \cref{thm:distribution} assuming \cref{thm:triple-conv-estimate}]
Let $\eps > 0$ be sufficiently small, $C$ be the maximum between $\eps^{-1}(1-\eps)^{-1}$ and the constant $C$ given by \cref{lem:harper}, $\delta$ be the minimum between $\eps/100$ and the $\delta$'s of \cref{lem:harper,thm:triple-conv-estimate}, and let $(\log x)^C \le y \le x^{1/C}$, $D := x^{\eps/10}$, $\theta := \theta_{\max}$. It suffices to show (up to a rescaling of $\eps$ at the end) that \cref{eq:exponent} holds for the range of moduli
\[
    r \le x^{\alpha}, \qquad\qquad {\alpha} := \frac{5-4\theta}{8-6\theta} - 1000 \eps,
\]
and we note that error terms of the form $O_\eps(x^{1-\delta})$ are acceptable in \cref{eq:exponent} (up to slightly modifying the value of $\delta$), due to the inequality $x^{1-\delta} \ll \Psi(x, y) y^{-\delta/2}$ (as in \cite[Section 3.7]{drappeau2015theoremes}). We may of course assume that $a_1$ and $a_2$ are relatively prime, by reducing any common factors.

We split the left-hand side of \cref{eq:exponent} into
\[
\begin{aligned}
    &\sum_{\substack{r \le x^{\alpha} \\ (r, a_1 a_2) = 1}} 
    \bigg\vert  \sum_{n \in S(x, y)} \left(\one_{n\bar{a_1}a_2 \equiv 1 \pmod{r}} - \frac{\one_{(n,r)=1}}{\varphi(r)} \right)
    \bigg\vert 
    \\
    &\le 
    \sum_{\substack{r \le x^{\alpha} \\ (r, a_1 a_2) = 1}} 
    \bigg\vert  \sum_{n \in S(x, y)} \mE_D(n\bar{a_1}a_2; r)
    \bigg\vert  
    +
    \sum_{\substack{r \le x^{\alpha} \\ (r, a_1 a_2) = 1}} 
    \bigg\vert  \sum_{n \in S(x, y)} \frac{\one_{(n,r)=1}}{\varphi(r)} \sum_{\substack{\chi \text{ prim.},\ \cond(\chi) \mid r \\ 1 < \cond(\chi) \le D}} \chi(n\bar{a_1}a_2)
    \bigg\vert.
\end{aligned}
\]
The second sum is at most
\[
    \sum_{\substack{r \le x^{\alpha} \\ (r, a_1 a_2) = 1}} \frac{1}{\varphi(r)} \sum_{\substack{\chi \pmod{r} \\ 1 < \cond(\chi) \le D}}
    \bigg\vert  \sum_{n \in S(x, y)} \chi(n)
    \bigg\vert,
\]
which is appropriately bounded by \cref{lem:harper} and the triangle inequality. It remains to bound the first sum. 

Recall that the range $n \in S(x, y)$ means $n \le x$ and $P^+(n) \le y$, where $P^+(n)$ denotes the greatest prime factor of $n$. We bound the contribution of $n \le x^{1-\eps}$ trivially, as in \cite[Section 3.7]{drappeau2015theoremes}:
\[
\begin{aligned}
    \sum_{\substack{r \le x^{\alpha} \\ (r, a_1 a_2) = 1}} 
    \bigg\vert  \sum_{\substack{n \le x^{1-\eps} \\ P^+(n) \le y}} \mE_D(n\bar{a_1}a_2; r)
    \bigg\vert 
    &\le
    \sum_{\substack{r \le x^{\alpha} \\ (r, a_1 a_2) = 1}} 
    \sum_{\substack{n \le x^{1-\eps}}} \left( \one_{a_2n \equiv a_1 \pmod{r}} + \frac{D^2}{\varphi(r)}\right)
    \\
    &\ll_\eps
    x^{\alpha} + \sum_{\substack{n \le x^{1-\eps} \\ a_2 n \neq a_1}} \tau(|a_2n - a_1|)
    +
    x^{\eps/2} \sum_{n \le x^{1-\eps}} 1 \ll_\eps x^{1-\eps/2},
\end{aligned}
\]
and put the other $n$'s into $O(\log x)$ dyadic ranges $n \sim X$, with $x^{1-\eps} \le X \ll x$. We also extend the range of $r$ in these sums to $r \le X^{{\alpha} + 10\eps}$, noting that $x^{\alpha} \le x^{(1-\eps)({\alpha} + 10\eps)}$. Putting $r$ into $O(\log x)$ dyadic ranges $r \sim R$, it remains to bound sums of the form
\begin{equation} \label{eq:sum-r-to-estimate}
    \sum_{\substack{r \sim R \\ (r, a_1 a_2) = 1}} 
    \bigg\vert \sum_{\substack{n \sim X \\ P^+(n) \le y}} \mE_D(n\bar{a_1}a_2; r)
    \bigg\vert,
\end{equation}
for $R \ll X^{{\alpha}+10\eps}$. The contribution of the Bombieri--Vinogradov range $R \ll X^{(1/2) - (\eps/10)}$ is handled by classical methods (e.g., using \cite[Theorem 17.4]{iwaniec2021analytic}; see \cite[Proof of Prop.\,2.4]{drappeau2017smooth}, \cite[Proof of Prop.\,2]{drappeau2015theoremes}). For any $R$ in the remaining range $X^{(1/2) - (\eps/10)} \le R \ll X^{{\alpha}+10\eps}$, we set
\begin{equation} \label{eq:parameters-2}
    M_0 := \frac{X^{1-10\eps}}{R},
    \qquad\quad
    N_0 := \frac{X^{1-10\eps}}{R},
    \qquad\quad
    L_0 := \frac{R^2}{X^{1-20\eps}},
\end{equation}
and factorize smooth numbers as in \cite[Lemme 7]{drappeau2015theoremes} (or \cite{fouvry1996repartition}) to rewrite the sum in \cref{eq:sum-r-to-estimate} as
\[
    \sum_{\substack{r \sim R \\ (r, a_1a_2) = 1}}
    \bigg\vert 
    \sum_{\substack{L_0 < \ell \le L_0 P^-(\ell) \\ P^+(\ell) \le y}}
    \sum_{\substack{M_0 < m \le M_0 P^-(m) \\ P^+(m) \le P^-(\ell)}}
    \sum_{\substack{mn\ell \sim X \\ P^+(n) \le P^-(m)}}
    \mE_D(mn\ell \bar{a_1}a_2; r)
    \bigg\vert,
\]
where $P^-(n)$ denotes the smallest prime factor of $n$.

We also put $m, n, \ell$ into $O((\log y)^3)$ dyadic ranges $m \sim M$, $n \sim N$, $\ell \sim L$, with $M \in [M_0, yM_0]$, $N \in [y^{-2}N_0, N_0]$, $L \in [L_0, yL_0]$, and $MNL \asymp X$. Recalling that $y \le x^{\eps(1-\eps)} \le X^\eps$, it is easily checked that for such $M, N, L$ and $X^{(1/2)-(\eps/10)} \le R \ll X^{{\alpha}+10\eps} = X^{(5-4\theta)/(8-6\theta) - 990\eps}$, and small enough $\eps$, the conditions in \cref{eq:conditions} are satisfied (with respect to $X$ instead of $x$). 

At this point, our sums are almost in the right form to apply the triple convolution estimate in \cref{thm:triple-conv-estimate}, except for a few joint constraints on the variables $m, n, \ell$ (these are $P^+(m) \le P^-(\ell)$, $P^+(n) \le P^-(m)$, respectively $mn\ell \sim X$). The last step of analytically separating these constraints is identical to that in \cite[Section 3.7]{drappeau2015theoremes}, except that in the end we apply \cref{thm:triple-conv-estimate} instead of \cite[Th\'eor\`eme 3]{drappeau2015theoremes}. Overall, the contribution of the range $X^{(1/2)-(\eps/10)} \le R \ll X^{{\alpha} + 10\eps}$ is $O_\eps\left( (\log x)^{O(1)} x^{1-\delta} \right)$, which is acceptable; this completes our proof.
\end{proof}

We only briefly note that the result for smooth-supported multiplicative functions in \cref{thm:distribution-functions} follows by an analogous modification to the arguments in \cite{drappeau2017smooth}, using the parameters in \cref{eq:parameters-2}, and \cref{thm:triple-conv-estimate} instead of \cite[Lemma 2.3]{drappeau2017smooth}. The main additional difficulty in \cite{drappeau2017smooth} lies in the contribution of the small-conductor characters, since \cref{lem:harper} is no longer applicable; as a replacement, Drappeau, Granville and Shao developed a large sieve inequality for smooth-supported sequences \cite[Theorem 5.1]{drappeau2017smooth}. (We also point the reader to the follow-up work of Shparlinski in \cite{shparlinski2018character}.)

\section{Dispersion and deamplification} \label{sec:setup}
Our goal for the rest of this paper is to prove \cref{thm:triple-conv-estimate}, proceeding by Linnik's dispersion method. For the reader following the outline in \cref{subsec:first-steps}, the exponential sum from \cref{eq:sketch-main-exponential} will ultimately arise in the first dispersion sum, after Poisson summation (see \cref{prop:simplification}).

Assume the setup of \cref{thm:triple-conv-estimate}. We may take $x$ larger than an absolute constant, since the conclusion of \cref{thm:triple-conv-estimate} is trivial otherwise, and $(\alpha_m)$, $(\beta_n)$, $(\gamma_\ell)$ to be supported on $m \sim M$, $n \sim N$, $\ell \sim L$, without loss of generality.
We first combine the sequences $\beta_n$ and $\gamma_\ell$ into one sequence
\begin{equation} \label{eq:u-def}
    u_k := \sum_{n\ell = k} \beta_n \gamma_\ell,
\end{equation}
supported in $(K, 4K]$ where $K := NL$, $|u_k| \le \tau(k) \ll_\eps x^{\eps/2}$, and $\sum_k |u_k| \ll K$. Denoting the left-hand side of \cref{eq:triple-conv-estimate} by $\Delta = \Delta_D(M,N,L,R)$, we can introduce coefficients $\rho_r$ of absolute value $1$, supported in $(R, 2R]$, to rewrite
\[
    \Delta
    =
    \sum_{(r,a_1a_2)=1} \rho_r \sum_{(m, r) = 1} \alpha_m \sum_{(k, r) = 1} u_k \left(\one_{mk \equiv a_1 \bar{a_2} \pmod{r}} - \frac{1}{\varphi(r)} \omega_D(mk \bar{a_1}a_2; r) \right),
\]
where we recall that $\omega_D$ was defined in \cref{eq:omega-d}. Normally, at this point we would apply Cauchy--Schwarz in the $r, m$ variables, but we first perform a ``deamplification'' step (following Maynard \cite{maynard2025primes} with minor modifications), as anticipated in \cref{subsec:completing-argument}. The idea is to split the inner sum according to the residue class of $k \pmod{re}$ for some $e \ge 1$, and then to average over a convenient set of $e \sim E$; this artificially introduces a new parameter $E$ (to be chosen later), which will help reduce the contribution of a certain diagonal sum by a small power of $x$ (at the expense of increasing the corresponding off-diagonal terms, which already had a power-saving bound). For now, we require that
\begin{equation} \label{eq:E-conditions}
    x^{4\eps} \le E \ll x^{-\eps} \frac{K}{R},
\end{equation}
which is compatible with \cref{eq:conditions} since $R \ll x^{-5\eps} NL$. For multiple reasons of convenience throughout our proof, we will actually average over the set
\begin{equation} \label{eq:averaging-set}
    \msE := \{e \sim E : e \text{ prime}\}.
\end{equation}

\begin{proposition}[Dispersion setup with deamplification] \label{prop:setup}
Let $\Phi$ be a smooth function satisfying
\begin{equation} \label{eq:phi-smooth}
    \one_{[1, 2]} \le \Phi \le \one_{(0.5, 3)},
\end{equation}
and $\|\Phi^{(j)}\|_\infty \ll_j 1$ for $j \ge 0$. Then for any $\eps > 0$ and $1 \gg \delta > 0$, under the parameter conditions in \cref{eq:conditions,eq:E-conditions}, one has
\begin{equation} \label{eq:setup}
    \Delta \ll_\eps x\sqrt{\frac{R \log x}{K^2} (\mS_1 - 2\Re \mS_2 + \mS_3)}
    + x^{1-\eps},
\end{equation}
where
\begin{equation} \label{eq:disp-sums}
\begin{aligned}
    \mS_1 &:= \sum_{e \in \msE} \sum_{(r,a_1a_2)=1} \Phi\left(\frac{r}{R}\right) \sum_{(m,r)=1} \frac{1}{M} \Phi\left(\frac{m}{M}\right) \sum_{\substack{(k_1k_2,re)=1 \\ k_1 \equiv k_2 \pmod{re} \\ k_i \equiv a_1\bar{a_2m} \pmod{r}}} u_{k_1} \bar{u_{k_2}},
    \\
    \mS_2 &:= \sum_{e \in \msE} \sum_{(r,a_1a_2)=1} \Phi\left(\frac{r}{R}\right) \frac{1}{\varphi(re)}  \sum_{\substack{(m,r)=1}} \frac{1}{M} \Phi\left(\frac{m}{M}\right)  \sum_{\substack{(k_1k_2,re)=1 \\ k_2 \equiv a_1\bar{a_2m} \pmod{r}}} u_{k_1} \omega_D(mk_1 \bar{a_1}a_2; r)\bar{u_{k_2}},
    \\
    \mS_3 &:= \sum_{e \in \msE} \sum_{ (r,a_1a_2)=1} \Phi\left(\frac{r}{R}\right) \frac{1}{\varphi(r)\varphi(re)} \sum_{(m,r)=1} \frac{1}{M} \Phi\left(\frac{m}{M}\right) \left\vert \sum_{(k,re)=1} u_k \omega_D(mk \bar{a_1}a_2; r)\right\vert^2.
\end{aligned}
\end{equation}
\end{proposition}

\begin{proof}
For a fixed prime $e \sim E$, we wish to eliminate the contribution to $\Delta$ of the terms $k$ with $(e, k) \neq 1$, i.e., $e \mid k$. This contribution is
\[
\begin{aligned}
    &\sum_{(r,a_1a_2)=1} \rho_r \sum_{(m, r) = 1} \alpha_m \sum_{\substack{k \in (K, 4K] \\ (k, r) = 1, e \mid k}} u_{k} \left(\one_{mk \equiv a_1 \bar{a_2} \pmod{r}} - \frac{1}{\varphi(r)} \omega_D(mk \bar{a_1}a_2; r) \right)
    \\
    &\ll_\eps 
    \sum_{\substack{r \sim R \\ (r,a_1a_2e)=1}} \sum_{\substack{s \in (MK, 8MK] \\ e \mid s}} x^{\eps/2} \left(\one_{s \equiv a_1 \bar{a_2} \pmod{r}} + \frac{1}{\varphi(r)} D^2 \right)
    \\
    &\ll_\eps
    x^{\eps} \sum_{r \sim R} \left(\frac{MK}{RE} + 1 + \left(\frac{MK}{E} + 1\right) \frac{D^2}{R}\right)
    \quad 
    \\
    &\ll
    x^{\eps} \left(\frac{xD^2}{E} + R + D^2\right)
    \qquad \ll 
    x^{\eps} \left(\frac{x^{1+2\eps}}{x^{4\eps}} + R + x^{2\eps}\right),
\end{aligned}
\]
which is $\ll x^{1-\eps}$ since $R \ll x^{(2/3)-11\eps}$ by \cref{eq:conditions}. It follows that for any $e \sim E$,
\[
    \Delta = \sum_{(r,a_1a_2)=1} \rho_r \sum_{(m,r)=1} \alpha_m \sum_{(k,re)=1} u_k \left(\one_{mk \equiv a_1 \bar{a_2} \pmod{r}} - \frac{1}{\varphi(r)} \omega_D(mk \bar{a_1}a_2; r) \right) + O_\eps(x^{1-\eps}).
\]
Now for fixed $m$ and $e$ we have
\[
    \sum_{(k,re)=1} u_k \one_{mk \equiv a_1 \bar{a_2} \pmod{r}}
    =
    \sum_{\substack{b \in (\Z/re\Z)^\times \\ b \equiv a_1\bar{a_2 m} \pmod{r}}}
    \sum_{(k,re)=1} u_k \one_{k \equiv b \pmod{re}},
\]
and there are precisely $\varphi(re)/\varphi(r)$ choices of $b$ in the summation; thus
\[
\begin{aligned}
    \Delta &= \sum_{(r,a_1a_2)=1} \rho_r \sum_{(m,r)=1} \alpha_m \sum_{\substack{b \in (\Z/re\Z)^\times \\ b \equiv a_1\bar{a_2 m} \pmod{r}}} \sum_{(k,re)=1} u_k \left(\one_{k \equiv b \pmod{re}} - \frac{\omega_D(mk \bar{a_1}a_2; r)}{\varphi(re)} \right) 
    \\
    &+ O_\eps(x^{1-\eps}).
\end{aligned}
\]
(For a complete deamplification setup, one could also try to split the term $\omega_D(mk\bar{a_1}a_2; r)$ according to the residue $b$ of $k \pmod{re}$, but we do not need to do this in our proof.)

We then average over $e$ in the set $\mE$ from \cref{eq:averaging-set},
which has size $|\msE| \asymp_\eps E/\log E$ (recalling that $E \gg x^{\eps}$ and $|a_2| \le x^\delta$). Thus up to an error of $O_\eps(x^{1-\eps})$, we can rewrite $\Delta$ as
\[
    \frac{1}{|\msE|} \sum_{e \in \msE} \sum_{(r,a_1a_2)=1} \rho_r \sum_{(m,r)=1} \alpha_m \sum_{\substack{b \in (\Z/re\Z)^\times \\ b \equiv a_1\bar{a_2 m} \pmod{r}}} \sum_{(k,re)=1} u_k \left(\one_{k \equiv b \pmod{re}} - \frac{\omega_D(mk \bar{a_1}a_2; r)}{\varphi(re)} \right).
\]
We now apply Cauchy--Schwarz in the $e,r,m,b$ variables, allowing us to eliminate the $\rho_r$ and $\alpha_m$ coefficients; using that $\varphi(re) \le \varphi(r) e$, this gives
\begin{equation} \label{eq:delta-squared}
\begin{aligned}
    |\Delta|^2 &\ll 
    \frac{1}{|\msE|^2} \left( 
    \sum_{e \in \msE} \sum_{(r,a_1a_2)=1} \sum_{(m,r)=1}  \sum_{\substack{b \in (\Z/re\Z)^\times \\ b \equiv a_1\bar{a_2 m} \pmod{r}}} |\rho_r|^2 |\alpha_m|^2
    \right) \Delta' 
    +
    O_\eps(x^{2(1-\eps)})
    \\ 
    &\ll
    \frac{1}{|\msE|^2} \left( 
    \sum_{e \in \msE} \sum_{\substack{r \sim R \\ (r,a_1a_2)=1}} \sum_{\substack{m \sim M \\ (m,r)=1}}  \frac{\varphi(re)}{\varphi(r)}
    \right) \Delta' 
    +
    O_\eps(x^{2(1-\eps)})
    \\
    &\ll_\eps 
    MR (\log E) \Delta' + x^{2(1-\eps)},
\end{aligned}
\end{equation}
where
\[
\begin{aligned}
    \Delta' &:= \sum_{e \in \msE} \sum_{\substack{r \sim R \\ (r,a_1a_2)=1}} \sum_{\substack{m \sim M \\ (m,r)=1}} \sum_{\substack{b \in (\Z/re\Z)^\times \\ b \equiv a_1\bar{a_2 m} \pmod{r}}} \left\vert \sum_{(k,re)=1} u_k \left(\one_{k \equiv b \pmod{re}} - \frac{\omega_D(mk \bar{a_1}a_2; r)}{\varphi(re)} \right)\right\vert^2.
\end{aligned}
\]
Anticipating a later application of Poisson summation, we bound the indicator functions $\one_{m \sim M}$ and $\one_{r \sim R}$ from above by $\Phi(m/M)$ and $\Phi(r/R)$. Then we expand the square and perform the $b$-summation to obtain
\begin{equation} \label{eq:delta-prime}
\begin{aligned}
    &\Delta' 
    \\
    &\le \sum_{e \in \msE} \sum_{\substack{(r,a_1a_2)=1}} \Phi\left(\frac{r}{R}\right) \sum_{\substack{(m,r)=1}} \Phi\left(\frac{m}{M}\right) \sum_{\substack{b \in (\Z/re\Z)^\times \\ b \equiv a_1\bar{a_2 m} \pmod{r}}} \left\vert \sum_{(k,re)=1} u_k \left(\one_{k \equiv b\ (re)} - \frac{\omega_D(mk \bar{a_1}a_2; r)}{\varphi(re)} \right)\right\vert^2
    \\
    &= M \left(\mS_1 - 2\Re \mS_2 + \mS_3\right).
\end{aligned}
\end{equation}
Combining \cref{eq:delta-squared} with \cref{eq:delta-prime} and recalling that $M \asymp x/K$, we recover \cref{eq:setup}.
\end{proof}

Since the error term of $O_\eps(x^{1-\eps})$ in \cref{prop:setup} is admissible for \cref{thm:triple-conv-estimate}, it remains to estimate the dispersion sums $\mS_1, \mS_2, \mS_3$.

\section{The main terms}
We note that except for the coefficients $\Phi\left(\frac{m}{M}\right)$, only the residue of $m$ modulo $r$ matters in the inner summations from \cref{eq:disp-sums}. Thus if we define
\begin{equation} \label{eq:Poisson-difference}
    \mE_r(c) := \left(\frac{1}{M} \sum_{m \equiv c \pmod{r}} \Phi\left(\frac{m}{M}\right) \right) - \frac{\hat{\Phi}(0)}{r},
\end{equation}
which can be estimated via the truncated Poisson summation in \cref{lem:poisson}, we can rewrite
\begin{equation} \label{eq:disp-sums-2}
\begin{aligned}
    \mS_1 &= \hat{\Phi}(0) X_1 + \sum_{e \in \msE} \sum_{(r,a_1a_2)=1} \Phi\left(\frac{r}{R}\right) \sum_{c \in (\Z/r\Z)^\times} \mE_r(c) \sum_{\substack{(k_1k_2,re)=1 \\ k_1 \equiv k_2 \pmod{re} \\ k_i \equiv a_1\bar{a_2c} \pmod{r}}} u_{k_1} \bar{u_{k_2}},
    \\
    \mS_2 &= \hat{\Phi}(0) X_2 + \sum_{e \in \msE} \sum_{(r,a_1a_2)=1} \Phi\left(\frac{r}{R}\right) \frac{1}{\varphi(re)}  \sum_{c \in (\Z/r\Z)^\times} \mE_r(c) \sum_{\substack{(k_1k_2,re)=1 \\ k_2 \equiv a_1\bar{a_2c} \pmod{r}}} u_{k_1} \omega_D(k_1 \bar{a_1}a_2c; r)\bar{u_{k_2}},
    \\
    \mS_3 &= \hat{\Phi}(0) X_3 + \sum_{e \in \msE} \sum_{(r,a_1a_2)=1} \Phi\left(\frac{r}{R}\right) \frac{1}{\varphi(r)\varphi(re)}  \sum_{c \in (\Z/r\Z)^\times} \mE_r(c) \left\vert \sum_{(k,re)=1} u_k \omega_D(k \bar{a_1}a_2c; r)\right\vert^2,
\end{aligned}
\end{equation}
where
\begin{equation} \label{eq:main-terms}
\begin{aligned}
    X_1 &:= \sum_{e \in \msE} \sum_{(r,a_1a_2)=1} \Phi\left(\frac{r}{R}\right) \frac{1}{r} \sum_{\substack{(k_1k_2,re)=1 \\ k_1 \equiv k_2 \pmod{re}}} u_{k_1} \bar{u_{k_2}},
    \\
    X_2 &:= \sum_{e \in \msE} \sum_{(r,a_1a_2)=1} \Phi\left(\frac{r}{R}\right)  \frac{1}{r\varphi(re)} \sum_{\substack{(k_1k_2,re)=1}} u_{k_1} \bar{u_{k_2}} \omega_D(k_1 \bar{k_2}; r),
    \\
    X_3 &:= \sum_{e \in \msE} \sum_{(r,a_1a_2)=1} \Phi\left(\frac{r}{R}\right) \frac{1}{r\varphi(r)\varphi(re)}  \sum_{c \in (\Z/r\Z)^\times} \left\vert \sum_{(k,re)=1} u_k \omega_D(kc; r)\right\vert^2.
\end{aligned}
\end{equation}
Intuitively, these main terms reflect what would happen if in the summations from \cref{eq:disp-sums}, the variable $m$ (weighted by $\Phi\left(\frac{m}{M}\right)$) were uniformly distributed modulo $r$. Thus for $j \in \{1, 2, 3\}$, $\hat{\Phi}(0)X_j$ is essentially the best approximation to $\mS_j$ which does not depend on $M$. We now bound the contribution to \cref{eq:setup} of $X_1 - 2\Re X_2 + X_3$, using the multiplicative large sieve as in \cite{drappeau2015theoremes,drappeau2017smooth}.

\begin{proposition}[Contribution of main terms] \label{prop:main-terms}
With the notation above, one has
\[
    0 \le X_1 - 2\Re X_2 + X_3 \ll \frac{K^2}{RD} (\log x)^6,
\]
under the parameter conditions in \cref{eq:conditions,eq:E-conditions}.
\end{proposition}

\begin{proof}
In analogy with \cref{eq:delta-prime}, we can write
\begin{equation} \label{eq:main-terms-expansion}
\begin{aligned}
    &X_1 - 2\Re X_2 + X_3 
    \\
    &=
    \sum_{e \in \msE} \sum_{(r, a_1a_2) = 1}  \Phi\left(\frac{r}{R}\right) \frac{1}{r} \sum_{c \in (\Z/r\Z)^\times} \sum_{\substack{b \in (\Z/re\Z)^\times \\ b \equiv c \pmod{r}}} \left\vert \sum_{(k,re)=1} u_k \left(\one_{k \equiv b \pmod{re}} - \frac{\omega_D(k\bar{c}; r)}{\varphi(re)} \right)\right\vert^2
    \\
    &\ll
    \frac{1}{R} \sum_{e \in \msE} \sum_{r} \Phi\left(\frac{r}{R}\right) \sum_{\substack{b \in (\Z/re\Z)^\times}} \left\vert \sum_{(k,re)=1} u_k \left(\one_{k \equiv b \pmod{re}} - \frac{\omega_D(k\bar{b}; r)}{\varphi(re)} \right)\right\vert^2
    \\
    &=
    \frac{1}{R} \sum_{e \in \msE} \sum_r \Phi\left(\frac{r}{R}\right) \frac{1}{\varphi(re)^2} \sum_{\substack{b \in (\Z/re\Z)^\times}} \left\vert \sum_{(k,re)=1} u_k \left(\omega_\infty(k\bar{b}; re) - \omega_D(k\bar{b}; r) \right)\right\vert^2,
\end{aligned}
\end{equation}
where the first equality shows that $X_1 - 2\Re X_2 + X_3 \ge 0$.
Note that
\[
\begin{aligned}
    \omega_\infty(k\bar{b}; re) - \omega_D(k\bar{b}; r)
    &=
    \left( \omega_\infty(k\bar{b}; re) - \omega_\infty(k\bar{b}; r) \right) + \left( \omega_\infty(k\bar{b}; r) - \omega_D(k\bar{b}; r) \right)
    \\
    &=
    \sum_{\substack{\chi \in S}} \chi(k)\bar{\chi(b)},
\end{aligned}
\]
where $S = S(r,e,\eps) = S_1 \cup S_2$ and
\[
\begin{aligned}
    S_1 &:= \left\{ \chi \text{ primitive} : \cond(\chi) \nmid r, \text{ but } \cond(\chi) \mid re \right\},
    \\
    S_2 &:=
    \left\{ \chi \text{ primitive} : \cond(\chi) > D \text{ and } \cond(\chi) \mid r \right\}.
\end{aligned}
\]
Since all the characters in $S$ are primitive, any distinct $\chi_1, \chi_2 \in S$ must induce different characters modulo $re$. Thus $\chi_1 \bar{\chi_2} \one_{(re, \cdot) = 1}$ is not the principal character modulo $re$, so it must have average $0$. But then
\begin{equation} \label{eq:inside-main-terms}
\begin{aligned}
    &\sum_{\substack{b \in (\Z/re\Z)^\times}} \left\vert \sum_{(k,re)=1} u_k \left(\omega_\infty(k\bar{b}; re) - \omega_D(k\bar{b}; r) \right)\right\vert^2
    \\
    &=
    \sum_{\substack{b \in (\Z/re\Z)^\times}} \left\vert \sum_{\chi \in S} \bar{\chi}(b) \sum_{(k,re)=1} u_k \chi(k) \right\vert^2
    \\
    &=
    \sum_{\chi_1,\chi_2 \in S} \left(\sum_{(k,re)=1} u_k \chi_1(k)\right)\bar{\left(\sum_{(k,re)=1} u_k \chi_2(k)\right)} \sum_{\substack{b \in (\Z/re\Z)^\times}} \bar{\chi_1}(b) \chi_2(b) 
    \\
    &=
    \varphi(re) \sum_{\chi \in S} \left\vert \sum_{(k,re)=1} u_k \chi(k)\right\vert^2.
\end{aligned}    
\end{equation}
From \cref{eq:main-terms-expansion}, \cref{eq:inside-main-terms}, and the fact that all characters $\chi \in S_1$ also have $\cond(\chi) > D$ (due to $D \le x^\eps < E \le e \mid \cond(\chi)$), we conclude that
\[
\begin{aligned}
    X_1 - 2\Re X_2 + X_3 
    \ &\ll  
    \frac{1}{R}\sum_{e \in \msE} \sum_r \Phi\left(\frac{r}{R}\right) \frac{1}{\varphi(re)}  \sum_{\chi \in S} \left\vert \sum_{(k,re)=1} u_k \chi(k)\right\vert^2
    \\
    &\le
    \frac{1}{R}\sum_{e \in \msE} \sum_r \Phi\left(\frac{r}{R}\right) \frac{1}{\varphi(re)}  \sum_{\substack{\chi \text{ primitive} \\ \cond(\chi) > D \\ \cond(\chi) \mid re}} \left\vert \sum_{(k,re)=1} u_k \chi(k)\right\vert^2.
\end{aligned}
\]
Now letting $Q := RE$, substituting $q$ for $re$, using that $q$ has $O(\log q)$ different prime factors, and decomposing $\one_{(a, b) = 1} = \sum_{d \mid (a, b)} \mu(d)$ to get rid of the coprimality restriction, we can bound the sum above by 
\[
\begin{aligned}
    X_1 - 2\Re X_2 + X_3 
    &\ll
    \frac{\log x}{R}\sum_{Q/2 \le q \le 6Q} \frac{1}{\varphi(q)} \sum_{\substack{\chi \text{ primitive} \\ \cond(\chi) > D \\ \cond(\chi) \mid q}} \left\vert \sum_{(k,q)=1} u_k \chi(k)\right\vert^2
    \\
    &\le
    \frac{\log x}{R}
    \sum_{D < s \le 6Q}
    \sum_{\substack{\chi \text{ primitive} \\ \cond(\chi) = s}}\,
    \sum_{\substack{q \le 6Q \\ s \mid q}}
    \frac{1}{\varphi(q)} \left\vert \sum_{(k,q/s)=1} u_k \chi(k)\right\vert^2
    \\
    &\le 
    \frac{\log x}{R}
    \sum_{D < s \le 6Q}
    \sum_{\substack{\chi \text{ primitive} \\ \cond(\chi) = s}}\,
    \sum_{\substack{q \le 6Q \\ s \mid q}}
    \frac{\tau(q/s)}{\varphi(q)} 
    \sum_{d \mid q/s} \left\vert \sum_{k'} u_{dk'} \chi(dk')\right\vert^2
    \\
    &\le
    \frac{\log x}{R}
    \sum_{d \le 6Q}
    \sum_{D < s \le 6Q}
    \sum_{\substack{\chi \text{ primitive} \\ \cond(\chi) = s}}\,
    \left\vert \sum_{k'} u_{dk'} \chi(k')\right\vert^2 
    \sum_{\substack{q \le 6Q \\ ds \mid q}}
    \frac{\tau(q/s)}{\varphi(q)}.
\end{aligned}
\]
Noting that
\[
    \sum_{\substack{q \le 6Q \\ sd \mid q}}
    \frac{\tau(q/s)}{\varphi(q)}
    \le
    \sum_{\substack{q' \le 6Q}}
    \frac{\tau(q'd)}{\varphi(q'ds)}
    \le 
    \frac{\tau(d)}{\varphi(d)\varphi(s)}
    \sum_{\substack{q' \le 6Q}}
    \frac{\tau(q')}{\varphi(q')}
    \ll 
    \frac{\tau(d)}{\varphi(d)\varphi(s)} (\log x)^2,
\]
we further have
\[
\begin{aligned}
    X_1 - 2\Re X_2 + X_3 
    &\ll 
    \frac{(\log x)^3}{R}
    \sum_{d \le 6Q}
    \frac{\tau(d)}{\varphi(d)}
    \sum_{D < s \le 6Q}
    \frac{1}{\varphi(s)}
    \sum_{\substack{\chi \text{ primitive} \\ \cond(\chi) = s}}
    \left\vert \sum_{k'} u_{dk'} \chi(k')\right\vert^2 
    \\
    &=
    \frac{(\log x)^3}{R}
    \sum_{d \le 6Q}
    \frac{\tau(d)}{\varphi(d)}
    \int_D^\infty
    \sum_{D < s \le \min(6Q, t)}
    \frac{s}{\varphi(s)}
    \sum_{\substack{\chi \text{ primitive} \\ \cond(\chi) = s}}
    \left\vert \sum_{k'} u_{dk'} \chi(k')\right\vert^2 \frac{dt}{t^2}.
\end{aligned} 
\]
Finally, applying the multiplicative large sieve from \cref{lem:large-sieve} as in \cite[(2.6)]{drappeau2017smooth}, we conclude that
\[ 
\begin{aligned}
    X_1 - 2 \Re X_2 + X_3
    &\ll 
    \frac{(\log x)^3}{R} \sum_{d \le 6Q} \frac{\tau(d)}{\varphi(d)} \int_D^\infty \left(\frac{K}{d} + \min(6Q, t)^2 \right) \sum_{k' \asymp K/d} |u_{dk'}|^2\ \frac{dt}{t^2}
    \\
    &\ll 
    \frac{(\log x)^3}{R} K (\log K)^3 \sum_{d \le 6Q} \frac{\tau(d)^3}{d\varphi(d)} \left( \frac{K}{dD} + Q \right)
    \\
    &\ll
    (\log x)^6 \frac{K}{R} \left(\frac{K}{D} + Q\right).
\end{aligned}
\]
Using the condition $RE \ll x^{-\eps}K$ from \cref{eq:E-conditions} to bound $Q = RE \ll K/D$, we conclude that
\[
    X_1 - 2 \Re X_2 + X_3 \ll 
    (\log x)^6 K^2 (RD)^{-1},
\] 
as we wanted.
\end{proof}

To bridge \cref{prop:setup,prop:main-terms}, it remains to compare the dispersion sums $\mS_j$ to their main terms $\hat{\Phi}(0) X_j$; we make the following claim.

\begin{proposition}[Estimates for dispersion sums] \label{prop:dispersion-sums}
For all sufficiently small $\eps > 0$, there exists $\delta > 0$ such that with the notation in \cref{eq:disp-sums-2}, the following hold:
\begin{itemize}
\item[(i).] Assuming the ranges in \cref{eq:conditions}, there exists a choice of $E$ satisfying \cref{eq:E-conditions} such that 
\begin{equation} \label{eq:S1-estimate}
    \mS_1 - \hat{\Phi}(0)X_1 \ll_\eps x^{-2\delta} \frac{K^2}{R}.
\end{equation}
\item[(ii).] Assuming both \cref{eq:conditions,eq:E-conditions}, one has
\begin{align} \label{eq:S2-estimate}
    \mS_2 - \hat{\Phi}(0)X_2 &\ll_\eps x^{-2\delta} \frac{K^2}{R}, \\
    \label{eq:S3-estimate}
    \mS_3 - \hat{\Phi}(0)X_3 &\ll_\eps x^{-2\delta} \frac{K^2}{R}.
\end{align}
\end{itemize}
\end{proposition} 

\begin{proof}[Proof of \cref{thm:triple-conv-estimate} assuming \cref{prop:dispersion-sums}]
The hypothesis of \cref{thm:triple-conv-estimate} assumes \cref{eq:conditions}, so we can pick $E$ as in \cref{prop:dispersion-sums}.$(i)$, subject to \cref{eq:E-conditions}. Then, by combining \cref{prop:setup,prop:dispersion-sums,prop:main-terms}, we obtain
\[
\begin{aligned}
    \Delta
    &\ll_\eps
    x\sqrt{\frac{R\log x}{K^2} \left(\hat{\Phi}(0)\frac{K^2}{RD}(\log x)^6 + x^{-2\delta} \frac{K^2}{R} \right)} + x^{1-\eps}
    \\
    &\ll
    x\sqrt{(\log x)^7 \left(\frac{1}{D} + \frac{1}{x^{2\delta}}\right)} + x^{1-\eps}
\end{aligned}
\]
The conclusion of \cref{thm:triple-conv-estimate} follows after replacing $\delta$ with $\min(\delta, \eps)$.
\end{proof} 

Our remaining task is to prove  \cref{prop:dispersion-sums}; the truncated Poisson expansion of the coefficients $\mE_r(c)$ from \cref{eq:disp-sums-2} will ultimately reduce our problem to that of bounding various exponential sums. We note that we have not chosen the value of $\delta$ in terms of $\eps$ yet; the condition $\delta \le \eps/2$ will suffice for estimating $\mS_2$ and $\mS_3$, but more will be needed for the (much more involved) study of $\mS_1$.

\section{The second and third dispersion sums}

Here we prove \cref{prop:dispersion-sums}.$(ii)$, adapting Drappeau's arguments in \cite[Sections 3.2 and 3.3]{drappeau2015theoremes}. We assume all the parameter conditions in \cref{eq:conditions} and \cref{eq:E-conditions}.

\begin{proof}[Proof of \cref{eq:S3-estimate}, estimating $\mS_3$]
Recall from \cref{eq:disp-sums-2} that
\[
\begin{aligned}
    &\mS_3 - \hat{\Phi}(0) X_3 
    \\
    &= \sum_{e \in \msE} \sum_{(r,a_1a_2)=1} \Phi\left(\frac{r}{R}\right) \frac{1}{\varphi(r)\varphi(re)}  \sum_{c \in (\Z/r\Z)^\times} \mE_r(c) \left\vert \sum_{(k,re)=1} u_k \sum_{\substack{\chi \in \mX_D \\ \cond(\chi) \mid r}} \chi(k \bar{a_1}a_2 c)\right\vert^2
    \\
    &=
    \sum_{e \in \msE} \sum_{(r,a_1a_2)=1} \Phi\left(\frac{r}{R}\right) \frac{1}{\varphi(r)\varphi(re)} 
    \sum_{\substack{\chi_1, \chi_2 \in \mX_D \\ \cond(\chi_i) \mid r}}
    \bar{\chi_1} \chi_2 (a_1 \bar{a_2})
    \sum_{\substack{k_1, k_2 \\ (k_1k_2, re) = 1}} 
    \chi_1(k_1) u_{k_1} \bar{\chi_2(k_2) u_{k_2}} \\
    &
    \qquad\qquad\qquad\qquad\qquad\qquad\qquad\qquad\qquad\qquad\qquad\qquad\quad
    \times 
    \sum_{c \in (\Z/r\Z)^\times} \chi_1 \bar{\chi_2}(c) \mE_r(c),
\end{aligned}
\]
where $\mE_r(c)$ is given by \cref{eq:Poisson-difference}. Expanding $\mE_r(c)$ according to \cref{lem:poisson} with $H := x^\eps R M^{-1}$, we obtain
\begin{equation} \label{eq:Gauss}
    \sum_{c \in (\Z/r\Z)^\times} \chi_1 \bar{\chi_2}(c) \mE_r(c)
    =
    \frac{1}{r} \sum_{0 < |h| < H} \hat{\Phi}\left(\frac{hM}{r}\right) \sum_{c \in (\Z/r\Z)^\times} \chi_1 \bar{\chi_2}(c)\ \e\left(\frac{ch}{r}\right)
    + O_\eps\left(x^{-99}\right).
\end{equation}
(In such manipulations, we warn the reader of the potential confusion between the integer variable $e \in \msE$ and the function $\e(\cdot)$; the difference should be clear from context.)

The inner sum (over $c$) in \cref{eq:Gauss} is a Gauss sum, which we can bound using \cref{lem:gauss} for the Dirichlet character $\chi_1 \bar{\chi_2} \one_{(\cdot, r) = 1} \pmod{r}$ (whose conductor divides $r$ and is at most equal to $D^2 \le x^{2\eps}$). This yields
\[
\begin{aligned}
    \sum_{c \in (\Z/r\Z)^\times} \chi_1 \bar{\chi_2}(c) \mE_r(c) 
    &\ll_\eps x^{-100} +
    \frac{1}{r} \sum_{0 < |h| < H} x^\eps \sum_{d \mid (h, r)} d
    \\
    &\ll x^{-100} +
    \frac{x^\eps}{r} \sum_{d \mid r} d \sum_{0 < |h| < H/d} 1
    \ll
    \frac{x^{\eps}}{R} \tau(r) H
    =
    x^{2\eps} \frac{\tau(r)}{M},
\end{aligned}
\]
which leads to
\[
\begin{aligned}
    \mS_3 - \hat{\Phi}(0) X_3
    \ll_\eps 
    x^{2\eps}
    \sum_{e \in \msE}
    \sum_{(r,a_1a_2)=1} 
    \Phi\left(\frac{r}{R}\right)
    \frac{\tau(r)}{\varphi(r)\varphi(re) M} 
 |\mX_D|^2 K^2
    &\ll 
    x^{6\eps} \frac{K^2}{M}
    \sum_{r} \Phi\left(\frac{r}{R}\right)
    \frac{\tau(r)}{\varphi(r)^2} 
    \\
    &\ll 
    x^{6\eps} \frac{K^2}{MR} (\log R)^{O(1)}.
\end{aligned}
\]
Since $x/M \asymp K \ll x^{(2/3)-6\eps}$ by \cref{eq:conditions}, we have $M \gg x^{1/3+6\eps} \gg x^{7\eps}$ for small enough $\eps$, and in particular $\mS_3 - \hat{\Phi}(0)X_3\ll_\eps x^{-\eps} K^2/R$, proving the easiest third of \cref{prop:dispersion-sums}.
\end{proof}

\begin{proof}[Proof of \cref{eq:S2-estimate}, estimating $\mS_2$]
Recall from \cref{eq:disp-sums-2} that
\[
\begin{aligned}
    &\mS_2 - \hat{\Phi}(0) X_2 
    \\
    &= 
    \sum_{e \in \msE} \sum_{(r,a_1a_2)=1} \Phi\left(\frac{r}{R}\right) \frac{1}{\varphi(re)} \sum_{c \in (\Z/r\Z)^\times} \mE_r(c) \sum_{\substack{(k_1k_2,re)=1 \\ k_2 \equiv a_1\bar{a_2c} \pmod{r}}} u_{k_1} \bar{u_{k_2}} \sum_{\substack{\chi \in \mX_D \\ \cond(\chi) \mid r}} \chi(k_1 \bar{a_1} a_2 c)
    \\
    &=
    \sum_{e \in \msE} 
    \sum_{\chi \in \mX_D}
    \sum_{(k_1 k_2, e) = 1} \chi(k_1) u_{k_1} \bar{\chi(k_2) u_{k_2}} 
    \sum_{\substack{(r, a_ik_i)=1 \\ \cond(\chi) \mid r}} \Phi\left(\frac{r}{R}\right) \frac{1}{\varphi(re)} \mE_r(a_1 \bar{a_2 k_2}).
\end{aligned}
\]
Applying \cref{lem:poisson} with $H := x^\eps R M^{-1}$ to expand $\mE_r(a_1\bar{a_2 k_2})$ (as given in \cref{eq:Poisson-difference}), we obtain
\[
\begin{aligned}
    &\mS_2 - \hat{\Phi}(0) X_2 
    \\
    &= 
    \sum_{e \in \msE}
    \sum_{\chi \in \mX_D}
    \sum_{(k_1 k_2, e) = 1} \chi(k_1) u_{k_1} \bar{\chi(k_2) u_{k_2}} 
    \sum_{\substack{(r, a_ik_i)=1 \\ \cond(\chi) \mid r}} \frac{1}{r\varphi(re)} \Phi\left(\frac{r}{R}\right) 
    \sum_{1 \le |h| \le H} \hat{\Phi}\left(\frac{hM}{r}\right) 
    \e \left(\frac{a_1 h \bar{a_2k_2}}{r}\right)
    \\
    &+ 
    O_\eps\left(x^{-90}\right),
\end{aligned}
\]
where we used that $\varphi(re) \gg \varphi(r) e$ as before (since $e$ is prime). The error term is acceptable, so let us focus on the main term in the right-hand side (denote this by $Y_2$). By \cref{lem:bezout}, we have
\[
    \frac{\bar{a_2 k_2}}{r} + \frac{\bar{r}}{a_2 k_2} 
    \equiv 
    \frac{1}{a_2 k_2 r} \quad \pmod{1},
\]
so that
\begin{equation} \label{eq:Y2}
    Y_2 \ll 
    \sum_{e \in \msE}
    \sum_{\chi \in \mX_D}
    \sum_{k_1, k_2} |u_{k_1} u_{k_2}|
    \sum_{1 \le |h| \le H}
    \left\vert 
    \sum_{\substack{(r, a_ik_i)=1 \\ \cond(\chi) \mid r}} \frac{1}{r\varphi(re)} \Phi\left(\frac{r}{R}\right) 
    \hat{\Phi}\left(\frac{hM}{r}\right) 
    \e \left(\frac{a_1 h}{a_2 k_2 r} - \frac{a_1 h\bar{r}}{a_2 k_2}\right) \right\vert.
\end{equation}
At this point we decompose
\[
    \frac{1}{\varphi(re)}
    =
    \frac{1}{\varphi(r)} \left(\frac{1}{e-1} - \frac{1}{e(e-1)} \one_{e \mid r} \right),
\]
aiming to apply the exponential sum bound in \cref{lem:incomplete-weil}. Fixing $e, a_i, k_i, h$, this lets us rewrite the sum over $r$ in the right-hand side of \cref{eq:Y2} as 
\[
    \frac{1}{e-1} Z_2(\cond(\chi)) - \frac{1}{e(e-1)} Z_2([\cond(\chi), e]),
\]
where
\[
    Z_2(\ell) := 
    \sum_{\substack{(r, a_ik_i)=1 \\ \ell \mid r}} u(r) \frac{r}{\varphi(r)}\
    \e \left(-\frac{a_1 h\bar{r}}{a_2k_2} \right)
\]
and
\[
    u(r) := \frac{1}{r^2} \Phi\left(\frac{r}{R}\right) \hat{\Phi}\left(\frac{hM}{r}\right) \e \left( \frac{a_1h}{a_2k_2 r} \right).
\]
Note that $u$ extends to a differentiable function of a real variable $\xi$, supported in $[R/2, 3R]$, and with derivative bounds
\[
\begin{aligned}
    |u'(\xi)| &\ll \frac{1}{R^3} + \frac{1}{R^2} \left(\frac{1}{R} + \frac{HM}{R^2} + \frac{|a_1| H}{KR^2} \right)
    \\
    &\ll 
    \frac{1}{R^2} \left(\frac{1}{R} + \frac{x^\eps R}{R^2} + \frac{|a_1| x^\eps R}{xR^2} \right)
    \quad 
    \ll 
    \frac{x^{2\eps}}{R^3}.
\end{aligned}
\]
in this region (we used that $H = x^\eps R M^{-1}$, $MK \asymp x$, and the very crude bound $|a_1| \ll x^{1+\eps}$). So we may use integration by parts to estimate $Z_2$; letting 
\[
    v_\ell(\xi) := \sum_{\substack{(r, a_ik_i) = 1 \\ \ell \mid r \\ r \le \xi}} \frac{r}{\varphi(r)}\ \e \left(- \frac{a_1 h\bar{r}}{a_2k_2} \right),
\]
which can be bounded via \cref{lem:incomplete-weil} (with $n = a_1 h$, $c = a_2k_2$, $m = r$ and $k = a_1k_1$), we obtain
\[
\begin{aligned}
    Z_2(\ell)
    =
    \int u(\xi) dv_\ell(\xi)
    = 
    - \int v_\ell(\xi) du(\xi)
    &\ll 
    \frac{x^{2\eps}}{R^2} \sup_{\xi \in [R/2, 3R]} |v_\ell(\xi)|
    \\
    &\ll_\eps
    \frac{x^{3\eps}}{R^2} \left((a_1 h, a_2 k_2)^{1/2}K^{1/2} + (a_1 h, a_2 k_2) \frac{R}{K} \right),
\end{aligned}
\]
uniformly in $\ell \ge 1$. Returning to \cref{eq:Y2} and summing over $h$ and $k_2$, the GCD terms contribute at most $O_\eps(x^{\eps})$ on average (since $(a_1, a_2) = 1$). Thus
\[
\begin{aligned}
    Y_2 \ll_\eps x^{4\eps}
    |\msE| |\chi_D| K^2 H  \frac{1}{E} \frac{1}{R^2} \left(\sqrt{K} + \frac{R}{K} \right)
    &\le 
    x^{7\eps} \frac{K^2}{MR} \left(\sqrt{K} + \frac{R}{K} \right)
    \\
    &\ll 
    x^{7\eps - 1} \left(\frac{K^{7/2}}{R} + K^2\right).
\end{aligned}
\]
By \cref{eq:conditions}, we have $K^{3/2} \ll x^{1-9\eps}$ and $R \ll x^{1-11\eps}$, so we get a final bound of $Y_2 \ll_\eps x^{-\eps} K^2 / R$. This completes our proof of \cref{prop:dispersion-sums}.$(ii)$.
\end{proof}

\section{The first dispersion sum}
Finally, we work towards establishing \cref{prop:dispersion-sums}.$(i)$ (for a suitable choice of $\delta$ in terms of $\eps$); the first part of this section is very similar to \cite[Section 3.4]{drappeau2015theoremes}. Recall from \cref{eq:disp-sums-2} that
\[
    \mS_1 - \hat{\Phi}(0) X_1 = 
    \sum_{e \in \msE} \sum_{(r,a_1a_2)=1} \Phi\left(\frac{r}{R}\right) \sum_{c \in (\Z/r\Z)^\times} \mE_r(c) \sum_{\substack{(k_1k_2,re)=1 \\ k_1 \equiv k_2 \pmod{re} \\ k_i \equiv a_1\bar{a_2c} \pmod{r}}} u_{k_1} \bar{u_{k_2}},
\]
where $\mE_r(c)$ is given by \cref{eq:Poisson-difference}. We wish to bound this by $O_\eps(x^{-2\delta} K^2 / R)$, as in \cref{eq:S1-estimate}.

We still aim to apply Poisson summation for the sums $\mE_r(c)$, and reduce our problem to bounding certain exponential sums. But due to issues that would arise later in manipulating these exponential sums, we first need to eliminate the contribution of certain `bad' pairs $(k_1, k_2)$, in terms of a small parameter $\eta$ (to be chosen later in terms of $\eps$, as an intermediary step to choosing $\delta$). 

\begin{proposition}[Eliminating bad index pairs] \label{prop:bad-pairs}
For $\eps \ge \eta > 0$, under the parameter conditions in \cref{eq:conditions,eq:E-conditions}, one has
\[
\begin{aligned}
    \mS_1 - \hat{\Phi}(0) X_1 
    =
    \sum_{e \in \msE} \sum_{(r,a_1a_2)=1} \Phi\left(\frac{r}{R}\right) \sum_{c \in (\Z/r\Z)^\times} \mE_r(c) \sum_{\substack{(k_1, k_2) \in \mK(\eta) \\ (k_1k_2, re) = 1 \\ k_1 \equiv k_2 \pmod{re} \\ k_i \equiv a_1 \bar{a_2 c} \pmod{r}}} u_{k_1} \bar{u_{k_2}} 
    \
    +\
    O_\eta\left(x^{-\eta/4} \frac{K^2}{R} \right),
\end{aligned}
\]
where
\begin{equation} \label{eq:bad-pairs}
    \mK(\eta) := \left\{ 
    (k_1, k_2) \in \N^2\  \Big\vert\  
    \twolinemath{(k_1, (a_2k_2)^\infty) \le x^\eta,\ (k_2, (a_2k_1)^\infty) \le x^\eta,}
    {
    (k_1 - k_2, (a_2k_1k_2)^\infty) \le x^\eta,\ |k_1 - k_2| > K/x^\eta
    }
    \right\}.
\end{equation}
\end{proposition}

\begin{proof}
We eliminate the contribution of several sets of pairs $(k_1, k_2)$, putting absolute values on all the coefficients involved; thus it does not matter if some of the `eliminated sets' have nonempty intersections.
First, we consider the almost-diagonal pairs with $|k_1 - k_2| \le K/x^\eta$; using that $\sum_{c \in (\Z/r\Z)^\times} |\mE_r(c)| \le \frac{1}{M} \sum_m \Phi\left(\frac{m}{M}\right) + \hat{\Phi}(0) \ll 1$ and $x^{\eta} RE \ll K$ (which follows from \cref{eq:E-conditions} and $\eta \le \eps$), these contribute to $\mS_1 - \hat{\Phi}(0)X_1$ at most
\[
\begin{aligned}
    &\ll \sum_{e \in \msE} \sum_{(r,a_1a_2)=1} \Phi\left(\frac{r}{R}\right) \sum_{c \in (\Z/r\Z)^\times} |\mE_r(c)| \sum_{\substack{k_1 \equiv a_1 \bar{a_2c} \pmod{r} \\ (k_1, e) = 1}} |u_{k_1}| \sum_{\substack{k_2 \equiv k_1 \pmod{re} \\ |k_1 - k_2| \le K/x^\eta}}  |u_{k_2}|
    \\
    &\ll_\eta x^{\eta/2} E R \left(\frac{K}{R} + 1\right) \left(\frac{K}{x^\eta RE} + 1\right)
    \\
    &\ll_\eta x^{\eta/2}ER \frac{K}{R} \frac{K}{x^\eta RE}
    \\
    &\asymp 
    x^{-\eta/2} \frac{K^2}{R}.
\end{aligned}
\]
Then, we consider those pairs with $v := (k_1, k_2) > x^{\eta/2}$. Their contribution to $\mS_1 - \hat{\Phi}(0)X_1$ is at most
\[
    \ll 
    \sum_{e \in \msE} \sum_{(r,a_1a_2)=1} \Phi\left(\frac{r}{R}\right) \sum_{c \in (\Z/r\Z)^\times} |\mE_r(c)|
    \sum_{\substack{b \in (\Z/re\Z)^\times \\ b \equiv a_1\bar{a_2c} \pmod{r}}} 
    \sum_{\substack{v > x^{\eta/2} \\ (v,re)=1}}
    \left(\sum_{\substack{k \equiv b \pmod{re} \\ v \mid k}} |u_k| \right)^2.
\]
Using that $(v, re) = 1$, we can bound one inner sum over $k$ by $\ll_\eta x^{\eta/8}(K(vRE)^{-1} + 1) \ll x^{-3\eta/8}K(RE)^{-1}$ (recall that $x^{\eta} RE \ll K$ by \cref{eq:E-conditions}). This yields a total contribution of
\[
\begin{aligned}
    &\ll_\eta
    \frac{K}{REx^{3\eta/8}}
    \sum_{e \in \msE} \sum_{(r,a_1a_2)=1} \Phi\left(\frac{r}{R}\right) \sum_{c \in (\Z/r\Z)^\times} |\mE_r(c)|
    \sum_{\substack{v > x^\eta \\ (v,re)=1}}
    \sum_{\substack{k \equiv a_1\bar{a_2c} \pmod{r} \\ v \mid k}} |u_k| 
    \\
    &\le
    \frac{K}{Rx^{3\eta/8}}
    \sum_{(r,a_1a_2)=1} \Phi\left(\frac{r}{R}\right) \sum_{c \in (\Z/r\Z)^\times} |\mE_r(c)|
    \sum_{\substack{k \equiv a_1\bar{a_2c} \pmod{r}}} 
    \tau(k) |u_k| 
    \\
    &\ll_\eta
    \frac{K}{Rx^{2\eta/8}} R
    \left(\frac{K}{R} + 1\right)
    \\
    &\ll x^{-\eta/4} \frac{K^2}{R},
\end{aligned}
\]
which is also acceptable. Keeping the notation $v = (k_1, k_2)$, which we may now assume is at most $x^{\eta/2}$, note that 
\[
    d_1 := (k_1, (a_2k_2)^\infty) = (k_1, (a_2 v)^\infty),
\]
and let us consider those pairs $(k_1, k_2)$ with $d_1 > x^\eta$. Using that $x^{\eta/2} RE \ll K$ and swapping sums, these contribute at most
\[
\begin{aligned}
    &
    \sum_{v \le x^{\eta/2}}
    \sum_{e \in \msE} \sum_{(r,a_1a_2)=1} \Phi\left(\frac{r}{R}\right) \sum_{(m, r) = 1} \frac{1}{M} \Phi\left(\frac{m}{M}\right) 
    \sum_{\substack{d_1 \mid (a_2v)^\infty \\ d_1 > x^\eta}}
    \sum_{\substack{k_1 \equiv a_1 \bar{a_2m} \pmod{r} \\ (k_1, re) = 1,\ d_1 \mid k_1}} |u_{k_1}| \sum_{\substack{k_2 \equiv k_1 \pmod{re} \\ (k_1, k_2) = v}}  |u_{k_2}|
    \\
    &\ll_\eta 
    x^{\eta/8} \sum_{v \le x^{\eta/2}}
    \sum_{e \in \msE} \sum_{(r,a_1a_2)=1} \Phi\left(\frac{r}{R}\right) \sum_{(m, r) = 1} \frac{1}{M} \Phi\left(\frac{m}{M}\right) 
    \sum_{\substack{d_1 \mid (a_2v)^\infty \\ d_1 > x^\eta}} 
    \sum_{\substack{k_1 \in (K, 4K] \\ (k_1, re) = 1,\ d_1 \mid k_1 \\ r \mid a_2 m k_1 - a_1}} \left(\frac{K}{REv} + 1\right)
    \\
    &\ll 
    x^{\eta/8} \frac{K}{RE} \sum_{v \le x^{\eta/2}} \frac{1}{v}
    \sum_{e \in \msE} \sum_{m} \frac{1}{M} \Phi\left(\frac{m}{M}\right) 
    \sum_{\substack{d_1 \mid (a_2v)^\infty \\ d_1 > x^\eta}} 
    \sum_{\substack{k_1 \in (K, 4K] \\ d_1 \mid k_1}} 
    \sum_{\substack{(r,m a_1a_2)=1 \\ r \mid a_2 m k_1 - a_1}} \Phi\left(\frac{r}{R}\right).
\end{aligned} 
\]
Considering the cases $a_2 m k_1 = a_1$ and $a_2 m k_1 - a_1 \neq 0$ separately, and using $R \ll x^{1-\eta/2} \asymp KM x^{-\eta/2}$ (by \cref{eq:conditions}), this is further bounded by
\[
\begin{aligned}
    &\ll 
    x^{\eta/8} \frac{K}{RE} \sum_{v \le x^{\eta/2}} \frac{1}{v}
    \sum_{e \in \msE} \left(\tau(|a_1|)^3 \frac{R}{M} + \sum_{m} \frac{1}{M} \Phi\left(\frac{m}{M}\right) 
    \sum_{\substack{d_1 \mid (a_2v)^\infty \\ x^\eta < d_1 \le 4K}} 
    \sum_{\substack{k_1 \in (K, 4K] \\ d_1 \mid k_1 \\ a_2 m k_1 \neq a_1}} 
    \tau(|a_2 m k_1 - a_1|) \right)
    \\
    &\ll_\eta
    x^{\eta/4}\frac{K}{M} + x^{\eta/4} \frac{K}{RE} \max_{v \le x^{\eta/2}}
    \sum_{e \in \msE} \sum_{m} \frac{1}{M} \Phi\left(\frac{m}{M}\right) 
    \sum_{\substack{d_1 \mid (a_2v)^\infty \\ x^\eta < d_1 \le 4K}} 
    \left(\frac{K}{d_1} + 1\right)
    \\
    &\ll 
    x^{-\eta/4}\frac{K^2}{R} +
    x^{\eta/4} \frac{K}{R} \max_{v \le x^{\eta/2}} \sum_{\substack{d_1 \mid (a_2 v)^\infty \\ x^\eta < d}} \frac{K}{d_1}.
\end{aligned}
\]
Now since the number of distinct prime factors of a positive integer $b$ is $O(\log b / \log \log b)$, for $b \ll x$ we have the majorization
\begin{equation} \label{eq:harmonic-divisors}
\begin{aligned}
    \sum_{\substack{d \mid b^\infty \\ x^\eta < d}} d^{-1} 
    &\le
    x^{-2\eta/3} \sum_{d \mid b^\infty} d^{-1/3}
    \\
    &=
    x^{-2\eta/3} \prod_{\text{prime } p \mid b} \frac{1}{1 - p^{-1/3}}
    \ll
    x^{-2\eta/3} O(1)^{O(\log b / \log \log b)}
    \ll_\eta 
    x^{-\eta/2}.
\end{aligned}
\end{equation}
Using this, we find that the previous sum contributes an acceptable $O_\eta(x^{-\eta/4} K^2/R)$. 

The contribution of the pairs with $d_1 > x^\eta$ to $\hat{\Phi}(0)X_1$ is simpler and similarly bounded, and the contribution of the pairs with $d_2 := (k_2, (a_2k_1)^\infty) = (k_2, (a_2v)^\infty)$ is bounded symmetrically. All that is left is to eliminate the contribution of the pairs with large values of $(k_1 - k_2, (a_2k_1k_2)^\infty)$; since $(k_1 - k_2, k_1k_2) = (k_1 - k_2, v^2)$, we have
\[
    d_\Delta := (k_1 - k_2, (a_2k_1k_2)^\infty)
    =
    (k_1 - k_2, (a_2v)^\infty).
\]
Using that $Rx^{\eta/2} \ll K$, the pairs with $d_\Delta > x^\eta$ (as well as $v \le x^{\eta/2}$ and $|k_1 - k_2| > K/x^\eta$) contribute to $\mS_1$ at most
\[
\begin{aligned}
    &\ll 
    \sum_{v \le x^{\eta/2}}
    \sum_{e \in \msE} \sum_{(r,a_1a_2)=1} \Phi\left(\frac{r}{R}\right) \sum_{(m,r)=1} \frac{1}{M} \Phi\left(\frac{m}{M}\right)
    \sum_{\substack{d_\Delta \mid (a_2v)^\infty \\ d_\Delta > x^\eta}} 
    \sum_{\substack{k_1 \equiv a_1 \bar{a_2m} \pmod{r} \\ (k_1, re) = 1,\ (k_1, k) = v \\ 0 < |k| \le 8K \\ [d_\Delta, re] \mid k}} |u_{k_1}u_{(k_1+k)}|
    \\
    &\ll_\eta x^{\eta/8}
    \sum_{v \le x^{\eta/2}}
    \sum_{e \in \msE} \sum_{(r,a_1a_2)=1} \Phi\left(\frac{r}{R}\right) \sum_{(m,r)=1} \frac{1}{M} \Phi\left(\frac{m}{M}\right)
    \sum_{\substack{d_\Delta \mid (a_2v)^\infty \\ d_\Delta > x^\eta}} 
    \sum_{\substack{0 < |k| \le 8K \\ [d_\Delta, re] \mid k}} \sum_{\substack{k_1 \in (k, 4K] \\ k_1 \equiv a_1 \bar{a_2m} \pmod{r} \\ (k_1, re) = 1,\ (k_1, k) = v}} 1
    \\
    &\ll 
    x^{\eta/8}
    \sum_{v \le x^{\eta/2}}
    \sum_{e \in \msE} \sum_{(r,a_1a_2)=1} \Phi\left(\frac{r}{R}\right) \sum_{(m,r)=1} \frac{1}{M} \Phi\left(\frac{m}{M}\right)
    \sum_{\substack{d_\Delta \mid (a_2v)^\infty \\ d_\Delta > x^\eta}} 
    \sum_{\substack{0 < |k| \le 8K \\ [d_\Delta, re] \mid k}} \left(\frac{K}{Rv} + 1\right)
    \\
    &\ll 
    x^{\eta/8} \frac{K}{R}
    \sum_{v \le x^{\eta/2}} \frac{1}{v}
    \sum_{m} \frac{1}{M} \Phi\left(\frac{m}{M}\right)
    \sum_{\substack{d_\Delta \mid (a_2v)^\infty \\ d_\Delta > x^\eta}} 
    \sum_{\substack{0 < |k| \le 8K \\ d_\Delta \mid k}} \sum_{\substack{e \in \msE \\ (r,a_1a_2)=1 \\ re \mid k}} \Phi\left(\frac{r}{R}\right),
\end{aligned}
\]
Bounding the inner sum by $\tau(k)^2 \ll_\eta x^{\eta/9}$, this further becomes
\[
\begin{aligned}
    &\ll_\eta 
    x^{\eta/4} \frac{K}{R}
    \max_{v \le x^{\eta/2}}
    \sum_{m} \frac{1}{M} \Phi\left(\frac{m}{M}\right)
    \sum_{\substack{d_\Delta \mid (a_2v)^\infty \\ x^\eta < d_\Delta \le 8K}} 
    \sum_{\substack{|k| \le 8K \\ d_\Delta \mid k}} 1
    \\
    &\ll 
    x^{\eta/4} \frac{K}{R}
    \max_{v \le x^{\eta/2}}
    \sum_{\substack{d_\Delta \mid (a_2v)^\infty \\ x^\eta < d_\Delta \le 8K}} 
    \left(\frac{K}{d_\Delta} + 1 \right)
    \\
    &\ll 
    x^{\eta/4} \frac{K}{R} \max_{v \le x^{\eta/2}} \sum_{\substack{d_\Delta \mid (a_2 v)^\infty \\ x^\eta < d_\Delta \le 8K}} \frac{K}{d_\Delta}.
\end{aligned}
\]
Using the majorization from \cref{eq:harmonic-divisors}, this is again $O_\eta(x^{-\eta/4} K^2/R)$. 

The contribution of the pairs with $d_\Delta > x^\eta$ to $\hat{\Phi}(0) X_1$ is simpler and similarly bounded by $O_\eta(x^{-\eta/4} K^2/R)$. Having eliminated the absolute contribution of all pairs in $\mK(\eta)$ at least once, while incurring only admissible errors, we can conclude our proof of \cref{prop:bad-pairs}.
\end{proof}

We can now apply Poisson summation to prove the following.

\begin{proposition}[Reduction to exponential sum] \label{prop:simplification}
For $\eps \ge \eta > 0$ and $H := x^\eta R M^{-1}$, under the parameter conditions in \cref{eq:conditions,eq:E-conditions}, one has
\[
\begin{aligned}
    \mS_1 - \hat{\Phi}(0) X_1 
    =
    \sum_{e \in \msE} \sum_{(r, a_1a_2) = 1} \Phi\left(\frac{r}{R}\right) \frac{1}{r}
    \sum_{\substack{(k_1, k_2) \in \mK(\eta) \\ (k_1 k_2, re) = 1 \\ k_1 \equiv k_2 \pmod{re}}} u_{k_1} \bar{u_{k_2}} 
    \sum_{1 \le |h| \le H} \hat{\Phi}\left(\frac{hM}{r}\right) \e\left(a_1 h \frac{\bar{a_2k_1}}{r} \right)
    \\
    + O_\eta\left(x^{-\eta/4} \frac{K^2}{R} \right).
\end{aligned}
\]
\end{proposition}
\begin{proof}
Rewrite the sum in \cref{prop:bad-pairs} as
\[
    \sum_{e \in \msE} \sum_{(r,a_1a_2)=1} \Phi\left(\frac{r}{R}\right) \sum_{\substack{(k_1, k_2) \in \mK(\eta) \\ (k_1k_2, re) = 1 \\ k_1 \equiv k_2 \pmod{re}}} u_{k_1} \bar{u_{k_2}}\ \mE_r(a_1 \bar{a_2 k_1}),
\]
and apply \cref{lem:poisson} to expand $\mE_r(a_1 \bar{a_2 k_1})$. The resulting main term is precisely the sum in \cref{prop:simplification}, while the error terms are acceptable.
\end{proof}

\begin{remark}
The trivial bound for the right-hand side of \cref{prop:simplification} is $H$ times worse than for the right-hand side of \cref{prop:bad-pairs}, due to the additional sum over $h$. This is relevant because $H$ is a nontrivial power of $x$ for the choice of parameters in \cref{eq:optimal-parameters} (where $H \approx R M^{-1} \approx R^2/x$), since we are working with moduli $r$ well-beyond the $\sqrt{x}$ barrier. This is why we needed to eliminate the bad pairs $(k_1, k_2)$ (via \cref{prop:bad-pairs}) \emph{before} applying Poisson summation.
\end{remark}

We now go through a series of fairly technical manipulations, following \cite{drappeau2015theoremes} and \cite{maynard2025primes}, to reduce the sum in \cref{prop:simplification} to a variation of the exponential sum considered by Bombieri, Friedlander and Iwaniec in \cite[Section 10]{bombieri1987primes2}; the goal is to prove the remaining dispersion estimate for $\mS_1$ in \cref{prop:dispersion-sums}. We do this in two steps (after the statements of \cref{prop:drappeau-expo,prop:bfi-expo}); first, we assume the following exponential sum bound, which can be compared to Drappeau's \cite[Proposition 1]{drappeau2015theoremes}.

\begin{proposition}[Improved Drappeau-style exponential sum bound] \label{prop:drappeau-expo}
For all sufficiently small $\eps > 0$ and all $\eta \in (0, 1)$, under the parameter conditions in \cref{eq:conditions}, there exists $E$ satisfying \cref{eq:E-conditions} (with $K := NL$) such that the following holds. For any nonzero integer $a \ll x^{O(\eta)}$ and positive integers $b, d_a, d_1, d_2, d_\Delta, v, \delta_1, \delta_2 \ll x^{O(\eta)}$ with $d_2 = \delta_1 \delta_2$ and 
\[
    d_a \mid a, \qquad\quad 
    (d_1, d_2) = (d_1, d_\Delta) = (d_2, d_\Delta) = v, \qquad\quad 
    d_1d_2d_\Delta \mid b^\infty,
\]
one has
\[
\begin{aligned}
    \sum_{\substack{e \sim E \text{ prime} \\ (e, b) = 1}} \sum_{\substack{(r, b) = 1 \\ r \equiv 0 \pmod{d_a}}} \Phi\left(\frac{r}{R}\right) \frac{1}{r} 
    \sum_{\substack{(k, n\ell) = 1 \\ (k n \ell, reb) = 1 \\ d_1k - d_2n \ell = red_\Delta t \\ (t, b) = 1,\  red_\Delta |t| > K/x^\eta}} u'_k \beta'_n \gamma'_\ell
    \sum_{1 \le |h| \le H} \hat{\Phi}\left(\frac{hM}{r}\right) \e\left(a h \frac{\bar{bk}}{r} \right)
    \\
    \ll_{\eps, \eta} x^{O(\eta)-\eps/4} \frac{K^2}{R},
\end{aligned}
\]
where $H := x^\eta R M^{-1}$, and $|u'_k| \le \tau(d_1 k)$, $|\beta'_n| \le 1$, $|\gamma'_\ell| \le 1$ are sequences supported in $k \asymp K/d_1$, $n \sim N/\delta_1$, and $\ell \sim L/\delta_2$.
\end{proposition}

\begin{remark}
The exponential sum from \cref{prop:drappeau-expo} is essentially that anticipated in \cref{eq:sketch-main-exponential}.
\end{remark}

\begin{proof}[Proof of \cref{prop:dispersion-sums}.$(i)$, assuming \cref{prop:drappeau-expo}]
Let $\eps \in (0, 1)$ be sufficiently small, and let us pick $E$ as in \cref{prop:drappeau-expo}. By \cref{prop:simplification}, it remains to establish the bound
\[
\begin{aligned}
    \mD_1 := \sum_{e \in \msE} \sum_{(r, a_1a_2) = 1} \Phi\left(\frac{r}{R}\right) \frac{1}{r}
    \sum_{\substack{(k_1, k_2) \in \mK(\eta) \\ (k_1 k_2, re) = 1 \\ k_1 \equiv k_2 \pmod{re}}} u_{k_1} \bar{u_{k_2}} 
    \sum_{1 \le |h| \le H} \hat{\Phi}\left(\frac{hM}{r}\right) \e\left(a_1 h \frac{\bar{a_2k_1}}{r} \right)
    \\
    \ll_{\eps,\eta}
    x^{-2\delta} \frac{K^2}{R},
\end{aligned}
\]
for some choice of $0 < 8\delta \le \eta \le \eps$ in terms of $\eps$ (since $\delta \le \eta/8$, the error term of $x^{-\eta/4} K^2 / R$ from \cref{prop:simplification} is acceptable). For now, let us fix $\delta$ and $\eta$ such that $8\delta \le \eta$; we will give explicit choices at the end of this proof.

By the definition of $\mK(\eta)$ from \cref{eq:bad-pairs}, we may consider the $x^{\eta}$-bounded variables 
\[
    d_1 := (k_1, (a_2 k_2)^\infty), \qquad d_2 := (k_2, (a_2 k_1)^\infty), \qquad d_\Delta := (k_1 - k_2, (a_2 k_1 k_2)^\infty),
\]
\begin{equation} \label{eq:v}
    v := (k_1, k_2) = (d_1, d_2) = (d_1, d_\Delta) = (d_2, d_\Delta).
\end{equation}
Noting that $d_1$, $d_2$ and $d_\Delta$ all divide $(a_2 v)^\infty$, we may then expand
\begin{equation} \label{eq:D12}
    \mD_1 = \sum_{\substack{d_1, d_2, d_\Delta \le x^{\eta} \\ v = (d_1, d_2) = (d_1, d_\Delta) = (d_2, d_\Delta) \\ d_1 d_2 d_\Delta \mid (a_2v)^\infty}} \mD_2(d_1, d_2, d_\Delta, v),
\end{equation}
where
\[
    \mD_2 := \sum_{e \in \msE} \sum_{(r, a_1 a_2) = 1} \Phi\left(\frac{r}{R}\right) \frac{1}{r}
    \sum_{\substack{(k_1, (a_2 k_2)^\infty) = d_1 \\ (k_2, (a_2 k_1)^\infty) = d_2 \\ (k_1 - k_2, (a_2 k_1 k_2)^\infty) = d_\Delta \\ |k_1 - k_2| > K/x^\eta \\ (k_1 k_2, re) = 1 \\ k_1 \equiv k_2 \pmod{re}}}
    u_{k_1} \bar{u_{k_2}}
    \sum_{1 \le |h| \le H} \hat{\Phi}\left(\frac{hM}{r}\right) \e\left(a_1 h \frac{\bar{a_2 k_1}}{r} \right).
\]
Changing variables $k_i \mapsto d_i k_i$ and adjusting coprimality conditions accordingly (e.g., we now have $(k_1, a_2d_2k_2) = 1$ as well as $(k_1, d_1) = 1$, and $(d_1k_1d_2k_2, re) = 1$), we get
\[
\begin{aligned}
    &\mD_2 
    \\
    &= 
    \sum_{e \in \msE} \sum_{(r, a_1 a_2) = 1} \Phi\left(\frac{r}{R}\right) \frac{1}{r}
    \sum_{\substack{(k_1, a_2 d_1 d_2 k_2) = 1 \\ (k_2, a_2 d_1 d_2 k_1) = 1 \\ (d_1 k_1 - d_2 k_2, (a_2 d_1 d_2 k_1 k_2)^\infty) = d_\Delta \\ |d_1 k_1 - d_2 k_2| > K/x^{\eta} \\ (d_1 k_1 d_2 k_2, re) = 1 \\ d_1 k_1 \equiv d_2 k_2 \pmod{re}}} u_{d_1k_1} \bar{u_{d_2k_2}} 
    \sum_{1 \le |h| \le H} \hat{\Phi}\left(\frac{hM}{r}\right) \e\left(a_1 h \frac{\bar{a_2d_1k_1}}{r} \right)
    \\
    &=
    \sum_{\substack{e \in \msE \\ (r, a_1 a_2) = 1 \\ (re, d_1 d_2) = 1}} \Phi\left(\frac{r}{R}\right) \frac{1}{r}
    \sum_{\substack{(k_1, k_2) = 1 \\ (k_1 k_2, re a_2 d_1 d_2) = 1 \\ (d_1 k_1 - d_2 k_2, (a_2 d_1 d_2)^\infty) = d_\Delta \\ |d_1 k_1 - d_2 k_2| > K/x^\eta \\ d_1k_1 \equiv d_2k_2 \pmod{re}}} u_{d_1k_1} \bar{u_{d_2k_2}} 
    \sum_{1 \le |h| \le H} \hat{\Phi}\left(\frac{hM}{r}\right) \e\left(a_1 h \frac{\bar{a_2d_1k_1}}{r} \right).
\end{aligned}
\]
Let us denote
\[
    a := a_1 \sgn(a_2)
    \qquad\qquad \text{and} \qquad\qquad 
    b := |a_2| d_1
\] 
for convenience. 
At this point we record that, since $d_1d_2d_\Delta \mid (a_2v)^\infty$, $v \mid d_1$, and $a_1, a_2 \ll x^{\delta} \le x^\eta$ by \cref{eq:conditions}, we have
\[
    a_2 d_1 d_2 d_\Delta \mid b^\infty 
    \qquad\qquad 
    \text{and} 
    \qquad\qquad 
    a b d_1 d_2 d_\Delta v \ll x^{O(\eta)},
\]
as needed in \cref{prop:drappeau-expo} (in particular, $b$ will act as a bookkeeper for the prime factors of $d_1, d_2, d_\Delta$, $a_2$ inside coprimality constraints). 
Recalling that we chose $\mE = \{e \sim E : e \text{ prime}\}$ in \cref{eq:averaging-set}, we can ensure that $(e, a_2d_1) = (e, b) = 1$ for $e \in \mE$ by enforcing $\delta < 4\eps$ (since then $|a_2| \le x^\delta \le x^{4\eps} \le E$).
Writing $d_1k_1 - d_2k_2 = red_\Delta t$, where $(t, a_2d_1d_2k_1k_2) = 1$ (which is further absorbed by the conditions $(t, b) = (k_1, k_2) = (k_1 k_2, b) = 1$), we further get
\[
    \mD_2 = \sum_{\substack{e \sim E \text{ prime} \\ (e, b) = 1}} \sum_{(r, ab) = 1} \Phi\left(\frac{r}{R}\right) \frac{1}{r}
    \sum_{\substack{(k_1, k_2) = 1 \\ (k_1 k_2, re b) = 1 \\ d_1k_1 - d_2k_2 = red_\Delta t \\ (t, b) = 1,\  red_\Delta |t| > K/x^\eta}} u_{d_1k_1} \bar{u_{d_2k_2}} 
    \sum_{1 \le |h| \le H} \hat{\Phi}\left(\frac{hM}{r}\right) \e\left(a h \frac{\bar{bk_1}}{r} \right).
\]

We can also get rid of the restriction $(r, a) = 1$ using M\"obius inversion, by writing $\one_{(r, a) = 1} = \sum_{d_a \mid a} \mu(d_a) \one_{d_a \mid r}$ and expanding
\begin{equation} \label{eq:D23}
    \mD_2(d_1, d_2, d_\Delta, v) = \sum_{d_a \mid a} \mu(d_a)\, \mD_3(d_1, d_2, d_\Delta, v, d_a),
\end{equation}
where
\[
    \mD_3 := \sum_{\substack{e \sim E \text{ prime} \\ (e, b) = 1}} \sum_{\substack{(r, b) = 1 \\ r \equiv 0 \pmod{d_a}}} \Phi\left(\frac{r}{R}\right) \frac{1}{r}
    \sum_{\substack{(k_1, k_2) = 1 \\ (k_1 k_2, re b) = 1 \\ d_1k_1 - d_2k_2 = red_\Delta t \\ (t, b) = 1,\  red_\Delta |t| > K/x^\eta}} u_{d_1k_1} \bar{u_{d_2k_2}} 
    \sum_{1 \le |h| \le H} \hat{\Phi}\left(\frac{hM}{r}\right) \e\left(a h \frac{\bar{bk_1}}{r} \right).
\]
Finally, using the definition of $(u_k)$ from \cref{eq:u-def} and the fact that $(d_2, k_2) \mid (b, k_2) = 1$, we can expand
\[
    u_{d_2 k_2} = \sum_{\delta_1 \delta_2 = d_2} \sum_{n\ell = k_2} \beta_{\delta_1 n} \gamma_{\delta_2 \ell},
\]
and thus
\begin{equation} \label{eq:D34}
    \mD_3(d_1, d_2, d_\Delta, v, d_a) = \sum_{\delta_1 \delta_2 = d_2} \mD_4(d_1, d_2, d_\Delta, v, d_a, \delta_1),
\end{equation}
where
\[
\begin{aligned}
    &\mD_4 \\
    &:= \sum_{\substack{e \sim E \text{ prime} \\ (e, b) = 1}} \sum_{\substack{(r, b) = 1 \\ r \equiv 0 \pmod{d_a}}} \Phi\left(\frac{r}{R}\right) \frac{1}{r} 
    \sum_{\substack{(k, n\ell) = 1 \\ (k n \ell, re b) = 1 \\ d_1k - d_2n \ell = red_\Delta t \\ (t, b) = 1,\  red_\Delta |t| > K/x^\eta}} u_{d_1k} \bar{\beta_{\delta_1 n} \gamma_{\delta_2 \ell}} 
    \sum_{1 \le |h| \le H} \hat{\Phi}\left(\frac{hM}{r}\right) \e\left(a h \frac{\bar{bk}}{r} \right).
\end{aligned}
\]
We can now apply \cref{prop:drappeau-expo} for the sequences
\[
    u'_k := u_{d_1 k},
    \qquad\quad 
    \beta'_n := \bar{\beta_{\delta_1 n}},
    \qquad\quad
    \gamma'_\ell := \bar{\gamma_{\delta_2 \ell}},
\]
supported in $k \asymp K/d_1$, $n \sim N/\delta_1$, and $\ell \sim L/\delta_2$ respectively, to get
\begin{equation} \label{eq:D4-bound}
    \mD_4 \ll_{\eta,\eps} x^{O(\eta) - \eps/4} \frac{K^2}{R}.
\end{equation}
Putting together \cref{eq:D12,eq:D23,eq:D34,eq:D4-bound}, we obtain
\begin{equation} \label{eq:D1-bound}
    \mD_1 \ll_\eta x^{O(\eta)} \max_{d_1,d_2,d_\Delta,v,d_a,\delta_1} |\mD_4| \ll_{\eps, \eta} x^{O(\eta) - \eps/4} \frac{K^2}{R},
\end{equation}
where the maximum includes all applicable restrictions on the tuple $(d_1,d_2,d_\Delta,v,d_a,\delta_1)$ (which takes at most $O_\eta(x^{O(\eta)})$ values).

Now let $C > 0$ denote the absolute constant in the final exponent of $O(\eta)$ from \cref{eq:D1-bound}, and let us pick 
\[
    \eta = \eta(\eps) := \min(\eps/(8C), \eps),
    \qquad\qquad 
    \delta = \delta(\eps) := \min(\eps/16, \eta/8).
\] 
Then we have $0 < 8\delta \le \eta \le \eps$ as desired, and the bound in \cref{eq:D1-bound} implies
\[
    \mD_1 \ll_{\eps, \eta} x^{\eps/8 - \eps/4} \frac{K^2}{R}
    \le 
    x^{-2\delta} \frac{K^2}{R},
\]
completing our proof.
\end{proof}

Finally, we prove \cref{prop:drappeau-expo} assuming the following BFI-style bound, the proof of which is left to the later sections. This should be compared to Maynard's \cite[Lemma 18.5]{maynard2025primes}.

\begin{proposition}[Improved BFI/Maynard exponential sum bound] \label{prop:bfi-expo}
For all sufficiently small $\eps > 0$ and all $\eta \in (0, 1)$, the following holds. Under the conditions in \cref{eq:conditions}, there exists $E$ satisfying \cref{eq:E-conditions}, such that for any positive integers $b, d$ with $b \ll x^{O(\eta)}$ and $d \ll x^{O(1)}$, and for any parameters $K' \ll NL x^{O(\eta)}$, $N' \asymp N x^{O(\eta)}$, $L' \asymp L x^{O(\eta)}$, $T' \ll NL(RE)^{-1} x^{O(\eta)}$, $H' \ll R M^{-1} x^{O(\eta)}$, one has
\begin{equation} \label{eq:bfi-expo}
\begin{aligned}
    \sum_{\substack{k \sim K' \\ (k, b) = 1}} 
    \sum_{\substack{n \sim N' \\ (n, k) = 1}} 
    \sum_{\substack{t \sim T' \\ (t, dk) = 1}} t \left\vert
    \sum_{\substack{e \sim E \\ (e, dk) = 1}}
    \sum_{h \sim H'} 
    \sum_{\substack{\ell \sim L' \\ (\ell, ket) = 1 \\ n \ell \equiv d k \pmod{et}}} 
    \beta(e, h, \ell)\ \e\left(het \frac{\bar{b n \ell}}{k}\right)
    \right\vert^2
    \\
    \ll_{\eps, \eta} x^{O(\eta) - \eps/2} N^2 L^3,
\end{aligned}
\end{equation}
for any $1$-bounded complex coefficients $\beta(e, h, \ell)$ (independent of $k, n, t$). 
\end{proposition}

\begin{remark}
The exponential sum from \cref{prop:bfi-expo} is essentially anticipated in \cref{eq:sketch-bfi-style,eq:sketch-bfi-style-2}.
\end{remark}

\begin{proof}[Proof of \cref{prop:drappeau-expo}, assuming \cref{prop:bfi-expo}]
Let us denote the sum in \cref{prop:drappeau-expo} by $\mD_4$, and assume without loss of generality that $(d_a, b) = 1$ (since otherwise $\mD_4$ vanishes). We choose $E$ as in \cref{prop:bfi-expo}, and take a closer look at the exponential: by iterating \cref{lem:bezout}, since $b, k, r$ are pairwise coprime we have
\[
    \frac{\bar{bk}}{r} + \frac{\bar{rb}}{k} + \frac{\bar{rk}}{b} \equiv \frac{1}{bkr} \quad \pmod{1},
\]
and thus, using that $re d_\Delta t \equiv -d_2 n \ell \pmod{k}$ and $d_1d_2d_\Delta \mid b^\infty$ (so in particular $(d_2, k) = 1$),
\[
    ah\frac{\bar{b k}}{r} \equiv 
    ah e d_\Delta t \frac{\bar{b d_2n\ell}}{k} - ah\frac{\bar{rk}}{b} + \frac{ah}{b k r} \quad \pmod{1}.
\]
Since $|a| \ll x^{O(\eta)}$, $h \ll x^\eta R M^{-1}$, $kr \gg KR/x^{O(\eta)}$ and $KM \asymp x$, we obtain
\[
    \e\left(ah \frac{\bar{b k}}{r} \right)
    =
    \e\left(ahe d_\Delta t \frac{\bar{b d_2n \ell}}{k} - ah\frac{\bar{rk}}{b} \right)
    +
    O(x^{O(\eta)-1}),
\]
and thus
\[
\begin{aligned}
    \mD_4 =
    &\sum_{\substack{e \sim E \text{ prime} \\ (e, b) = 1}} \sum_{\substack{(r, b) = 1 \\ r \equiv 0 \pmod{d_a}}} \Phi\left(\frac{r}{R}\right) \frac{1}{r} 
    \sum_{\substack{(k, n\ell) = 1 \\ (k n \ell, reb) = 1 \\ d_1k - d_2n \ell = red_\Delta t \\ (t, b) = 1,\  red_\Delta |t| > K/x^\eta}} u'_k \beta'_n \gamma'_\ell
    \\
    &\times 
    \sum_{1 \le |h| \le H} \hat{\Phi}\left(\frac{hM}{r}\right) \e\left(ahe d_\Delta t \frac{\bar{b d_2n \ell}}{k} - ah\frac{\bar{rk}}{b} \right)
    + O_\eta\left(E\frac{R}{R}NL \left(\frac{K}{RE} + 1\right) H x^{O(\eta) - 1} \right).
\end{aligned}
\]
Recalling that $K \gg RE$ (due to \cref{eq:E-conditions}), $H = x^\eta RM^{-1}$, $MK \asymp x$, and $KR \ll x^{2-\eps}$ (again by \cref{eq:conditions}), the error term gives an acceptable contribution of
\[
    \ll_\eta \frac{EK^2 H x^{O(\eta) - 1}}{RE} 
    \ll K^3 x^{O(\eta) - 2}
    \ll x^{O(\eta)-\eps} \frac{K^2}{R}.
\]
We now change variables in the main term by replacing the $r$-summation with a summation over 
\[
    t := \frac{d_1 k - d_2 n \ell}{red_\Delta},
\]
noting that the condition $red_\Delta |t| > K/x^\eta$ implies $|t| > K/(RE x^{O(\eta)})$. We also put $|t|$ into dyadic intervals $|t| \sim T$ to obtain
\begin{equation} \label{eq:D45}
    \mD_4 \ll (\log x) \sup_{T \asymp x^{O(\eta)} K/(RE)} |\mD_5(T)|
    +
    O_\eta\left(x^{O(\eta) - \eps} \frac{K^2}{R}\right),
\end{equation}
where, after adjusting coprimality conditions as explained below,
\[
\begin{aligned}
    \mD_5 := \sum_{\substack{e \sim E \text{ prime} \\ (e, b) = 1}}
    \sum_{\substack{|t| \sim T \\ (t, b) = 1}}\,
    \sum_{\substack{(k, n\ell) = (k n \ell, d_aebt) = 1 \\ d_1k \equiv d_2n \ell \pmod{d_aed_\Delta t} \\ |d_1k - d_2n \ell| > K/x^\eta \\ r := (d_1 k - d_2 n \ell)/(ed_\Delta t) \\ (r, b) = 1}} 
    &\Phi\left(\frac{r}{R}\right) \frac{1}{r} 
    u'_k \beta'_n \gamma'_\ell 
    \\
    &\times \sum_{1 \le |h| \le H} \hat{\Phi}\left(\frac{hM}{r}\right) \e\left(ahe d_\Delta t \frac{\bar{b d_2n \ell}}{k} - ah\frac{\bar{rk}}{b} \right).
\end{aligned}
\]
(We inserted the condition $(kn\ell, t) = 1$; this must happen since $t \mid d_1 k - d_2 n \ell$ and $(t, d_1 d_2) \mid (t, b^\infty) = 1$; if a prime divides both $t$ and one of $k$ and $n\ell$, then it must also divide the other, contradicting $(k, n\ell) = 1$. Moreover, the conditions in the sum over $k, n, \ell$ are enough to imply $(kn\ell, r) = 1$, since $(d_1 k - d_2 n\ell, k) = (d_2 n\ell, k) = (d_2, k) \mid (b^\infty, k) = 1$ and similarly $(d_1 k-d_2 n\ell, n\ell) = 1$.)

We now aim to simplify the term $ah \bar{rk}/b$ from the exponential, by fixing all relevant residues modulo $b$. With this goal, we denote the residues of $e, t$ modulo $b$ by $\hat{e}, \hat{t}$, and those of $k, n, \ell$ modulo $d_\Delta b$ by $\hat{k}, \hat{n}, \hat{\ell}$. Since $(d_1 k - d_2n\ell)/d_\Delta = ret$ is coprime with $b$, we must have $d_1 \hat{k} - d_2 \hat{n} \hat{\ell} \in d_\Delta (\Z/b\Z)^\times$ $= \{d_\Delta (n + b \Z) : (n, b) = 1\} \subset \Z/d_\Delta b \Z$. This allows us to expand $\mD_5$ as
\begin{equation} \label{eq:D56}
    \mD_5(T) = \sum_{\substack{\hat{e}, \hat{t} \in (\Z/b\Z)^\times}} \sum_{\substack{\hat{k}, \hat{n}, \hat{\ell} \in \Z/d_\Delta b\Z \\ (\hat{k} \hat{n} \hat{\ell}, b) = 1 \\ 
    d_1 \hat{k} - d_2 \hat{n} \hat{\ell} \in d_\Delta (\Z/b\Z)^\times}}
    \mD_6\left(T, \hat{e}, \hat{t}, \hat{k}, \hat{n}, \hat{\ell}\right),
\end{equation}
with
\[
\begin{aligned}
    \mD_6 := \sum_{\substack{e \sim E \text{ prime} \\ e \equiv \hat{e} \pmod{b}}}
    \sum_{\substack{|t| \sim T \\ t \equiv \hat{t} \pmod{b}}}\,
    &\sum_{\substack{(k, n, \ell) \equiv (\hat{k}, \hat{n}, \hat{\ell}) \pmod{d_\Delta b} \\ (k, n\ell) = (k n \ell, d_aet) = 1 \\ d_1k \equiv d_2n \ell \pmod{d_aed_\Delta t} \\ |d_1k - d_2n \ell| > K/x^\eta \\ r := (d_1 k - d_2 n \ell)/(ed_\Delta t)}} 
    \Phi\left(\frac{r}{R}\right) \frac{1}{r} 
    u'_k \beta'_n \gamma'_\ell 
    \\
    &\times \sum_{1 \le |h| \le H} \hat{\Phi}\left(\frac{hM}{r}\right) \e\left(ahe d_\Delta t \frac{\bar{b d_2n \ell}}{k} - ah\frac{\bar{\hat{r}\hat{k}}}{b} \right),
\end{aligned}
\]
where $\hat{r} = \hat{r}(\hat{e}, \hat{t}, \hat{k}, \hat{n}, \hat{\ell}) \in (\Z/b\Z)^\times$ is the unique residue mod $b$ such that $d_\Delta \hat{r} \hat{e} \hat{t} = d_1 \hat{k} - d_2 \hat{n} \hat{\ell} \in d_\Delta (\Z/b\Z)^\times$ (this $\hat{r}$ is the residue of $r$ mod $b$, and it is fixed inside each $\mD_6$). Denoting $y(h) := \e\left(- ah\bar{\hat{r} \hat{k}}/b \right)$ and suppressing the congruences to $\hat{e}, \hat{t}, \hat{k}, \hat{n}, \hat{\ell}$ through the notation $\sum^*$, we obtain
\begin{equation} \label{eq:before-partial}
\begin{aligned}
    \mD_6 = \psum_{\substack{|t| \sim T \\ (t, b) = 1}}
    \psum_{\substack{(k, n) = 1 \\ (kn, d_abt) = 1}}  u'_k \beta'_n
    \psum_{\substack{e \sim E \text{ prime} \\ (e, knb) = 1}}
    \psum_{\substack{(\ell, d_aekbt) = 1 \\ d_2 n \ell \equiv d_1 k \pmod{d_aed_\Delta t} \\ |d_2 n \ell - d_1 k | > K/x^\eta}}  
    \sum_{1 \le |h| \le H} \gamma'_\ell\ y(h)\ \e\left(ahe d_\Delta t \frac{\bar{b d_2n \ell}}{k}\right)
    \\
    \times \left[ \Phi \left(\frac{d_1 k - d_2 n \ell}{ed_\Delta tR} \right) \frac{e d_\Delta t}{d_1 k - d_2 n \ell} \hat{\Phi}\left(\frac{h e d_\Delta t M}{d_1 k - d_2 n \ell}\right) \right].
\end{aligned}
\end{equation}
We now remove some of the dependencies between the variables $t, k, n$ and $e, \ell, h$, as in the proof of \cite[Lemma 18.4]{maynard2025primes}. Consider the function
\[
\begin{aligned}
    \Psi(e, \ell, h) &:= \Phi \left(\frac{d_1 k - d_2 n \ell}{ed_\Delta tR} \right) \frac{e d_\Delta t}{d_1 k - d_2 n \ell} \hat{\Phi}\left(\frac{h e d_\Delta t M}{d_1 k - d_2 n \ell}\right)
    =
    \frac{1}{R} \alpha\, \Phi\left(\frac{1}{\alpha}\right) \hat{\Phi}\left(\frac{Mh}{R} \alpha \right),
\end{aligned}
\]
where $\alpha := e d_\Delta t R / (d_1 k - d_2 n \ell)$; note that $\Psi$ is smooth in $e, \ell, h$, and nonzero only if $\alpha \asymp 1$. Since $Mh/R \ll x^\eta$ and $\Phi$, $\hat{\Phi}$ have bounded derivatives, the chain rule and the bounds $d_2 n \ell \asymp K$,  $|d_2 n \ell - d_1 k| > K/x^\eta$ imply
\[
\begin{aligned}
    \frac{\partial^{j_1 + j_2 + j_3}}{(\partial e)^{j_1} (\partial (d_1 k - d_2 n \ell))^{j_2} (\partial h)^{j_3}}\Psi(e, \ell, h)
    &\ll_{j_1, j_2, j_3} \frac{x^{\eta(j_1 + j_2 + j_3)}}{R} |e|^{-j_1} |d_1 k - d_2 n \ell|^{-j_2} |h|^{-j_3}
    \\
    &\ll \frac{x^{\eta(j_1 + 2j_2 + j_3)}}{R} |e|^{-j_1} |d_2 n \ell|^{-j_2} |h|^{-j_3}.
\end{aligned}
\]
We thus have
\[
    \frac{\partial^{j_1 + j_2 + j_3}}{(\partial e)^{j_1} (\partial \ell)^{j_2} (\partial h)^{j_3}}\Psi(e, \ell, h)
    \ll_{j_1, j_2, j_3}
    \frac{x^{\eta(j_1 + 2j_2 + j_3)}}{R} |e|^{-j_1} |\ell|^{-j_2} |h|^{-j_3},
\]
and then by partial summation, \cref{eq:before-partial} implies that
\begin{equation} \label{eq:D67}
    \mD_6 \ll_\eta \frac{x^{O(\eta)}}{R} \sup_{\substack{H'' \le H \\ E'' \le 2E,\ L'' \le 2L}} \mD_7,
\end{equation}
where, after removing the residue constraints in the outer sums over $t, k, n$ for an upper bound and putting $|h|$ in a dyadic interval,
\[
    \mD_7 := \sum_{\substack{|t| \sim T \\ (t, b) = 1}} \sum_{\substack{(k, n) = 1 \\ (kn, d_abt) = 1}} |u'_k \beta'_n| \left\vert
    \psum_{\substack{e \sim E \text{ prime} \\ e \le E'' \\ (e, knb) = 1}} 
    \psum_{\substack{(\ell, d_aekbt) = 1,\ \ell \le L'' \\ d_2 n \ell \equiv d_1 k \pmod{d_aed_\Delta t} \\ |d_2 n \ell - d_1 k| > K/x^\eta \\ (d_1 k - d_2 n \ell)/t > 0}} \gamma'_\ell
    \sum_{|h| \sim H''} 
    y(h)\ \e\left(ahe d_\Delta t \frac{\bar{b d_2n \ell}}{k}\right)
    \right\vert.
\]
According to the desired bound in \cref{prop:drappeau-expo}, and in light of \cref{eq:D45,eq:D56,eq:D67}, it remains to show that
\begin{equation} \label{eq:B-bound}
    \mD_7 \ll_{\eps,\eta} x^{O(\eta)-\eps/4} K^2,
\end{equation}
for all $\eta \in (0, 1)$. Now let $\mI(t, k, n)$ be the subinterval of $[L/\delta_2, 2L/\delta_2]$ (which is the support of $\gamma'_\ell$) containing those $\ell$'s such that
\[
    \ell \le L'', \qquad |d_2 n \ell - d_1 k| > K/x^\eta, \qquad 
    \text{and} \qquad 
    (d_1 k - d_2 n \ell)/t > 0.
\]
As in the proof of \cite[Lemma 18.4]{maynard2025primes}, we can remove the constraint $\ell \in \mI(t, k, n)$ using the identity
\[
    \one_{\ell \in \mI(t, k, n)} = 
    \int_0^1 \e(\ell \omega) \sum_{j \in \mI(t, k, n)} \e(-j\omega)\, d\omega
    =
    \int_0^1 \e(\ell \omega) c(t, k, n, \omega)\min(L, \omega^{-1} (1 - \omega)^{-1}) \,d\omega,
\]
for some coefficients $c(t, k, n, \omega) \ll 1$, and the $L^1$ bound $\int_0^{1/2} \min(L, 2\omega^{-1}) d\omega \ll \log x$. Together with the divisor bound $|u'_k| \ll_\eta x^\eta$, this shows that
\[
    \mD_7 \ll_\eta x^{O(\eta)} \sup_{\omega \in \R/\Z} \mD_8,
\]
where 
\[
\begin{aligned}
    &\mD_8 
    \\
    &:= \sum_{\substack{|t| \sim T \\ (t, b) = 1}} \sum_{\substack{k \asymp K/d_1,\ n \sim N/\delta_1 \\ (k, n) = 1 \\ (kn, d_abt) = 1}} \left\vert
    \psum_{\substack{e \sim E \text{ prime} \\ e \le E'' \\ (e, knb) = 1}} 
    \psum_{\substack{(\ell, d_aekbt) = 1 \\ d_2 n \ell \equiv d_1 k \pmod{d_aed_\Delta t}}} \gamma'_\ell\, \e(\ell \omega)
    \sum_{|h| \sim H''} 
    y(h)\ \e\left(ahe d_\Delta t \frac{\bar{b d_2n \ell}}{k}\right)
    \right\vert.
\end{aligned}
\]
We denote $d_i' := d_i/v$ for $i \in \{1, 2, \Delta\}$, so that the exponential term and the congruence in the summation over $\ell$ may be rewritten as 
\[
    \e\left(ahe d_\Delta' t \frac{\bar{b d_2' n \ell}}{k}\right), 
    \qquad\qquad
    d_2' n \ell \equiv d_1' k \pmod{e d_a d_\Delta' t},
\]
where $(d_1', d_2') = (d_1', d_\Delta') = (d_2', d_\Delta') = 1$. At this point, it also makes sense to denote
\[
    h' := \frac{a}{d_a} h, \qquad\quad 
    t' := d_a d_\Delta' t, \qquad\quad 
    n' := d_2' n,
\] 
\[
    H' := \frac{a}{d_a}H'', \qquad\quad 
    T' := d_a d_\Delta' T, \qquad\quad
    N' := \frac{d_2'}{\delta_1}N, \qquad\quad
    L' := \frac{L}{\delta_2}, \qquad\quad 
    K' := \frac{K}{d_1},
\]
to bound (by dropping some divisibility constraints on $n', t'$)
\[
    \mD_8 \ll 
    \sum_{\substack{k \asymp K' \\ (k, b) = 1}} \sum_{\substack{n' \sim N' \\ (n', k) = 1}} \sum_{\substack{|t'| \sim T' \\ (t', d_1'k) = 1}} 
    \left\vert
    \psum_{\substack{e \sim E \text{ prime} \\ e \le E'' \\ (e, kn'b) = 1}} 
    \psum_{\substack{\ell \sim L' \\ (\ell, bket') = 1 \\ n' \ell \equiv d_1' k \pmod{et'}}} \gamma'_\ell\, \e(\ell \omega)
    \sum_{\substack{|h'| \sim H' \\ \frac{a}{d_a} \mid h'}} 
    y(h'd_a/a)\ \e\left(h' e t' \frac{\bar{b n' \ell}}{k}\right)
    \right\vert.
\]
(To verify the new coprimality constraints, recall that $(d_a, b) = 1$ and $d_1' d_2' d_\Delta' \mid b^\infty$.) We may replace the restriction that $(e, n') = 1$ with $(e, d_1') = 1$, since each follows from the other and the congruence $n'\ell \equiv d_1' k \pmod{e}$, where $(\ell k, e) = 1$. Moreover, by inserting $1$-bounded coefficients $\beta(e, h', \ell)$, we can get rid of the coefficients $\gamma'_\ell\, \e(\ell\omega) y(h'd_a/a)$, the residue constraints (modulo $b$) in the summations over $e$ and $\ell$, as well as of the constraints that $e$ is a prime and $e \le E'$, and that $a/d_a \mid h'$. Overall, this yields
\[
    \mD_8 \ll 
    \sum_{\substack{k \asymp K' \\ (k, b) = 1}} \sum_{\substack{n' \sim N' \\ (n', k) = 1}} \sum_{\substack{|t'| \sim T' \\ (t', d_1'k) = 1}} 
    \left\vert
    \sum_{\substack{e \sim E \\ (e, d_1'k) = 1}}
    \sum_{|h'| \sim H'}
    \sum_{\substack{\ell \sim L' \\ (\ell, ket') = 1 \\ n' \ell \equiv d_1' k \pmod{et'}}} 
    \beta(e, h', \ell)\, 
    \e\left(h' e t' \frac{\bar{b n' \ell}}{k}\right)
    \right\vert.
\]
Finally, we insert a factor of $\sqrt{|t'|/T'}$ into the sum, and apply Cauchy--Schwarz in the outer variables $k, n', t'$ to bound
\[
\begin{aligned}
    &\mD_8^2 \ll 
    \\
    &K' N' 
    \sum_{\substack{k \asymp K' \\ (k, b) = 1}} 
    \sum_{\substack{n' \sim N' \\ (n', k) = 1}} 
    \sum_{\substack{|t'| \sim T' \\ (t', d_1' k) = 1}} |t'| \left\vert
    \sum_{\substack{e \sim E \\ (e, d_1'k) = 1}}
    \sum_{|h'| \sim H'}
    \sum_{\substack{\ell \sim L' \\ (\ell, ket') = 1 \\ n' \ell \equiv d_1' k \pmod{et'}}} 
    \beta(e, h', \ell)\, 
    \e\left(h' e t' \frac{\bar{b n' \ell}}{k}\right)
    \right\vert^2.
\end{aligned}
\]
Conjugating if necessary when $h' < 0$ or $t' < 0$, \cref{prop:bfi-expo} implies that
\[
    \mD_8^2 \ll_{\eps,\eta} K'N' \cdot x^{O(\eta) - \eps/2} N^2 L^3
    \ll 
    x^{O(\eta)-\eps/2} K^4,
\]
for all $\eta \in (0, 1)$. Putting things together, we conclude that
\[
    \mD_7 \ll_{\eps,\eta} x^{O(\eta)} \sup_{\omega \in \R/\Z} \mD_8 \ll_\eta x^{O(\eta) - \eps/4} \sqrt{K^4}
    \asymp x^{O(\eta) - \eps/4} K^2, 
\]
as we required in \cref{eq:B-bound}.
\end{proof}

\section{Bombieri--Friedlander--Iwaniec-style estimates} \label{sec:bfi}

In this section we establish \cref{prop:bfi-expo}, thus completing the proof of \cref{prop:dispersion-sums} and \cref{thm:triple-conv-estimate}. We build on Maynard's work in \cite[Chapter 18]{maynard2025primes} (in a slightly more general setting, and using \cref{thm:di-type-bound} instead of \cite[Theorem 9]{deshouillers1982kloosterman}), which is in turn based on Bombieri--Friedlander--Iwaniec's work in \cite[Section 10]{bombieri1987primes2}. To aid future research, we shall consider a general sum
\[
    \mB(K, N, T, E, H, L) :=
    \sum_{\substack{k \sim K \\ (k, b) = 1}} 
    \sum_{\substack{n \sim N \\ (n, k) = 1}} 
    \sum_{\substack{t \sim T \\ (t, dk) = 1}} t \left\vert
    \sum_{\substack{e \sim E \\ (e, dk) = 1}}
    \sum_{h \sim H} 
    \sum_{\substack{\ell \sim L \\ (\ell, ket) = 1 \\ n \ell \equiv d k \pmod{et}}} 
    \beta(e, h, \ell)\ \e\left(het \frac{\bar{b n \ell}}{k}\right)
    \right\vert^2,
\]
where $b, d$ are given positive integers with $b \ll x^{O(\eta)}$ and $d \ll x$, $\beta(e, h, \ell)$ are arbitrary $1$-bounded coefficients, and the parameters $K, N, T, E, H, L$ are almost arbitrary. 

\begin{remark}
The trivial bound for $\mB$ is $KN\left(TEH\left(\frac{L}{ET} + 1\right)\right)^2 \ll KN(HL)^2 + KN(TEH)^2$, but we need more than a power saving over this (note that the desired bound in \cref{eq:bfi-expo} is on the order of $KNL^2$, since we need to make up for the factors of $H$ introduced during Poisson summation). So the relative sizes of $K, N, T, E, H, L$ (as given by \cref{prop:bfi-expo} and \cref{eq:conditions}) will ultimately be crucial, although we only take them into account after proving a general bound in \cref{prop:bfi-expo-2}.
\end{remark}

After expanding the square inside $\mB$, we reach a more complicated version of the sum anticipated in \cref{eq:sketch-bfi-style-2}. The `diagonal terms' with $h_1 e_1 \ell_2 = h_2 e_2 \ell_1$ bring a contribution of roughly $O(KNTHL + KEHT^2L)$, similarly to \cref{eq:sketch-diagonal-terms}; our deamplification setup will be helpful here. In the off-diagonal terms, we complete Kloosterman sums via \cref{lem:completion}, and the principal frequency will contribute $O(NH^2L^2 + NH^2TL)$. The remaining terms are ultimately separated into $\mB_=$ and $\mB_{\neq}$ (the latter corresponding to \cref{eq:sketch-final-sum}), depending on whether $\ell_1 = \ell_2$ or $\ell_1 \neq \ell_2$.

\begin{lemma}[Splitting the BFI-style sum] \label{lem:splitting-bfi}
For $\eta \in (0, 1)$, $1 \ll K, N, T, E, H, L \ll x$, and any positive integers $b, d$ with $b \ll x^{O(\eta)}$ and $d \ll x$, one has
\[
\begin{aligned}
    \mB(K, N, T, E, H, L) \ll_\eta\ 
    &
    x^{O(\eta)} \Bigg( KNTHL + KEHT^2L + N H^2 L^2 + N H^2 T L
    \\
    &+ \frac{N}{KE^2} \sup_{\substack{E_0 \ll E,\ S \ll TE_0, \\ J \ll (K T E^2 x^\eta)/(NE_0) \\ w \in \R/\Z,\, g_0}} E_0\, (\mB_= + \mB_{\neq}) \Bigg),
\end{aligned}
\]
where
\[
\begin{aligned}
    \mB_= &:=
    \sum_{t \sim T} \sum_{e_0 \sim E_0}\,
    \sum_{\substack{e_1', e_2' \sim E/e_0 \\ (e_1', e_2') = 1}}
    \sum_{\substack{s \sim S \\ s \mid te_0 \\ (s, e_1'e_2'b) = 1}}
    \sum_{\substack{\ell \sim L \\ (\ell, te_0e_1'e_2') = 1}} 
    \sum_{\substack{h_1, h_2 \sim H \\ h_1e_1' \neq h_2e_2'}}
    \\
    &\times
    \left\vert 
    \sum_{\substack{k' \\ (k', e_1'e_2'\ell b) = 1}} g_0 \left(\frac{k'}{K/S}\right) 
    \sum_{|j| \sim J}
    \e\left(j \left(\frac{d \bar{\ell}}{te_0e_1'e_2'} - w\right)\right)\, S\left( (h_1e_1' - h_2e_2')\, \bar{b \ell e_1' e_2'}, j; k's \right)
    \right\vert,
\end{aligned}
\]
\[
\begin{aligned}
    &\mB_{\neq} :=
    \sum_{t \sim T} \sum_{e_0 \sim E_0}\,
    \sum_{\substack{e_1', e_2' \sim E/e_0 \\ (e_1', e_2') = 1}}
    \sum_{\substack{s \sim S \\ s \mid te_0 \\ (s, e_1'e_2'b) = 1}}
    \sum_{\substack{\ell_1, \ell_2 \sim L \\ \ell_1 \equiv \ell_2 \pmod{te_0} \\ (\ell_1 \ell_2, te_0) = 1 \\ (\ell_1, e_1') = (\ell_2, e_2') = 1 \\ \ell_1 \neq \ell_2}} 
    \sum_{\substack{h_1, h_2 \sim H \\ h_1e_1'\ell_2 \neq h_2e_2'\ell_1}}
    \\
    &\times 
    \left\vert 
    \sum_{\substack{k' \\ (k', e_1'e_2'\ell_1\ell_2b) = 1}} g_0 \left(\frac{k'}{K/S}\right)
    \sum_{|j| \sim J} \e\left(j \left(\frac{\mu}{te_0e_1'e_2'} - w\right)\right)\, 
    S\left( (h_1e_1'\ell_2 - h_2e_2'\ell_1)\, \bar{b \ell_1 \ell_2 e_1' e_2'}, j; k's\right)
    \right\vert,
\end{aligned}
\]
and $g_0(t)$ runs over smooth functions supported on $t \asymp 1$, satisfying $\|g_0^{(j)}\|_\infty \ll_j 1$ for each $j \ge 0$ (with fixed implicit constants). Here $\mu = \mu(\ell_1, \ell_2, t, e_0, e_1', e_2', d)$ is the unique solution $\pmod{te_0e_1'e_2'}$ to the congruences $\mu \equiv d\bar{\ell_1} \pmod{te_0e_1'}$ and $\mu \equiv d\bar{\ell_2} \pmod{te_0e_2'}$.
\end{lemma}

\begin{proof}
This is essentially the same as the proof of \cite[Lemma 18.5]{maynard2025primes}, but in a slightly more general setting (the main difference being the additional parameters $b, d$). 
We first replace the indicator functions of $k \sim K$ and $n \sim N$ with smooth majorants, using a suitable smooth compactly supported function $f_0$ (we choose this as in the proof of \cite[Lemma 18.5]{maynard2025primes}). Expanding out the square in $\mB$ and swapping sums, then using that $(n, et) = 1$ to deduce a congruence between the resulting variables $\ell_1$ and $\ell_2$ (indeed, if a prime $p$ divided both $n$ and $et$, then it would divide $dk$, but $(et, dk) = 1$), we obtain
\begin{equation} \label{eq:b-estimate}
\begin{aligned}
    \mB &\le 
    \sum_{\substack{k \\ (k, b) = 1}} f_0\left(\frac{k}{K}\right)
    \sum_{\substack{n \\ (n, k) = 1}} f_0\left(\frac{n}{N}\right) 
    \sum_{\substack{t \sim T \\ (t, dk) = 1}} t \left\vert
    \sum_{\substack{e \sim E \\ (e, dk) = 1}}
    \sum_{h \sim H} 
    \sum_{\substack{\ell \sim L \\ (\ell, ket) = 1 \\ n \ell \equiv d k \pmod{et}}} 
    \beta(e, h, \ell)\ \e\left(het \frac{\bar{b n \ell}}{k}\right)
    \right\vert^2
    \\
    &=
    \sum_{\substack{t \sim T \\ (t, d) = 1}} t 
    \sum_{\substack{e_1, e_2 \sim E \\ (e_1 e_2, d) = 1}}\ 
    \sum_{\substack{\ell_1, \ell_2 \sim L \\ \ell_1 \equiv \ell_2 \pmod{t(e_1, e_2)} \\ (\ell_1, te_1) = 1 \\ (\ell_2, te_2) = 1}}
    \sum_{h_1, h_2 \sim H} \beta(e_1, h_1, \ell_1) \, \bar{\beta(e_2, h_2, \ell_2)} 
    \\
    &\qquad\qquad \times 
    \sum_{\substack{k \\ (k, te_1e_2\ell_1\ell_2b) = 1}} 
    f_0\left(\frac{k}{K}\right)
    \sum_{\substack{(n, k) = 1 \\ n \equiv dk\bar{\ell_1} \pmod{te_1} \\ n \equiv dk\bar{\ell_2} \pmod{te_2}}}
    f_0\left(\frac{n}{N}\right)
    \e \left(t\frac{(h_1e_1\ell_2 - h_2e_2\ell_1)\bar{bn \ell_1 \ell_2}}{k} \right).
\end{aligned}
\end{equation}

Let $\mB_1$ denote the contribution of the `diagonal' terms with $h_1 e_1 \ell_2 = h_2 e_2 \ell_1$, and $\mB_2$ contain the other terms; thus we have $\mB \le \mB_1 + \mB_2$. As in \cref{eq:sketch-diagonal-terms}, we first bound $\mB_1$ trivially (using the divisor bound), by
\begin{equation} \label{eq:b1-estimate}
\begin{aligned}
    \mB_1 &\ll 
    \sum_{\substack{t \sim T \\ (t, d) = 1}} t 
    \sum_{\substack{e_1 \sim E \\ \ell_2 \sim L \\ h_1 \sim H}}
    \,
    \sum_{\substack{e_2 \sim E \\ \ell_1 \sim L \\ e_2\ell_1 \mid h_1e_1\ell_2 \\ \ell_1 \equiv \ell_2 \pmod{t(e_1,e_2)} \\ (\ell_1, te_1) = (\ell_2, te_2) = 1}}
    \sum_{\substack{h_2 \sim H \\ h_1e_1\ell_2 = h_2e_2\ell_1}}
    \sum_k 
    f_0\left(\frac{k}{K}\right)
    \sum_{\substack{n \equiv dk\bar{\ell_1} \pmod{te_1} \\ n \equiv dk\bar{\ell_2} \pmod{te_2}}}
    f_0\left(\frac{n}{N}\right)
    \\
    &\ll 
    \sum_{t \sim T} t 
    \sum_{\substack{e_1 \sim E \\ \ell_2 \sim L \\ h_1 \sim H}}
    \,
    \sum_{\substack{e_2 \sim E,\ \ell_1 \sim L \\ h_2 \sim H \\ h_1e_1\ell_2 = h_2e_2\ell_1  \\ (\ell_1, te_1) = 1}}
    \sum_{k \ll K} 
    \sum_{\substack{n \ll N \\ n \equiv dk\bar{\ell_1} \pmod{te_1}}}
    1
    \\
    &\ll_\eta x^{O(\eta)}\, 
    T^2 ELHK \left(\frac{N}{TE} + 1\right),
\end{aligned}
\end{equation}
recovering the first two terms in the desired bound. Next, we consider $\mB_2$, containing the terms with $h_1 e_1 \ell_2 \neq h_2 e_2 \ell_1$. We let $e_0 := (e_1, e_2)$, $e_1' := e_1/e_0$ and $e_2' := e_2/e_0$ and put $e_0$ in dyadic ranges $e_0 \sim E_0$ to write 
\[
\begin{aligned}
    \mB_2 &\ll (\log x)
    \\
    &\times \sup_{E_0 \ll E} \bigg\vert 
    \sum_{\substack{t \sim T \\ (t, d) = 1}} t \sum_{\substack{e_0 \sim E_0 \\ (e_0, d) = 1}}\, 
    \sum_{\substack{e_1', e_2' \sim E/e_0 \\ (e_1', e_2') = 1 \\ (e_1' e_2', d) = 1}}
    \sum_{\substack{\ell_1, \ell_2 \sim L \\ \ell_1 \equiv \ell_2 \pmod{te_0} \\ (\ell_1 \ell_2, te_0) = 1 \\ (\ell_1, e_1') = (\ell_2, e_2') = 1}}
    \sum_{\substack{h_1, h_2 \sim H \\ h_1e_1'\ell_2 \neq h_2e_2'\ell_1}} \beta(e_0e_1', h_1, \ell_1) \, \bar{\beta(e_0e_2', h_2, \ell_2)} 
    \\
    &\times 
    \sum_{\substack{k \\ (k, te_0e_1'e_2'\ell_1\ell_2b) = 1}} 
    f_0\left(\frac{k}{K}\right)
    \sum_{\substack{(n, k) = 1 \\ n \equiv dk\bar{\ell_1} \pmod{te_0e_1'} \\ n \equiv dk\bar{\ell_2} \pmod{te_0e_2'}}}
    f_0\left(\frac{n}{N}\right)
    \e \left(te_0\frac{(h_1e_1'\ell_2 - h_2e_2'\ell_1)\bar{bn \ell_1 \ell_2}}{k} \right)\bigg\vert.
\end{aligned}
\]
Note that the inner sum over $n$ can be rewritten as
\[
    \sum_{\substack{n \equiv k\mu \pmod{te_0e_1'e_2'} \\ (n, k) = 1}}
    f_0\left(\frac{n}{N}\right)
    \e \left(\frac{\bar{n}r_0}{k} \right),
\]
where $r_0 := te_0 (h_1e_1'\ell_2 - h_2e_2'\ell_1)\bar{b \ell_1 \ell_2}$ (defined mod $k$), and $\mu = \mu(\ell_1, \ell_2, t, e_0, e_1', e_2', d)$ is the unique solution (mod $te_0e_1'e_2'$) to the congruences $\mu \equiv d\bar{\ell_1} \pmod{te_0e_1'}$ and $\mu \equiv d\bar{\ell_2} \pmod{te_0e_2'}$; the latter is well-defined by the Chinese Remainder Theorem, since $(te_0e_1', te_0e_2') = te_0(e_1', e_2') = te_0$, $[te_0e_1', te_0e_2'] = te_0e_1'e_2'$, and $d\bar{\ell_1} \equiv d\bar{\ell_2} \pmod{te_0}$. Crucially, note that $\mu$ does not depend on $k$.

We can thus complete Kloosterman sums using \cref{lem:completion}, with
\begin{gather*} 
    q := te_0e_1'e_2',  \\
    J_0 := 32\, x^\eta \frac{KTE^2}{NE_0} \ge x^\eta \frac{kq}{N},
\end{gather*}
giving us
\[
\begin{aligned}
     \sum_{\substack{n \equiv k\mu \pmod{q} \\ (n, k) = 1}} 
    f_0\left(\frac{n}{N}\right) \e\left(\frac{\bar{n}r_0}{k}\right)
    &=
    \frac{N}{kq} \sum_{|j| \le J_0} \hat{f_0}\left(\frac{jN}{kq}\right) \e\left(\frac{j\mu}{q}\right)
    S(j\bar{q}, r_0; k)
    +  
    O_{\eta}\left(x^{-99}\right)
    \\
    &=
    \frac{N}{kq} \sum_{|j| \le J_0} \hat{f_0}\left(\frac{jN}{kq}\right) \e\left(\frac{j\mu}{q}\right)
    S(r_1, j; k)
    +  
    O_{\eta}\left(x^{-99}\right),
\end{aligned}
\]
where
\[
    r_1 := r_0 \bar{q} = (h_1e_1'\ell_2 - h_2e_2'\ell_1)\, \bar{b \ell_1 \ell_2 e_1' e_2'}.
\]
We now plug this bound into our estimate for $\mB_2$, isolate the contribution of $j = 0$ into $\mB_3$, and let $\mB_4$ contain the terms with $j \neq 0$. This yields 
\begin{equation} \label{eq:b2-estimate}
    \mB_2 \ll_\eta x^{O(\eta)} \sup_{E_0 \ll E} (|\mB_3| + |\mB_4|) + O_\eta(x^{-50}),
\end{equation}
where
\[
\begin{aligned}
    \mB_3 := 
    \sum_{\substack{t \sim T \\ (t, d) = 1}} t \sum_{\substack{e_0 \sim E_0 \\ (e_0, d) = 1}}\, 
    \sum_{\substack{e_1', e_2' \sim E/e_0 \\ (e_1', e_2') = 1 \\ (e_1' e_2', d) = 1}}
    \sum_{\substack{\ell_1, \ell_2 \sim L \\ \ell_1 \equiv \ell_2 \pmod{te_0} \\ (\ell_1 \ell_2, te_0) = 1 \\ (\ell_1, e_1') = (\ell_2, e_2') = 1}}
    \sum_{\substack{h_1, h_2 \sim H \\ h_1e_1'\ell_2 \neq h_2e_2'\ell_1}} \beta(e_0e_1', h_1, \ell_1) \, \bar{\beta(e_0e_2', h_2, \ell_2)} 
    \\
    \qquad\qquad \times 
    \sum_{\substack{k \\ (k, te_0e_1'e_2'\ell_1\ell_2b) = 1}} 
    f_0\left(\frac{k}{K}\right)
    \frac{N}{kq} \hat{f_0}\left(0\right)
    S(r_1, 0; k),
\end{aligned}
\]
\[
\begin{aligned}
    \mB_4 := 
    \sum_{\substack{t \sim T \\ (t, d) = 1}} t 
    \sum_{\substack{e_0 \sim E_0 \\ (e_0, d) = 1}}\, 
    \sum_{\substack{e_1', e_2' \sim E/e_0 \\ (e_1', e_2') = 1 \\ (e_1' e_2', d) = 1}}
    \sum_{\substack{\ell_1, \ell_2 \sim L \\ \ell_1 \equiv \ell_2 \pmod{te_0} \\ (\ell_1 \ell_2, te_0) = 1 \\ (\ell_1, e_1') = (\ell_2, e_2') = 1}}
    \sum_{\substack{h_1, h_2 \sim H \\ h_1e_1'\ell_2 \neq h_2e_2'\ell_1}} \beta(e_0e_1', h_1, \ell_1) \, \bar{\beta(e_0e_2', h_2, \ell_2)} 
    \\
    \qquad\qquad \times 
    \sum_{\substack{k \\ (k, te_0e_1'e_2'\ell_1\ell_2b) = 1}} 
    f_0\left(\frac{k}{K}\right)
    \frac{N}{kq} \sum_{\substack{|j| \le J_0 \\ j \neq 0}} \hat{f_0}\left(\frac{jN}{kq}\right) \e\left(\frac{j\mu}{q}\right)
    S(r_1, j; k).
\end{aligned}
\]
We bound $\mB_3$ trivially using the Ramanujan bound (\cref{lem:weil}):
\begin{equation} \label{eq:b3-estimate}
\begin{aligned}
    \mB_3 &\ll 
    \sum_{t \sim T} t \sum_{e_0 \sim E_0}\,
    \sum_{\substack{e_1', e_2' \sim E/e_0}}
    \sum_{\substack{\ell_1, \ell_2 \sim L \\ \ell_1 \equiv \ell_2 \pmod{te_0}}}\,
    \sum_{\substack{h_1, h_2 \sim H}}
    \sum_{k \ll K} 
    \frac{N}{kq}
    (r_1, k)
    \\
    &\ll_\eta x^{O(\eta)} 
    \sum_{t \sim T} t \sum_{e_0 \sim E_0}\,
    \sum_{\substack{e_1', e_2' \sim E/e_0}}
    \sum_{\substack{\ell_1, \ell_2 \sim L \\ \ell_1 \equiv \ell_2 \pmod{te_0}}}\,
    \frac{H^2N}{q}
    \\
    &\ll 
    x^{O(\eta)}\, 
    T^2 \frac{E^2}{E_0}
    L\left(1 + \frac{L}{TE_0}\right)
    \frac{H^2NE_0}{TE^2}
    \\
    &\ll 
    x^{O(\eta)} \left(TLH^2N + L^2H^2N \right),
\end{aligned}
\end{equation}
giving the third and fourth terms in the desired bound.
We finally turn to estimating $\mB_4$, and start by removing the coprimality constraint $(k, te_0) = 1$, via M\"obius inversion. We write $\one_{(k, te_0) = 1} = \sum_{s \mid (k, te_0)} \mu(s)$ and $k = k' s$, and put $j, s$ into dyadic ranges $j \sim J, s \sim S$ to obtain 
\begin{equation} \label{eq:b4-estimate}
    \mB_4 \ll_\eta x^\eta T \sup_{\substack{S \ll TE_0 \\ J \ll J_0}} 
    \mB_5,
\end{equation}
where 
\begin{equation} \label{eq:b5-estimate}
    \mB_5 := 
    \sum_{\substack{t \sim T \\ (t, d) = 1}} \sum_{\substack{e_0 \sim E_0 \\ (e_0, d) = 1}}\, 
    \sum_{\substack{e_1', e_2' \sim E/e_0 \\ (e_1', e_2') = 1 \\ (e_1' e_2', d) = 1}}
    \sum_{\substack{s \sim S \\ s \mid te_0 \\ (s, e_1'e_2'b) = 1}}
    \sum_{\substack{\ell_1, \ell_2 \sim L \\ \ell_1 \equiv \ell_2 \pmod{te_0} \\ (\ell_1 \ell_2, te_0) = 1 \\ (\ell_1, e_1') = (\ell_2, e_2') = 1}}
    \mB_6,
\end{equation}
\[
    \mB_6 :=
    \sum_{\substack{h_1, h_2 \sim H \\ h_1e_1'\ell_2 \neq h_2e_2'\ell_1}}
    \left\vert
    \sum_{\substack{k' \\ (k', e_1'e_2'\ell_1\ell_2b) = 1}} 
    f_0\left(\frac{k's}{K}\right)
    \frac{N}{k'sq} \sum_{j \sim J} \hat{f_0}\left(\frac{jN}{k'sq}\right) \e\left(\frac{j\mu}{q}\right)
    S(r_1, j; k's)
    \right\vert.
\]

We now wish to separate the $j, k'$ variables in $\mB_6$ from the others, in the factors of $f_0$, $\hat{f_0}$, and the exponential term; note that $s$, $q = te_0e_1'e_2'$ and $\mu = \mu(\ell_1, \ell_2, c, t, e_0, e_1', e_2', d)$ do not depend on $j$ and $k'$. As in the proof of \cite[Lemma 18.5]{maynard2025primes}, we make use of the special choice of the smooth function $f_0(t) := \int_0^\infty \psi_0(y) \psi_0(t/y)\, dy/y$ (which is a multiplicative convolution of a bounded smooth function $\psi_0$ supported on $[1/2, 5/2]$ with itself) to find that
\[
    f_0\left(\frac{k's}{K}\right)
    =
    \int_U^{20U} \psi_0(su)\, \psi_0\left(\frac{k'}{Ku}\right) \frac{du}{u},
\]
where $U \asymp 1/S$, and also
\[
    \frac{N}{k'sq}\hat{f_0} \left( \frac{jN}{k'sq} \right) 
    =
    \int_V^{20V} \int_W^{20W} \psi_0(k'v)\, \psi_0\left(\frac{wsq}{Nv} \right) \e(-jw)\, dw \frac{dv}{v},
\]
where $V \asymp S/K$ and $W \asymp NVE_0/(STE^2) \asymp NE_0/(KTE^2)$. Plugging this into our expression for $\mB_6$, taking the integrals over $u, v, w$ outside the absolute value by the triangle inequality, and swapping them with the sum over $h_1, h_2$, we get 
\begin{equation} \label{eq:b6-estimate}
    \mB_6 \ll \int_U^{20 U} \int_V^{20 V} \int_W^{20 W} |\mB_7| \frac{du\, dv\, dw}{uv},
\end{equation}
where
\[
    \mB_7 := \sum_{\substack{h_1, h_2 \sim H \\ h_1e_1'\ell_2 \neq h_2e_2'\ell_1}}
    \left\vert 
    \sum_{\substack{k' \\ (k', e_1'e_2'\ell_1\ell_2b) = 1}} g_{u,v} \left(\frac{k'}{K/S}\right)
    \sum_{|j| \sim J} \e\left(j \left(\frac{\mu}{q} - w\right)\right)\, S(r_1, j; k's)
    \right\vert,
\]
and the smooth function
\[
    g_{u,v}(t) := \psi_0\left(\frac{t}{uS}\right) \psi_0\left(\frac{vtK}{S}\right)
\]
is supported on $t \asymp 1$. Combining this with \cref{eq:b4-estimate,eq:b5-estimate,eq:b6-estimate}, moving the integrals in $u, v, w$ to the front and taking an $L^\infty$ bound, we find that
\[
    \mB_4 \ll_\eta x^{\eta} TW
    \sup_{\substack{S \ll TE_0 \\ J \ll J_0}} 
    \sup_{\substack{u \asymp 1/S \\ v \asymp S/K \\ w \asymp NE_0/KE^2}}
    \sum_{t \sim T} \sum_{e_0 \sim E_0}\,
    \sum_{\substack{e_1', e_2' \sim E/e_0 \\ (e_1', e_2') = 1}}
    \sum_{\substack{s \sim S \\ s \mid te_0 \\ (s, e_1'e_2'b) = 1}}
    \sum_{\substack{\ell_1, \ell_2 \sim L \\ \ell_1 \equiv \ell_2 \pmod{te_0} \\ (\ell_1 \ell_2, te_0) = 1 \\ (\ell_1, e_1') = (\ell_2, e_2') = 1}}
    \mB_7,
\]
where $TW \asymp \frac{NE_0}{KE^2}$.
Letting $\mB_=$ be the contribution of the terms $\ell_1 = \ell_2$ and $\mB_{\neq}$ contain the terms with $\ell_1 \neq \ell_2$, and combining this with \cref{eq:b-estimate,eq:b1-estimate,eq:b2-estimate,eq:b3-estimate}, we recover the desired bound for $\mB$ (note that when $\ell_1 = \ell_2 = \ell$, one can take $\mu = d\bar{\ell}$).
\end{proof}

\begin{lemma}[Contribution of $\ell_1 = \ell_2$] \label{lem:split-bfi-eq}
With the notation of \cref{lem:splitting-bfi}, assuming that $EHT \ll x^{O(\eta)} KNL$, one has
\[
    \mB_= \ll_\eta 
    \frac{x^{O(\eta)} K E^2}{NE_0}
    \left(1 + \frac{K}{E^2L}\right)^{\theta_{\max}}
    \left(KT^5E^8H^2L^3 N\right)^{1/2}
    \left(1 + \frac{H}{E} + \frac{H^2}{E^2L}\right)^{1/2} \left(1 + \frac{K}{NL} \right)^{1/2}.
\]
\end{lemma}

\begin{proof}[Proof of \cref{lem:split-bfi-eq} assuming \cref{thm:di-type-bound}]
Here we adapt the proof of \cite[Lemma 18.7]{maynard2025primes}, using \cref{thm:di-type-bound} instead of \cite[Theorem 9]{deshouillers1982kloosterman}. To do this, we need to eliminate the dependency of the inner exponential coefficients on $\ell$, so we write
\[
\begin{aligned}
    \mB_= &=
    \sum_{t \sim T} \sum_{e_0 \sim E_0}\,
    \sum_{\substack{e_1', e_2' \sim E/e_0 \\ (e_1', e_2') = 1}}
    \sum_{\substack{s \sim S \\ s \mid te_0 \\ (s, e_1'e_2'b) = 1}}
    \sum_{\hat{\ell} \in (\Z/te_0e_1'e_2'\Z)^\times}
    \sum_{\substack{\ell \sim L \\ \ell \equiv \hat{\ell} \pmod{te_0e_1'e_2'}}} 
    \\
    &\times
    \sum_{\substack{h_1, h_2 \sim H \\ h_1e_1' \neq h_2e_2'}}
    \left\vert 
    \sum_{\substack{k' \\ (k', e_1'e_2'\ell b) = 1}} g_0 \left(\frac{k'}{K/S}\right) 
    \sum_{|j| \sim J}
    \e\left(j \omega \right)\, S\left( (h_1e_1' - h_2e_2')\, \bar{b \ell e_1' e_2'}, j; k's \right)
    \right\vert,
\end{aligned}
\]
where
\[
    \omega = \omega(t, e_0, e_1', e_2', \hat{\ell}, d, w) := \frac{d \bar{\hat{\ell}}}{te_0e_1'e_2'} - w \in \R/\Z.
\]
We now denote
\[
    m := h_1e_1' - h_2e_2' \ll \frac{HE}{E_0},
    \qquad\qquad 
    r := b \ell e_1' e_2' \asymp x^{O(\eta)} \frac{LE^2}{E_0^2},
\]
split into dyadic ranges $m \sim M_=$, $r \sim R_=$, and change variables from $\ell$ to $r$ to obtain
\[
\begin{aligned}
    \mB_= &\ll_\eta x^{O(\eta)}
    \sup_{\substack{M_= \ll HE/E_0 \\ R_= \asymp x^{O(\eta)} LE^2/E_0^2}}\,
    \sum_{t \sim T} \sum_{e_0 \sim E_0}\,
    \sum_{\substack{e_1', e_2' \sim E/e_0 \\ (e_1', e_2') = 1}}
    \sum_{\hat{\ell} \in (\Z/te_0e_1'e_2'\Z)^\times}
    \\
    &\times
    \sum_{\substack{r \sim R_= \\ b e_1'e_2' \mid r \\ r/(be_1'e_2') \equiv \hat{\ell} \\ \pmod{te_0e_1'e_2'}}}
    \sum_{\substack{s \sim S \\ (s, r) = 1 \\ s \mid te_0}}
    \sum_{m \sim M_=}
    \sum_{\substack{h_1, h_2 \sim H \\ h_1e_1' - h_2e_2' = m}}
    \left\vert 
    \sum_{\substack{k' \\ (k', r) = 1}} g_0 \left(\frac{k'}{K/S}\right) 
    \sum_{|j| \sim J}
    \e\left(j \omega \right)\, S\left(m \bar{r}, j; k's \right)
    \right\vert.
\end{aligned}
\]
Crucially, once the variables $t, e_0, e_1', e_2', \hat{\ell}$ are fixed, $\omega$ does not depend on $r, s, m, h_1, h_2, k'$ or $j$. Finally, we remove the absolute values by inserting $1$-bounded coefficients $\xi_{h_1,h_2}$ (also depending on $t, e_0, e_1', e_2', \hat{\ell}, r, s, m$), and denote
\[
    a_{m,r,s} = a_{m,r,s}\left(t, e_0, e_1', e_2', \hat{\ell}\right) := 
    \one_{b e_1'e_2' \mid r}
    \one_{\substack{r/(be_1'e_2') \equiv \hat{\ell} \\ \pmod{te_0e_1'e_2'}}}
    \one_{s \mid te_0}
    \sum_{\substack{h_1, h_2 \sim H \\ h_1e_1' - h_2e_2' = m}} 
    \xi_{h_1, h_2},
\]
to get
\begin{equation} \label{eq:b-eq-bound}
\begin{aligned}
    \mB_{=} 
    &\ll_\eta x^{O(\eta)}
    \sup_{\substack{M_= \ll HE/E_0 \\ R_= \asymp x^{O(\eta)} LE^2/E_0^2}}\,
    \sum_{t \sim T} \sum_{e_0 \sim E_0}\,
    \sum_{\substack{e_1', e_2' \sim E/e_0 \\ (e_1', e_2') = 1}} \sum_{\hat{\ell} \in (\Z/te_0e_1'e_2'\Z)^\times}
    |\mK_{=}|,
\end{aligned}
\end{equation}
with
\[
    \mK_{=} := 
    \sum_{\substack{r \sim R_= \\ s \sim S \\ (r, s) = 1}}
    \sum_{m \sim M_{=}} a_{m,r,s}
    \sum_{|j| \sim J} \e(j \omega)
    \sum_{\substack{k' \\ (k', r) = 1}} 
    g_0 \left(\frac{k'}{K/S}\right)\, S\left(m\bar{r}, j; k's \right).
\]
At this point we apply our Deshouillers--Iwaniec-style bound from \cref{thm:di-type-bound}, finding that
\[
\begin{aligned}
    \mK_{=} 
    \ll_\eta\ &x^{O(\eta)} 
    \left(1 + \frac{K/S}{R_= \sqrt{S}}\right)^{\theta_{\max}} \|a_{m,r,s}\|_2 \, \sqrt{J R_= S}
    \\ 
    &\times 
    \left(\frac{K^2/S^2}{R_=} \left(M_{=} + R_=S\right)\left(J + R_=S\right) + M_{=}J \right)^{1/2},
\end{aligned}
\]
where by Cauchy--Schwarz,
\[
\begin{aligned}
    &\sum_{t \sim T} \sum_{e_0 \sim E_0}\,
    \sum_{\substack{e_1', e_2' \sim E/e_0 \\ (e_1', e_2') = 1}} \sum_{\hat{\ell} \in (\Z/te_0e_1'e_2'\Z)^\times} \|a_{m,r,s}\|_2
    \\
    &\ll 
    \frac{TE^2}{E_0}
    \sqrt{\sum_{t \sim T} \sum_{e_0 \sim E_0} \sum_{\substack{e_1', e_2' \sim E/e_0 \\ (e_1', e_2') = 1}} \sum_{\substack{r \sim R_= \\ b e_1'e_2' \mid r}}
    \sum_{\substack{s \sim S \\ (s, r) = 1 \\ s \mid te_0}} \sum_{m \sim M_=} \left(\sum_{\substack{h_1, h_2 \sim H \\ h_1e_1' - h_2e_2' = m}} 1\right)^2 }
    \\
    &\ll_\eta x^{O(\eta)}
    \frac{TE^2}{E_0}
    \sqrt{ 
    TE_0 \frac{R_= E_0^2}{E^2} \sum_{m \sim M_{=}}\,
    \sum_{h_2 \sim H}
    \sum_{e_2' \ll E/E_0}
    \sum_{\substack{h_1, e_1' \\ (e_1', e_2') = 1 \\ h_1e_1' = m + h_2e_2'}}\,
    \sum_{\substack{h_1' \sim H \\ h_1'e_1' \equiv m \pmod{e_2'}}}
    \sum_{\substack{h_2' \sim H \\ h_2'e_2' = h_1'e_1' - m}} 1
    }
    \\
    &\ll_\eta x^{O(\eta)}
    \frac{TE^2}{E_0} \left(\frac{TE_0^3 R_=}{E^2} M_= H\frac{E}{E_0}\left(1 + \frac{HE_0}{E}\right)\right)^{1/2}
    \\
    &=
    x^{O(\eta)}
    \left(T^3 E^3 R_= M_= H\left(1 + \frac{H E_0}{E}\right)\right)^{1/2}.
\end{aligned}
\]
Plugging these bounds into \cref{eq:b-eq-bound}, we obtain
\[
\begin{aligned}
    \mB_{=} 
    &\ll_\eta x^{O(\eta)}
    \sup_{\substack{M_= \ll HE/E_0 \\ R_= \asymp x^{O(\eta)} LE^2/E_0^2}}\,
    \left(1 + \frac{K/S}{R_= \sqrt{S}}\right)^{\theta_{\max}} \left(T^3 E^3 R_= M_= H\left(1 + \frac{H E_0}{E}\right)\right)^{1/2} \left(J R_= S\right)^{1/2}
    \\ 
    &\times 
    \left(\frac{K^2/S^2}{R_=} \left(M_{=} + R_=S\right)\left(J + R_=S\right) + M_{=}J \right)^{1/2}.
\end{aligned}
\]

Using $M_= \ll HE/E_0$ and $R_= \asymp x^{O(\eta)} LE^2/E_0^2$, and recalling that $J \ll (KTE^2 x^\eta)/ (NE_0)$ (from \cref{lem:splitting-bfi}), we conclude that
\[
\begin{aligned}
    \mB_= &\ll_\eta x^{O(\eta)} 
    \left(1 + \frac{KE_0^2}{LE^2S^{3/2}}\right)^{\theta_{\max}}
    \left(T^3 E^3 \frac{LE^2}{E_0^2} \frac{HE}{E_0} H\left(1 + \frac{HE_0}{E}\right)\right)^{1/2}
    \left(\frac{KTE^2}{NE_0} \frac{LE^2}{E_0^2} S\right)^{1/2}
    \\
    &\times
    \left(\frac{K^2 E_0^2}{L E^2 S^2} \left(\frac{HE}{E_0} + \frac{L E^2}{E_0^2} S\right)\left(\frac{KTE^2}{NE_0} + \frac{L E^2}{E_0^2}S\right) + \frac{HKTE^3}{NE_0^2} \right)^{1/2}
    \\
    &\ll
    x^{O(\eta)} 
    \left(1 + \frac{KE_0^2}{LE^2}\right)^{\theta_{\max}}
    \left(T^3 E^3 \frac{LE^2}{E_0^2} \frac{HE}{E_0} H\left(1 + \frac{HE_0}{E}\right)\right)^{1/2}
    \left(\frac{KTE^2}{NE_0} \frac{LE^2}{E_0^2} \right)^{1/2}
    \\
    &\times
    \left(\frac{K^2 E_0^2}{L E^2} \left(\frac{HE}{E_0 S} + \frac{L E^2}{E_0^2}\right)\left(\frac{KTE^2}{NE_0} + \frac{L E^2}{E_0^2}S\right) + \frac{HKTE^3}{NE_0^2}S \right)^{1/2},
\end{aligned}
\]
where we lower-bounded $S \gg 1$ in the factor raised to $\theta_{\max}$. Since $1 \ll S \ll TE_0$ (from \cref{lem:splitting-bfi}), our bound becomes
\[
\begin{aligned}
    \mB_= &\ll_\eta
    x^{O(\eta)}
    \left(1 + \frac{KE_0^2}{LE^2}\right)^{\theta_{\max}}
    \left(T^3 E^3 \frac{LE^2}{E_0^2} \frac{HE}{E_0} H\left(1 + \frac{HE_0}{E}\right)\right)^{1/2}
    \left(\frac{KTE^2}{NE_0} \frac{LE^2}{E_0^2} \right)^{1/2}
    \\
    &\times
    \left(\frac{K^2 E_0^2}{L E^2} \left(\frac{HE}{E_0} + \frac{L E^2}{E_0^2}\right)\left(\frac{KTE^2}{NE_0} + \frac{L E^2 T}{E_0}\right) + \frac{HKT^2E^3}{NE_0} \right)^{1/2}.
\end{aligned}
\]
This expression is nonincreasing in $E_0$, even after extracting a factor of $E_0^{-1}$ (since $\theta_{\max} < 3/4$); thus lower-bounding $E_0 \gg 1$ we obtain
\[
\begin{aligned}
    \mB_= &\ll_\eta
    \frac{x^{O(\eta)}}{E_0}
    \left(1 + \frac{K}{LE^2}\right)^{\theta_{\max}}
    \left(T^3 E^6 L H^2 \left(1 + \frac{H}{E}\right)\right)^{1/2}
    \left(\frac{KTLE^4}{N}\right)^{1/2}
    \\
    &\times
    \left(\frac{K^2}{L E^2} \left(HE + L E^2\right)\left(\frac{KTE^2}{N} + L E^2 T\right) + \frac{HKT^2E^3}{N} \right)^{1/2}.
\end{aligned}
\]
We can simplify this further using our assumption that $EHT \ll x^{O(\eta)} KNL$, which implies
\[
    x^{-O(\eta)} \frac{HKT^2E^3}{N} \ll K^2 L E^2 T = \frac{K^2}{LE^2} LE^2 LE^2T,
\]
allowing us to discard the term of $HKT^2E^3/N$ in the second line. Thus
\[
\begin{aligned}
    \mB_= &\ll_\eta
    \frac{x^{O(\eta)}}{E_0}
    \left(1 + \frac{K}{LE^2}\right)^{\theta_{\max}}
    \left(1 + \frac{H}{E}\right)^{1/2}
    \left(T^3 E^6 L H^2 \frac{KTLE^4}{N} \right)^{1/2}
    \\
    &\times
    K \left(\frac{H}{LE} + 1\right)^{1/2} \left(LE^2T\right)^{1/2} \left(\frac{K}{NL} + 1 \right)^{1/2}
    \\
    &\ll 
    \frac{x^{O(\eta)}}{E_0} 
    \left(1 + \frac{K}{LE^2}\right)^{\theta_{\max}}
    \left(\frac{K^3T^5L^3E^{12}H^2}{N} \right)^{1/2}
    \left(1 + \frac{H}{E} + \frac{H^2}{LE^2}\right)^{1/2} \left(1 + \frac{K}{NL} \right)^{1/2}
    \\
    &\ll 
    \frac{x^{O(\eta)}}{E_0} \frac{KE^2}{N}
    \left(1 + \frac{K}{LE^2}\right)^{\theta_{\max}}
    \left(KT^5L^3E^8H^2 N\right)^{1/2}
    \left(1 + \frac{H}{E} + \frac{H^2}{LE^2}\right)^{1/2} \left(1 + \frac{K}{NL} \right)^{1/2}.
\end{aligned}
\]
After slightly rearranging factors, this yields the desired bound.
\end{proof}

\begin{lemma}[Contribution of $\ell_1 \neq \ell_2$] \label{lem:split-bfi-neq}
With the notation of \cref{lem:splitting-bfi}, assuming that $EHT \ll x^{O(\eta)} KNL$, one has
\[
    \mB_{\neq} \ll_\eta 
    \frac{x^{O(\eta)} K E^2}{NE_0} 
    \left(1 + \frac{K}{E^2 L^2}\right)^{\theta_{\max}}
    \left(K T^4 E^8 H^2 L^6 N\right)^{1/2}
    \left(1 + \frac{H}{EL}\right) \left(1 + \frac{K}{NL^2}\right)^{1/2}
\]
\end{lemma}

\begin{proof}[Proof of \cref{lem:split-bfi-neq} assuming \cref{thm:di-type-bound}]
Here we follow the proof of \cite[Lemma 18.6]{maynard2025primes}, using \cref{thm:di-type-bound} instead of \cite[Theorem 9]{deshouillers1982kloosterman}. As in \cref{lem:split-bfi-eq}, we need to eliminate the dependency of the inner exponential coefficients on $\ell_1$ and $\ell_2$, so we write
\[
\begin{aligned}
    \mB_{\neq} 
    &= 
    \sum_{t \sim T} \sum_{e_0 \sim E_0}\,
    \sum_{\substack{e_1', e_2' \sim E/e_0 \\ (e_1', e_2') = 1}}
    \sum_{\substack{s \sim S \\ s \mid te_0 \\ (s, e_1'e_2'b) = 1}}\,
    \sum_{\hat{\mu} \in \Z/te_0e_1'e_2'\Z}
    \sum_{\substack{\ell_1, \ell_2 \sim L \\ \ell_1 \equiv \ell_2 \pmod{te_0} \\ (\ell_1 \ell_2, te_0) = 1 \\ (\ell_1, e_1') = (\ell_2, e_2') = 1 \\ \ell_1 \neq \ell_2 \\ \mu(\ell_1, \ell_2) \equiv \hat{\mu} \pmod{te_0e_1'e_2'}}} 
    \\
    &\times 
    \sum_{\substack{h_1, h_2 \sim H \\ h_1e_1'\ell_2 \neq h_2e_2'\ell_1}}
    \left\vert 
    \sum_{\substack{k' \\ (k', e_1'e_2'\ell_1\ell_2b) = 1}} g_0 \left(\frac{k'}{K/S}\right)
    \sum_{|j| \sim J} \e\left(j \omega\right)\, 
    S\left( (h_1e_1'\ell_2 - h_2e_2'\ell_1)\, \bar{b \ell_1 \ell_2 e_1' e_2'}, j; k's\right)
    \right\vert,
\end{aligned}
\]
where
\[
    \omega = \omega(t, e_0, e_1', e_2', \hat{\mu}, d, w) := \frac{\hat{\mu}}{te_0e_1'e_2'} - w \in \R/\Z.
\]
This is essentially the exponential sum anticipated in \cref{eq:sketch-final-sum}.

We then let $\ell_0 := (\ell_1, \ell_2)$, $\ell_1' := \ell_1/\ell_0$, $\ell_2' := \ell_2/\ell_0$, and put the variables
\[
    m := h_1e_1'\ell_2' - h_2e_2'\ell_1' \ll \frac{HEL}{E_0 \ell_0}, \qquad\qquad
    r := b \ell_0 \ell_1' \ell_2' e_1' e_2' \asymp x^{O(\eta)} \frac{L^2 E^2}{\ell_0 E_0^2},
\]
and $\ell_0$ into dyadic ranges $m \sim M_{\neq}$, $r \sim R_{\neq}$, $\ell_0 \sim L_0$ to obtain 
\begin{equation} \label{eq:bneq-bound}
    \mB_{\neq} \ll_\eta x^{O(\eta)} \sup_{\substack{L_0 \ll L \\ M_{\neq} \ll HEL/(E_0L_0) \\ R_{\neq} \asymp x^{O(\eta)} L^2E^2/(L_0E_0^2)}} |\mB_{\neq}'|,
\end{equation}
where
\[
\begin{aligned}
    \mB_{\neq}' 
    &:=
    \sum_{t \sim T} \sum_{e_0 \sim E_0}\,
    \sum_{\substack{e_1', e_2' \sim E/e_0 \\ (e_1', e_2') = 1}}\,
    \sum_{\hat{\mu} \in \Z/te_0e_1'e_2'\Z}
    \\
    &\times 
    \sum_{\substack{r \sim R_{\neq} \\ be_1'e_2' \mid r \\ s \sim S \\ (s, r) = 1 \\ s \mid te_0}}
    \sum_{\substack{\ell_0 \sim L_0 \\ \ell_1', \ell_2' \sim L/\ell_0 \\ (\ell_1', \ell_2') = 1, \ell_1' \neq \ell_2' \\ b\ell_0 \ell_1' \ell_2' e_1' e_2' = r \\ \ell_1' \equiv \ell_2' \pmod{te_0} \\ (\ell_1', e_1') = (\ell_2', e_2') = 1 \\ \mu(\ell_0\ell_1', \ell_0\ell_2') \equiv \hat{\mu} \\ \pmod{te_0e_1'e_2'}}} 
    \sum_{m \sim M_{\neq}}
    \sum_{\substack{h_1, h_2 \sim H \\ h_1e_1'\ell_2 - h_2e_2'\ell_1 = m}}
    \left\vert 
    \sum_{\substack{k' \\ (k', r) = 1}} g_0 \left(\frac{k'}{K/S}\right)
    \sum_{|j| \sim J} \e\left(j \omega\right)\, 
    S\left(m \bar{r}, j; k's\right)
    \right\vert.
\end{aligned}
\]
As before, once the variables $t, e_0, e_1', e_2', \hat{\mu}$ are fixed, $\omega$ does not depend on $r, s, m, h_1, h_2, k'$.
We remove the absolute values by inserting $1$-bounded coefficients $\xi_{h_1,h_2}$ (also depending on $t, e_0, e_1', e_2', \hat{\mu}$ and $r, s, m$), and denote
\[
    a_{m,r,s} = a_{m,r,s}\left(t, e_0, e_1', e_2', \hat{\mu}\right) := 
    \one_{b e_1'e_2' \mid r}
    \one_{s \mid te_0}
    \sum_{\substack{\ell_0 \sim L_0 \\ \ell_1', \ell_2' \sim L/\ell_0 \\ (\ell_1', \ell_2') = 1, \ell_1' \neq \ell_2' \\ b\ell_0 \ell_1' \ell_2' e_1' e_2' = r \\ \ell_1' \equiv \ell_2' \pmod{te_0} \\ (\ell_1', e_1') = (\ell_2', e_2') = 1 \\ \mu(\ell_0\ell_1', \ell_0\ell_2') \equiv \hat{\mu} \\ \pmod{te_0e_1'e_2'}}}
    \sum_{\substack{h_1, h_2 \sim H \\ h_1e_1'\ell_2' - h_2e_2'\ell_1' = m}} 
    \xi_{h_1, h_2},
\]
to obtain
\begin{equation} \label{eq:bneqprime-bound}
    \mB_{\neq}' \ll_\eta x^{O(\eta)} 
    \sum_{t \sim T} \sum_{e_0 \sim E_0}\,
    \sum_{\substack{e_1', e_2' \sim E/e_0 \\ (e_1', e_2') = 1}}\,
    \sum_{\hat{\mu} \in \Z/te_0e_1'e_2'\Z}
    |\mK_{\neq}|,
\end{equation}
where
\[
    \mK_{\neq} := \sum_{\substack{r \sim R_{\neq} \\ s \sim S \\ (r, s) = 1}} 
    \sum_{m \sim M_{\neq}} a_{m,r,s}
    \sum_{|j| \sim J} \e(j \omega)
    \sum_{\substack{k' \\ (k', r) = 1}} 
    g_0 \left(\frac{k'}{K/S}\right) S\left(m\bar{r}, j; k's\right).
\]
This is roughly the sum of Kloosterman sums anticipated in \cref{eq:sketch-kloosterman-2}.
By \cref{thm:di-type-bound}, we have
\[
\begin{aligned}
    \mK_{\neq} 
    \ll_\eta\ &x^{O(\eta)}
    \left(1 + \frac{K/S}{R_{\neq} \sqrt{S}}\right)^{\theta_{\max}}
    \|a_{m,r,s}\|_2\, \sqrt{J R_{\neq} S}
    \\
    &\times 
    \left(\frac{K^2/S^2}{R_{\neq}} \left(M_{\neq} + R_{\neq} S\right)\left(J + R_{\neq} S\right) + M_{\neq}J \right)^{1/2},
\end{aligned}
\]
where by Cauchy--Schwarz,
\[
\begin{aligned}
    &\sum_{t \sim T} \sum_{e_0 \sim E_0}\,
    \sum_{\substack{e_1', e_2' \sim E/e_0 \\ (e_1', e_2') = 1}}\,
    \sum_{\hat{\mu} \in \Z/(te_0e_1'e_2'\Z)} \|a_{m,r,s}\|_2
    \\ 
    &\ll 
    \frac{TE^2}{E_0} \sqrt{
    \sum_{t \sim T} \sum_{e_0 \sim E_0}\,
    \sum_{\substack{e_1', e_2' \sim E/e_0 \\ (e_1', e_2') = 1}} \sum_{\substack{s \sim S \\ s \mid te_0}} \sum_{\substack{\ell_0 \sim L_0 \\ \ell_1', \ell_2' \sim L/\ell_0 \\ (\ell_1', \ell_2') = 1, \ell_1' \neq \ell_2' \\ \ell_1' \equiv \ell_2' \pmod{te_0} \\ (\ell_1', e_1') = (\ell_2', e_2') = 1}} 
    \sum_{m \sim M_{\neq}} 
    \left(\sum_{\substack{h_1, h_2 \sim H \\ h_1 e_1' \ell_2' - h_2 e_2' \ell_1' = m}} 1\right)^2
    } 
    \\
    &\ll \frac{TE^2}{E_0} \sqrt{
    L_0
    \sum_{e_1', e_2' \ll E/E_0}
    \sum_{\substack{\ell_1', \ell_2' \asymp L/L_0 \\ (\ell_1'e_2', \ell_2'e_1') = 1 \\ \ell_1' \neq \ell_2'}}
    \tau(\ell_1' - \ell_2')^3
    \sum_{m \sim M_{\neq}}
    \sum_{\substack{h_1, h_2 \sim H \\ h_1 e_1' \ell_2' - h_2 e_2' \ell_1' = m}} 
    \sum_{\substack{h_1', h_2' \sim H \\ h_1' e_1' \ell_2' - h_2' e_2' \ell_1' = m}} 1
    }
    \\
    &\ll_\eta
    x^{O(\eta)} \frac{TE^2}{E_0} \sqrt{L_0
    \sum_{m \sim M_{\neq}}\,
    \sum_{\substack{h_1 \sim H \\ e_1' \ll E/E_0 \\ \ell_2' \ll L/L_0 }}
    \sum_{\substack{h_2, e_2', \ell_1' \\ h_2e_2'\ell_1' = h_1 e_1' \ell_2' - m \\ (e_2'\ell_1', e_1'\ell_2') = 1}}
    \ \sum_{\substack{h_1' \sim H \\ h_1'e_1'\ell_2' \equiv m \pmod{e_2'\ell_1'}}} 
    \ \sum_{\substack{h_2' \sim H \\ h_2'e_2'\ell_1' = h_1'\ell_2'e_1' - m}} 1
    }
    \\
    &\ll_\eta 
    x^{O(\eta)} \frac{TE^2}{E_0} \left(L_0 M_{\neq} H \frac{E}{E_0} \frac{L}{L_0} \left(1 + \frac{HE_0L_0}{EL}\right)\right)^{1/2}
    \\
    &\ll 
    x^{O(\eta)} \left(\frac{T^2 E^5 M_{\neq} H L}{E_0^2} \left(1 + \frac{H L_0}{EL}\right)\right)^{1/2}.
\end{aligned}
\]
Plugging these bounds into \cref{eq:bneqprime-bound}, we find that
\[
\begin{aligned}
    \mB_{\neq}' \ll_\eta\ &x^{O(\eta)}
    \left(1 + \frac{K/S}{R_{\neq} \sqrt{S}}\right)^{\theta_{\max}}
    \left(\frac{T^2 E^5 M_{\neq} H L}{E_0^2} \left(1 + \frac{H L_0}{EL}\right)\right)^{1/2} \left(J R_{\neq} S\right)^{1/2}
    \\
    &\times 
    \left(\frac{K^2/S^2}{R_{\neq}} \left(M_{\neq} + R_{\neq} S\right)\left(J + R_{\neq} S\right) + M_{\neq}J \right)^{1/2}.
\end{aligned}
\]
Recalling that $J \ll (K T E^2 x^\eta)/(NE_0)$ (from \cref{lem:splitting-bfi}), $M_{\neq} \ll HEL/(E_0 L_0)$ and $R_{\neq} \asymp x^{O(\eta)} L^2 E^2 / (L_0 E_0^2)$ (from \cref{eq:bneq-bound}), this yields
\[
\begin{aligned}
    \mB_{\neq}' &\ll_\eta x^{O(\eta)} \left(1 + \frac{K L_0 E_0^2}{L^2 E^2 S^{3/2}}\right)^{\theta_{\max}}
    \left(\frac{T^2 E^5 HEL H L}{E_0^2 E_0 L_0} \left(1 + \frac{H L_0}{EL}\right)\right)^{1/2} 
    \left(\frac{KTE^2}{NE_0} \frac{L^2 E^2 }{L_0 E_0^2 } S \right)^{1/2}
    \\
    &\times
    \left(\frac{K^2 L_0 E_0^2}{L^2 E^2 S^2} \left(\frac{HEL}{E_0L_0} + \frac{L^2 E^2}{L_0 E_0^2} S\right)\left(\frac{KTE^2}{NE_0} + \frac{L^2 E^2}{L_0 E_0^2}S\right) + \frac{HE^3L K T}{E_0^2L_0 N}\right)^{1/2}.
\end{aligned}
\]
Note that this expression is nonincreasing in the GCD parameter $L_0$, since $\theta_{\max} \le 1/2$; thus lower-bounding $L_0 \gg 1$, and then using that $1 \ll S \ll TE_0$ (from \cref{lem:splitting-bfi}), we get
\[
\begin{aligned}
    \mB_{\neq}' &\ll_\eta x^{O(\eta)} \left(1 + \frac{K E_0^2}{L^2 E^2 S^{3/2}}\right)^{\theta_{\max}}
    \left(\frac{T^2 E^6 H^2 L^2}{E_0^3} \left(1 + \frac{H}{EL}\right)\right)^{1/2} 
    \left(\frac{KTE^2}{NE_0} \frac{L^2 E^2}{E_0^2} \right)^{1/2}
    \\
    &\times
    \left(\frac{K^2 E_0^2}{L^2 E^2} \left(\frac{HEL}{E_0 S} + \frac{L^2 E^2}{E_0^2}\right)\left(\frac{KTE^2}{NE_0} + \frac{L^2 E^2}{E_0^2}S\right) + \frac{HE^3L K T}{E_0^2 N} S\right)^{1/2}
    \\
    &\ll x^{O(\eta)} \left(1 + \frac{K E_0^2}{L^2 E^2}\right)^{\theta_{\max}}
    \left(\frac{T^2 E^6 H^2 L^2}{E_0^3} \left(1 + \frac{H}{EL}\right)\right)^{1/2} 
    \left(\frac{KTE^2}{NE_0} \frac{L^2 E^2}{E_0^2} \right)^{1/2}
    \\
    &\times
    \left(\frac{K^2 E_0^2}{L^2 E^2} \left(\frac{HEL}{E_0} + \frac{L^2 E^2}{E_0^2}\right)\left(\frac{KTE^2}{NE_0} + \frac{L^2 E^2 T}{E_0}\right) + \frac{HE^3L K T^2}{E_0 N} \right)^{1/2}.
\end{aligned}
\]
Finally, this expression is nonincreasing in the $E_0$ parameter even after extracting a factor of $E_0^{-1}$, so lower-bounding $E_0 \gg 1$ yields
\[
\begin{aligned}
    \mB_{\neq}' &\ll_\eta \frac{x^{O(\eta)}}{E_0}
    \left(1 + \frac{K}{L^2 E^2}\right)^{\theta_{\max}}
    \left(1 + \frac{H}{EL}\right)^{1/2} 
    \left(T^2 E^6 H^2 L^2 \frac{KTE^2}{N} L^2 E^2\right)^{1/2}
    \\
    &\times
    \left(\frac{K^2}{L^2 E^2} \left(HEL + L^2 E^2\right)\left(\frac{KTE^2}{N} + L^2 E^2 T\right) + \frac{HE^3L K T^2}{N} \right)^{1/2}.
\end{aligned}
\]
Due to our assumption that $EHT \ll x^{O(\eta)} KNL$, we have
\[
    x^{-O(\eta)}\frac{HE^3LKT^2}{N} \ll  K^2 L^2 E^2 T
    =
    \frac{K^2}{L^2E^2} L^2E^2 L^2 E^2 T,
\]
so we may ignore the term of $HE^3LKT^2/N$ on the second line to obtain
\[
\begin{aligned}
    \mB_{\neq}' &\ll_\eta \frac{x^{O(\eta)}}{E_0}
    \left(1 + \frac{K}{L^2 E^2}\right)^{\theta_{\max}}
    \left(1 + \frac{H}{EL}\right)^{1/2} 
    \left(\frac{KT^3 E^{10} H^2 L^4}{N}\right)^{1/2}
    \\
    &\times
    K \left(\frac{H}{LE} + 1\right)^{1/2} \left(L^2 E^2 T\right)^{1/2} \left(\frac{K}{NL^2} + 1\right)^{1/2}
    \\
    &\ll \frac{x^{O(\eta)}}{E_0} 
    \left(1 + \frac{K}{L^2 E^2}\right)^{\theta_{\max}}
    \left(\frac{K^3 T^4 E^{12} H^2 L^6}{N}\right)^{1/2}
    \left(1 + \frac{H}{LE}\right) \left(1 + \frac{K}{NL^2}\right)^{1/2}
    \\
    &\ll \frac{x^{O(\eta)}}{E_0} \frac{KE^2}{N} 
    \left(1 + \frac{K}{L^2 E^2}\right)^{\theta_{\max}}
    \left(K T^4 E^8 H^2 L^6 N\right)^{1/2}
    \left(1 + \frac{H}{LE}\right) \left(1 + \frac{K}{NL^2}\right)^{1/2}.
\end{aligned}
\]
Rearranging factors (and combining this with \cref{eq:bneq-bound}), we obtain the desired bound.
\end{proof}

Combining our results so far, we obtain the following general estimate.

\begin{proposition}[BFI-style bound with general parameters] \label{prop:bfi-expo-2}
For $\eta \in (0, 1)$, $K, N, T, E, H, L \ll x$, and any positive integers $b \ll x^{O(\eta)}$ and $d \ll x$, assuming that $EHT \ll x^{O(\eta)} KNL$, one has
\[
\begin{aligned}
    \mB(K, N, T, E, H, L)^2 \ll_\eta\,
    &x^{O(\eta)} \Bigg( N^2 K^2 H^2 T^2 L^2 + K^2 E^2 H^2 T^4 L^2 + N^2 H^4 L^4 + N^2 H^4 T^2 L^2
    \\
    &+ \left(1 + \frac{K^2}{E^4 L^4}\right)^{\theta_{\max}}
    K T^4 E^8 H^2 L^6 N
    \left(1 + \frac{H^2}{E^2L^2}\right) \left(1 + \frac{K}{NL^2}\right)
    \\
    &+
    \left(1 + \frac{K^2}{E^4L^2}\right)^{\theta_{\max}}
    K T^5 E^8 H^2 L^3 N
    \left(1 + \frac{H}{E} + \frac{H^2}{E^2L}\right) \left(1 + \frac{K}{NL} \right)
    \Bigg).
\end{aligned}
\]
\end{proposition}

\begin{proof}[Proof of \cref{prop:bfi-expo-2} assuming \cref{thm:di-type-bound}]
This follows by putting together \cref{lem:splitting-bfi,lem:split-bfi-neq,lem:split-bfi-eq} and squaring (the second line comes from \cref{lem:split-bfi-neq}, and the third line from \cref{lem:split-bfi-eq}).
\end{proof}

Finally, we use \cref{prop:bfi-expo-2} and the conditions from \cref{eq:conditions} to prove \cref{prop:bfi-expo}.

\begin{proof}[Proof of \cref{prop:bfi-expo} assuming \cref{thm:di-type-bound}]
Let $\theta := \theta_{\max}$; we will soon pick a value for $E$ such that \cref{eq:E-conditions} holds. We can assume without loss of generality that $K', N', T', E, H', L' \gg 1$, since otherwise the sum in \cref{prop:bfi-expo} is void. 
We now apply the bound in \cref{prop:bfi-expo-2} (which is increasing in all six parameters) for the parameters $K', N', T', E, H', L'$ from \cref{prop:bfi-expo}, noting that 
\[
\begin{aligned}
    EH'T' &\ll x^{O(\eta)} ERM^{-1} NL (ER)^{-1} 
    \\
    &= x^{O(\eta)} NL M^{-1} 
    \ll x^{O(\eta)} NL
    \ll x^{O(\eta)} K' N' L'.
\end{aligned}
\]
Plugging in the bounds $K' \ll NL x^{O(\eta)}$, $N' \asymp N x^{O(\eta)}$, $L' \asymp L x^{O(\eta)}$, $T' \ll NL(RE)^{-1} x^{O(\eta)}$, $H' \ll RNL x^{O(\eta)-1}$, we obtain
\[
\begin{aligned}
    &\mB(K', N', T', E, H', L')^2 \ll_\eta x^{O(\eta)}
    \\
    &\quad \times \Bigg( N^2 (NL)^2 \left(\frac{RNL}{x}\right)^2 \left(\frac{NL}{RE}\right)^2 L^2 + (NL)^2 E^2 \left(\frac{RNL}{x}\right)^2 \left(\frac{NL}{RE}\right)^4 L^2 
    \\
    &\quad + N^2 \left(\frac{RNL}{x}\right)^4 L^4 + N^2 \left(\frac{RNL}{x}\right)^4 \left(\frac{NL}{RE}\right)^2 L^2
    \\
    &\quad + \left(1 + \frac{(NL)^2}{E^4L^4}\right)^\theta NL \left(\frac{NL}{RE}\right)^4 E^8 \left(\frac{RNL}{x}\right)^2 L^6 N \left(1 + \frac{\left(\frac{RNL}{x}\right)^2}{E^2 L^2}\right)
    \\
    &\quad + \left(1 + \frac{(NL)^2}{E^4L^2} \right)^\theta NL \left(\frac{NL}{RE}\right)^5 E^8 \left(\frac{RNL}{x}\right)^2 L^3 N \left(1 + \frac{\left(\frac{RNL}{x}\right)}{E} + \frac{\left(\frac{RNL}{x}\right)^2}{E^2 L}\right).
\end{aligned}
\]
Simplifying terms and dividing both sides by $N^4 L^6$, we further get
\[
\begin{aligned}
    \frac{\mB(K', N', T', E, H', L')^2}{N^4 L^6} \ll_\eta 
    x^{O(\eta)} &\Bigg( \frac{N^4 L^2}{x^2 E^2} + \frac{N^4 L^4}{R^2 x^2 E^2} + \frac{N^2 L^2 R^4}{x^4} + \frac{N^4 L^2 R^2}{x^4 E^2}
    \\
    &+ \left(1 + \frac{N^2}{E^4 L^2}\right)^{\theta} \frac{N^4 L^7 E^4}{x^2 R^2} \left(1 + \frac{N^2 R^2}{x^2 E^2}\right)
    \\
    &+ \left(1 + \frac{N^2}{E^4} \right)^{\theta} \frac{N^5 L^5 E^3}{x^2 R^3} \left( 
    1 + \frac{NLR}{xE} + \frac{N^2 L R^2}{x^2 E^2}
    \right) \Bigg),
\end{aligned} 
\]
and we wish to show that the right-hand side is $\ll x^{O(\eta) - \eps}$. To handle the term of $N^4 L^2 x^{-2} E^{-2}$, we require that $N^2 L \ll x^{1-\eps} E$; thus we pick
\begin{equation} \label{eq:E-choice}
    E := \max\left( x^{4\eps}, \frac{N^2 L}{x^{1-\eps}} \right).
\end{equation}
For \cref{eq:E-conditions} to hold, we also need to have $E \ll x^{-\eps} NL R^{-1}$, so we impose the restrictions
\[
    R \ll x^{-5\eps} NL \qquad\quad \text{and} \qquad\quad N \ll \frac{x^{1-2\eps}}{R}
\]
(which are part of \cref{eq:conditions}). The fact that $NR \ll x$ simplifies our expression a bit; combined with the fact that $E \ll x^{-\eps} N L R^{-1} \ll NL R x^{-1}$ (due to $x^{1-\eps} \ll R^2$ from \cref{eq:conditions}), this shows that
\[
    \frac{N^2 R^2}{x^2 E^2} \ll 1 \qquad\quad \text{and} \qquad\quad
    \max\left(\frac{N^2 L R^2}{x^2 E^2}, 1\right) \ll \frac{N L R}{x E}.
\]
Moreover, since $x^{(1-\eps)/2} \ll R$ by \cref{eq:conditions}, we have $NR \ll x^{1-2\eps} \ll R^2$, so $N \ll R$, which implies
\[
    \frac{N^4 L^2 R^2}{x^4 E^2} \ll \frac{N^2 L^2 R^4}{x^4}.
\]
Overall, it remains to bound the expression
\begin{equation} \label{eq:final-to-bound}
    \frac{N^4 L^4}{R^2 E^2 x^2} + \frac{N^2 L^2 R^4}{x^4}
    + \left(1 + \frac{N^2}{E^4 L^2}\right)^{\theta} \frac{N^4 L^7 E^4}{x^2 R^2}
    + \left(1 + \frac{N^2}{E^4}\right)^{\theta} \frac{N^6 L^6 E^2}{x^3 R^2}
\end{equation}
by $O(x^{-\eps})$.

Using $x^{(1-\eps)/2} \ll R \ll NL \ll x^{2/3 - 6\eps}$ (from \cref{eq:conditions}) and $E \ge x^{4\eps}$, the first term is admissible since
\[
    \frac{N^2 L^2}{REx} \ll \frac{x^{4/3}}{x^{3\eps} x^{3/2}}
    =
    x^{-1/6 - 3\eps}.
\]
The (square root of the) second term is similarly bounded:
\[
    \frac{NL R^2}{x^2} \ll x^{(2/3) + (4/3) - 2 - 6\eps} = x^{- 6\eps}. 
\]
For the third term in \cref{eq:final-to-bound}, we use our choice of $E$ from \cref{eq:E-choice} to obtain
\[
    \left(1 + \frac{N^2}{E^4 L^2}\right)^{\theta} \frac{N^4 L^7 E^4}{x^2 R^2}
    \ll 
    \left(1 + \frac{N^2}{L^2}\right)^{\theta} \frac{N^4 L^7}{x^{2-16\eps} R^2}
    +
    \left(1 + \frac{x^4}{N^6 L^6}\right)^{\theta} \frac{N^{12} L^{11}}{x^{6-4\eps} R^2}.
\]
Since $NL \ll x^{2/3}$ by \cref{eq:conditions}, we can ignore the $1$-term in the last parenthesis. For the terms above to be admissible, we require the restrictions
\[
    N^4 L^7 \max(1, N/L)^{2\theta} \ll x^{2-17\eps}R^2, 
    \qquad\quad 
    N^{12-6\theta} L^{11-6\theta} \ll x^{6-4\theta -5\eps} R^2
\]
(which are part of \cref{eq:conditions}). Finally, using that $1 \ll E \ll x^{-\eps} NLR^{-1}$ (from \cref{eq:E-conditions}), we crudely bound the fourth and last term in \cref{eq:final-to-bound} by
\[
    \left(1 + \frac{N^2}{E^4}\right)^{\theta} \frac{N^6 L^6 E^2}{x^3 R^2}
    \ll 
    N^{2\theta} \frac{N^6 L^6 (x^{-\eps}NLR^{-1})^2}{x^3 R^2}
    \ll
    x^{-2\eps} \frac{N^9 L^8}{x^3 R^4},
\]
which is at most $O(x^{-\eps})$ by the last condition in the first line of \cref{eq:conditions}. This completes our proof.
\end{proof}


\section{Deshouillers--Iwaniec-style estimates} \label{sec:deshouillers-iwaniec}

The seminal work \cite{deshouillers1982kloosterman} of Deshouillers--Iwaniec on sums of Kloosterman sums makes repeated use of the Kuznetsov trace formula \cite{kuznetsov1980petersson,motohashi1997spectral}, which is in turn based on the spectral decomposition of $L^2(\Gamma_0(q) \backslash \H)$ with respect to the hyperbolic Laplacian (where $q$ is a positive integer and $\Gamma_0(q)$ is its associated Hecke congruence subgroup). Here we prove \cref{thm:di-type-bound}, which is an optimization of \cite[Theorem 11]{deshouillers1982kloosterman} in the $\theta$-aspect, using the same technology. We note that such optimizations of Deshouillers--Iwaniec bounds (specifically of \cite[Theorem 12]{deshouillers1982kloosterman}) have also been used in \cite{drappeau2023one}.

We will use all of the notation (and normalization) from \cite{deshouillers1982kloosterman}, with the exception of making some dependencies on the level $q$ explicit. In particular, we consider an orthonormal basis of Maass cusp forms $(u_{j,q})_{j \ge 1}$ such that $u_{j,q}$ has eigenvalue $\lambda_{j,q}$ (which increases to $\infty$ as $j \to \infty$), and Fourier coefficients $\rho_{j,\ma}(n)$ when expanding around the cusp $\ma$ of $\Gamma_0(q)$, via an implicit scaling matrix $\sigma_\ma \in \PSL_2(\R)$. We denote 
\[
    \mu(\ma) := \frac{\left(w, \frac{q}{w}\right)}{q},
\]
whenever $\ma$ is equivalent to $u/w$, for some relatively prime $u, w \in \Z_+$ such that $w \mid q$; in particular, one has $\mu(\infty) = q^{-1}$. We also write
\[
    \theta_{j,q} := 2i \kappa_{j,q},
    \qquad\qquad
    \kappa_{j,q}^2 = \lambda_{j,q} - \frac{1}{4}
    \qquad \iff \qquad
    \theta_{j,q}^2 = 1 - 4\lambda_{j,q},
\]
where $\kappa_{j,q}$ is chosen such that either $\kappa_{j,q} \ge 0$ (when $\lambda_{j,q} \ge 1/4$), or $i\kappa_{j,q} > 0$ (when $\lambda_{j,q}$ is exceptional). Recall from \cref{not:exceptional} that $\theta_q := \max_{\lambda_j < 1/4} \theta_{j,q}$ (with $\theta_q := 0$ if there are no exceptional eigenvalues), and that $\theta_{\max} := \sup_q \theta_q$. Also, recall that all exceptional eigenvalues lie in the interval $[3/16, 1/4)$ by \cite[Theorem 4]{deshouillers1982kloosterman} (in fact, the best currently known lower bound is $975/4096$, due to Kim--Sarnak \cite[Appendix 2]{kim2003functoriality}; this is equivalent to \cref{thm:kim-sarnak}).

The contribution of the exceptional Maass forms to the spectral side of the Kuznetsov trace formula would vanish if Selberg's eigenvalue conjecture (\cref{conj:selberg}) were true, but would be dominating in most applications otherwise. To deduce better bounds for the geometric side (which consists of weighted sums of Kloosterman sums), Deshouillers--Iwaniec \cite{deshouillers1982kloosterman} proved a series of large sieve inequalities for the Fourier coefficients of Maass cusp forms, which temper this exceptional contribution in bilinear sums. Remarkably, these results make further use of the Kuznetsov formula, applying it back and forth and ultimately reducing to the Weil bound.

\begin{lemma}[Large sieve inequalities from \cite{deshouillers1982kloosterman}] \label{lem:di-large-sieve}
Given $\eps > 0$, $q \in \Z_+$, $N \gg 1$, a complex sequence $(a_n)_{n \sim N}$, a cusp $\ma$ of $\Gamma_0(q)$, and an associated scaling matrix $\sigma_{\ma}$, one has
\begin{equation} 
    \label{eq:like-di-thm-5}
    \sum_{\substack{j \ge 1 \\ \lambda_{j,q} < 1/4}}
    X^{\theta_{j,q}} 
    \left\vert 
    \sum_{n \sim N} a_n\, \rho_{j,\ma}(n)
    \right\vert^2 
    \ll_\eps 
    (QN)^\eps
    \left(1 + \mu(\ma) N\right) \|a_n\|_2^2,
\end{equation}
for any $0 < X \ll \max(1, \mu(\ma)^{-1}N^{-1})$.
Moreover, if $(\ma, \sigma_{\ma}) = (\infty, \Id)$, then given $Q \gg 1$ and $\alpha \in \R/\Z$, one has
\begin{equation}
    \label{eq:like-di-thm-7}
    \frac{1}{Q} \sum_{q \sim Q} 
    \sum_{\substack{j \ge 1 \\ \lambda_{j,q} < 1/4}}
    X^{\theta_{j,q}}
    \left\vert \sum_{n \sim N} \e(n\omega)\, \rho_{j,\infty}(n) \right\vert^2
    \ll_\eps 
    (QN)^\eps \left(1 + Q^{-1}N\right) N,
\end{equation}
in the larger range $0 < X \ll \max\left(N, Q^2 N^{-1}\right)$.
\end{lemma}

\begin{proof}
The bounds in \cref{eq:like-di-thm-5,eq:like-di-thm-7} follow immediately from \cite[Theorems 5 and 7]{deshouillers1982kloosterman} respectively. We note that changing the choice of the scaling matrix $\sigma_{\ma}$ results in multiplying the Fourier coefficients $\rho_{j,\ma}(n)$ by an exponential phase $\e(n\omega)$; thus in \cref{eq:like-di-thm-7}, using an arbitrary value of $\omega$ is equivalent to using an arbitrary (but consistent) choice of the scaling matrix $\sigma_\infty$.

We also remark that the proof of \cite[Theorem 7]{deshouillers1982kloosterman} from \cite[Section 8.3]{deshouillers1982kloosterman} only considers the case $\omega = 0$ (and $\sigma_\infty = \Id$), but the same proof extends to any $\omega \in \R/\Z$ (or equivalently, to any valid scaling matrix $\sigma_\infty$); this was already noted, for instance, in \cite[Lemma 5]{bombieri1987primes2}. Ultimately, this is because the proof of \cite[Theorem 14]{deshouillers1982kloosterman} also extends to sums with additional weights of $\e(m\omega_1)\, \e(n\omega_2)$.
\end{proof}

\begin{remark}
The large sieve inequalities in \cite{deshouillers1982kloosterman} are stated for general values of $X$ in the left-hand sides (resulting in right-hand sides that depend on $X$), and are equivalent to those given in \cref{lem:di-large-sieve}. Indeed, to recover large sieve inequalities with an arbitrary $X > 0$ in the left-hand sides, it suffices to multiply the right-hand sides by $(1 + (X/X_0)^{\theta_q})$, where $X_0$ is the best allowable value in \cref{lem:di-large-sieve}.

We find the versions stated above easier to apply optimally in the $\theta$-aspect, and also easier to compare, by contrasting the maximal permitted values of $X$ (recalling that $\mu(\infty) = q^{-1}$).
\end{remark}

We now adapt the proof of \cite[Theorem 11]{deshouillers1982kloosterman}, making the dependence on $\theta_{\max}$ explicit.

\begin{theorem}[\cite{deshouillers1982kloosterman}-type multilinear Kloosterman bound] \label{thm:di-kloosterman}
Let $C, M, N, R, S \gg 1$, $(b_{n,r,s})$ be a complex sequence, and $\omega \in \R/\Z$. Then given a 5-variable smooth function $g(t_1, \ldots, t_5)$ with compact support in $t_1 \asymp 1$, and bounded derivatives $\|\frac{\partial^{\sum j_i}}{\prod (\partial t_i)^{j_i}} g\|_\infty \ll_{j_1,\ldots,j_5} 1$, one has
\begin{equation} \label{eq:like-di-thm-11}
\begin{aligned}
    \sum_{\substack{r \sim R \\ s \sim S \\ (r, s) = 1}} 
    \sum_{\substack{m \sim M \\ n \sim N}} \e(m\omega)\, b_{n,r,s}
    &\sum_{(c, r) = 1} 
    g\left(\frac{c}{C}, \frac{m}{M}, \frac{n}{N}, \frac{r}{R}, \frac{s}{S}\right)\, S(m\bar{r}, \pm n; sc)
    \\
    \ll_\eps
    (CMNRS)^\eps 
    &\left(1 + \frac{CS\sqrt{R}}{\max(M, RS) \sqrt{\max(N, RS)}} \right)^{\theta_{\max}} 
    \sqrt{MRS}\, \|b_{n,r,s}\|_2\,
    \\
    &\times
    \frac{\left(CS\sqrt{R} + \sqrt{MN} + C\sqrt{SM}\right)\left(CS\sqrt{R} + \sqrt{MN} + C\sqrt{SN}\right)}{CS\sqrt{R} + \sqrt{MN}}.
\end{aligned}
\end{equation}
\end{theorem}

\begin{proof}
We follow the proof of \cite[Theorem 11]{deshouillers1982kloosterman} in \cite[Section 9.1]{deshouillers1982kloosterman}, reducing to the case of smooth functions of the form $\frac{CS\sqrt{R}}{cs\sqrt{r}} f \left(\frac{4\pi \sqrt{mn}}{cs\sqrt{r}}\right)$ (up to using slightly different values of $\omega$ and $b_{n,r,s}$); here $f(t)$ is a smooth function supported in $t \asymp X^{-1}$, for $X := CS\sqrt{R}/\sqrt{MN}$. After applying the Kuznetsov formula, we bound the contribution of the exceptional spectrum more carefully; as in \cite[Section 9.1]{deshouillers1982kloosterman}, this is given by
\[
    \mS_{\text{exc}} := CS\sqrt{R}
    \sum_{\substack{r \sim R \\ s \sim S \\ (r, s) = 1}}
    \sum_{\substack{j \ge 1 \\ \lambda_{j,rs} < 1/4}} \frac{\hat{f}(\kappa_{j,rs})}{\ch(\pi \kappa_{j,rs})}
    \left( \sum_{m \sim M} \e(m\omega)\, \bar{\rho_{j,\infty}}(m) \right)
    \left( \sum_{n \sim N} b'_{n,r,s}\, \rho_{j,1/s}(n)
    \right),
\]
where
\[
    b'_{n,r,s} := \e\left(-n \frac{\bar{s}}{r}\right) b_{n,r,s}.
\]
Using the bounds $\ch(\pi \kappa_{j,rs}) \asymp 1$ and 
\[
    |\hat{f}(\kappa_{j,rs})| \ll \frac{1 + X^{2|\kappa_{j,rs}|}}{1 + X^{-1}} 
    \ll 
    \frac{(1 + X)^{\theta_{j,rs}}}{1+X^{-1}}
\]
(see \cite[(7.1)]{deshouillers1982kloosterman}), and denoting
\[
    X_0 := \frac{1 + X}{\sqrt{X_1 X_2}}
    \le 
    1 + \frac{X}{\sqrt{X_1 X_2}},
\]
for some $X_1, X_2 \ge 1$ to be chosen shortly, we obtain
\[
\begin{aligned}
    \mS_{\text{exc}} 
    \ll
    CS\sqrt{R}
    \sum_{\substack{r \sim R \\ s \sim S \\ (r, s) = 1}}
    \sum_{\substack{j \ge 1 \\ \lambda_{j,rs} < 1/4}} \frac{\left(X_0 \sqrt{X_1 X_2}\right)^{\theta_{j,rs}}}{1+X^{-1}}
    \left\vert \sum_{m \sim M} \bar{\e(m\omega)}\, \rho_{j,\infty}(m) \right\vert
    \left\vert \sum_{n \sim N} b'_{n,r,s}\, \rho_{j,1/s}(n)
    \right\vert
    \\
    \ll
    CS\sqrt{R}\
    \frac{\left(1 + X_0\right)^{\theta_{\max}}}{1 + X^{-1}}
    \left(\sum_{\substack{r \sim R \\ s \sim S}} \sum_{\substack{j \ge 1 \\ \lambda_{j,rs} < 1/4}} X_1^{\theta_{j,rs}} \left\vert \sum_{m \sim M} \e(-m\omega)\, \rho_{j,\infty}(m) \right\vert^2 \right)^{1/2}
    \\
    \times 
    \left(\sum_{\substack{r \sim R \\ s \sim S \\ (r, s) = 1}} \sum_{\substack{j \ge 1 \\ \lambda_{j,rs} < 1/4}} X_2^{\theta_{j,rs}} \left\vert \sum_{n \sim N} b'_{n,r,s}\, \rho_{j,1/s}(n) \right\vert^2 \right)^{1/2},
\end{aligned}
\]
by Cauchy--Schwarz. Recall that $\mu(1/s) = \mu(\infty) = (rs)^{-1}$ since $(r, s) = 1$; thus using the divisor bound and \cref{lem:di-large-sieve}, we conclude that
\[
\begin{aligned}
    \mS_{\text{exc}} \ll_\eps 
    (MNRS)^\eps\,
    CS\sqrt{R}\, \frac{\left(1 + X_0\right)^{\theta_{\max}}}{1 + X^{-1}} 
    \sqrt{RS} 
    \left(1 + \sqrt{\frac{M}{RS}}\right) \sqrt{M}
    \left(1 + \sqrt{\frac{N}{RS}}\right) \|b_{n,r,s}\|_2
    \\
    \ll
    (MNRS)^\eps \left(1 + \frac{CS\sqrt{R}}{\sqrt{MNX_1X_2}}\right)^{\theta_{\max}} \sqrt{MRS}\, \|b_{n,r,s}\|_2
    \\
    \times
    \frac{\left(CS\sqrt{R} + C\sqrt{SM}\right)\left(CS\sqrt{R} + C\sqrt{SN}\right)}{CS\sqrt{R} + \sqrt{MN}},
\end{aligned}
\]
for $X_1 = \max(M, R^2S^2M^{-1})$ (coming from \cref{eq:like-di-thm-7}), and $X_2 = \max(1, RSN^{-1})$ (from \cref{eq:like-di-thm-5}), which gives the desired bound up to minor rearrangements.  As in \cite[(9.4)]{deshouillers1982kloosterman}, the non-exceptional spectrum contributes a similar amount of
\[
\begin{aligned}
    \ll_\eps 
    (MNRS)^\eps \sqrt{MRS}\, \|b_{n,r,s}\|_2
    \frac{\left(CS\sqrt{R} + \sqrt{MN} + C\sqrt{SM}\right)\left(CS\sqrt{R} + \sqrt{MN} + C\sqrt{SN}\right)}{CS\sqrt{R} + \sqrt{MN}},
\end{aligned}
\]
and putting these together completes our proof.
\end{proof}

Finally, \cref{thm:di-type-bound} follows almost immediately from \cref{thm:di-kloosterman}.

\begin{proof}[Proof of \cref{thm:di-type-bound}]
We swap the $m$ and $n$ variables, and pick the second term in each maximum from \cref{eq:like-di-thm-11} for an upper bound, resulting in a $\theta$-factor of
\[
    \left(1 + \frac{C}{R\sqrt{S}} \right)^{\theta_{\max}}.
\]
We also rewrite the last fraction in \cref{eq:like-di-thm-11} as
\[
    CS\sqrt{R} + \sqrt{MN} + C\sqrt{SM} + C\sqrt{SN} + \frac{C^2 S \sqrt{MN}}{CS\sqrt{R} + \sqrt{MN}},
\]
and we use the lower bound $CS\sqrt{R} + \sqrt{MN} \ge CS\sqrt{R}$ in the final term. 

To reduce to a smooth function depending only on $c$, we can take
\[
    g(t_1, t_2, t_3, t_4, t_5) = 
    g_1\left(t_1\right) 
    g_2\left(t_2\right)
    g_3\left(t_3\right)
    g_4\left(t_4\right)
    g_5\left(t_5\right),
\]
for some smooth compactly supported functions $g_i$, where $g_2, g_3, g_4, g_5$ are equal to $1$ on $[1, 2]$.
\end{proof}

\textbf{Acknowledgements.}
The author wishes to thank his advisor, James Maynard, for his kind support and guidance, as well as Sary Drappeau, Lasse Grimmelt, R\'egis de la Bret\`eche, and the referees, for many helpful comments and suggestions. During the time of this work, the author was sponsored by the EPSRC Scholarship at the University of Oxford.

\bibliographystyle{plain}
\bibliography{main}

\end{document}